\documentclass[]{amsart}
\usepackage[utf8]{inputenc}
\usepackage[mathscr]{eucal} 
\usepackage{stmaryrd} 
\usepackage{amsmath,amsthm,amsfonts,amssymb,amscd} 
\usepackage[dvipsnames]{xcolor} 
\usepackage{xspace} 
\usepackage{fancyhdr} 
\usepackage{graphicx} 
\usepackage{listings} 
\usepackage[]{hyperref} 
\usepackage{enumitem} 
\usepackage{relsize} 
\usepackage{tikz-cd} 
\usepackage{mdwlist} 
\usepackage{multicol} 
\usepackage{float} 
\usepackage{adjustbox} 
\usepackage{tikz} 
\usepackage[noabbrev,nameinlink]{cleveref} 
\Crefformat{section}{#2\S#1#3}
\usepackage{setspace} 
\usepackage{bbm} 
\usepackage{mathtools} 
\usepackage{epigraph} 
\usetikzlibrary{matrix, calc, arrows}

\hypersetup{
    colorlinks=true, 
    linktoc=all,     
    linkcolor=Brown,citecolor=Brown,urlcolor=MidnightBlue,  
}
\setlist[enumerate]{label=(\alph*)} 
\usetikzlibrary{graphs,decorations.pathmorphing,decorations.markings}
\tikzcdset{scale cd/.style={every label/.append style={scale=#1},
        cells={nodes={scale=#1}}}}

\newcommand{\Hom}{\operatorname{Hom}}
\newcommand{\Aut}{\operatorname{Aut}}

\newcommand{\tr}{\operatorname{tr}}

\newcommand{\inv}{^{-1}}

\newcommand{\End}{\operatorname{End}}
\newcommand{\op}{^{\operatorname{op}}}
\newcommand{\Res}{\operatorname{Res}}
\newcommand{\Ind}{\operatorname{Ind}}
\newcommand{\Inf}{\operatorname{Inf}}

\newcommand{\id}{\operatorname{id}}

\newcommand{\im}{\operatorname{im}}

\newcommand{\Syl}{\operatorname{Syl}}

\newcommand{\Spc}{\operatorname{Spc}}
\newcommand{\Pic}{\operatorname{Pic}}

\newcommand{\open}{\operatorname{open}}
\newcommand{\cone}{\operatorname{cone}}

\newcommand{\on}[1]{\operatorname{#1}}

\newcommand{\calB}{\mathcal{B}}
\newcommand{\calC}{\mathcal{C}}

\newcommand{\calH}{\mathcal{H}}

\newcommand{\calM}{\mathcal{M}}
\newcommand{\calN}{\mathcal{N}}
\newcommand{\calO}{\mathcal{O}}

\newcommand{\catB}{\mathscr{B}}

\newcommand{\catD}{\mathscr{D}}

\newcommand{\catF}{\mathscr{F}}

\newcommand{\catH}{\mathscr{H}}
\newcommand{\catI}{\mathscr{I}}

\newcommand{\catK}{\mathscr{K}}
\newcommand{\catL}{\mathscr{L}}
\newcommand{\catM}{\mathscr{M}}

\newcommand{\catP}{\mathscr{P}}

\newcommand{\catS}{\mathscr{S}}
\newcommand{\catT}{\mathscr{T}}

\newcommand{\frakp}{\mathfrak{p}}
\newcommand{\frakq}{\mathfrak{q}}

\newcommand{\bbC}{\mathbb{C}}

\newcommand{\bbF}{\mathbb{F}}

\newcommand{\bbN}{\mathbb{N}}

\newcommand{\bbR}{\mathbb{R}}

\newcommand{\bbZ}{\mathbb{Z}}
\newcommand{\bbone}{\mathbbm{1}}

\newcommand{\Weyl}[2]{{#1}/\!\!/{#2}}
\newcommand{\ch}{\on{char}}
\newcommand{\comp}{\on{comp}}

\makeatletter
\newcommand*{\doublerightarrow}[2]{\mathrel{
  \settowidth{\@tempdima}{$\scriptstyle#1$}
  \settowidth{\@tempdimb}{$\scriptstyle#2$}
  \ifdim\@tempdimb>\@tempdima \@tempdima=\@tempdimb\fi
  \mathop{\vcenter{
    \offinterlineskip\ialign{\hbox to\dimexpr\@tempdima+1em{##}\cr
    \rightarrowfill\cr\noalign{\kern.5ex}
    \rightarrowfill\cr}}}\limits^{\!#1}_{\!#2}}}
\newcommand*{\triplerightarrow}[1]{\mathrel{
  \settowidth{\@tempdima}{$\scriptstyle#1$}
  \mathop{\vcenter{
    \offinterlineskip\ialign{\hbox to\dimexpr\@tempdima+1em{##}\cr
    \rightarrowfill\cr\noalign{\kern.5ex}
    \rightarrowfill\cr\noalign{\kern.5ex}
    \rightarrowfill\cr}}}\limits^{\!#1}}}
\makeatother

\newcommand{\supp}{\on{supp}}
\newcommand{\sbull}{{\scriptscriptstyle\bullet}}

\newcommand{\Spec}{\on{Spec}}

\newcommand{\Proj}{\on{Proj}}

\newtheorem{theorem}{Theorem}[section]
\newtheorem{lemma}[theorem]{Lemma}
\newtheorem{proposition}[theorem]{Proposition}
\newtheorem{corollary}[theorem]{Corollary}

\newtheorem*{theorem*}{Theorem}

\theoremstyle{definition}
\newtheorem{construction}[theorem]{Construction}

\theoremstyle{remark}
\newtheorem{remark}[theorem]{Remark}
\newtheorem{example}[theorem]{Example}
\newtheorem{observation}[theorem]{Observation}
\newtheorem{recollection}[theorem]{Recollection}
\newtheorem{notation}[theorem]{Notation}

\theoremstyle{definition}
\newtheorem{definition}[theorem]{Definition}

\makeatletter
\newcommand{\xhookdoubleheadrightarrow}[2][]{%
  \lhook\joinrel
  \ext@arrow 0359\rightarrowfill@ {#1}{#2}%
  \mathrel{\mspace{-15mu}}\rightarrow
}
\makeatother

\newtheorem{innercustomthm}{Theorem}
\newenvironment{customthm}[1]
  {\renewcommand\theinnercustomthm{#1}\innercustomthm}
  {\endinnercustomthm}

\AddToHook{env/corollary/begin}{\crefalias{theorem}{corollary}}
\AddToHook{env/proposition/begin}{\crefalias{theorem}{proposition}}
\AddToHook{env/lemma/begin}{\crefalias{theorem}{lemma}}
\AddToHook{env/definition/begin}{\crefalias{theorem}{definition}}
\AddToHook{env/example/begin}{\crefalias{theorem}{example}}
\AddToHook{env/remark/begin}{\crefalias{theorem}{remark}}
\AddToHook{env/observation/begin}{\crefalias{theorem}{observation}}
\AddToHook{env/construction/begin}{\crefalias{theorem}{construction}}
\AddToHook{env/recollection/begin}{\crefalias{theorem}{recollection}}
\AddToHook{env/notation/begin}{\crefalias{theorem}{notation}}

\begin{document}
    \title{The classification of integral endotrivial complexes}
    \author{Juan Omar G\'omez}
    \author{Sam K. Miller}

    \makeatletter
\patchcmd{\@setaddresses}{\indent}{\noindent}{}{}
\patchcmd{\@setaddresses}{\indent}{\noindent}{}{}
\patchcmd{\@setaddresses}{\indent}{\noindent}{}{}
\patchcmd{\@setaddresses}{\indent}{\noindent}{}{}
\makeatother

\address{Juan Omar G\'omez, Fakultat f\"ur Mathematik, Universit\"at Bielefeld, D-33501 Bielefeld, Germany}
\email{jgomez@math.uni-bielefeld.de}
\urladdr{https://sites.google.com/cimat.mx/juanomargomez/home}

\address{Sam K. Miller, Department of Mathematics, University of Georgia, Athens GA 30602, United States of America}
\email{sam.miller@uga.edu}
\urladdr{https://www.samkmiller.com/}
    \subjclass[2020]{20C10, 20J05, 18E35, 18G80, 18M05, 18N60} 
    \keywords{Permutation modules, endotrivial, integral coefficients, tensor-triangular geometry, Picard group, descent, forerunner} 
    \begin{abstract}
        We describe the group of endotrivial complexes, i.e., the Picard group, of the derived category of permutation modules for a finite group over a commutative Noetherian ring. As a result, we deduce that not every endotrivial with integer coefficients arises from a homotopy representation, i.e., an invertible genuine equivariant spectrum. Along the way, we establish a descent result with respect to subgroups of prime-power order, show that oriented endotrivial complexes are line bundles, that is, locally trivial with respect to an open cover of the Balmer spectrum, and provide a topological construction of forerunner homomorphisms, answering a question of the second-named author. 
    \end{abstract}

    \maketitle
    
    \epigraph{\textit{\hfill The Mayor of The Merry Error-Makers’\\
    \hfill error\\
    \hfill made\\
    \hfill The Mayor of The Merry Error-Makers\\
    \hfill merry.}}{Greyson Meyer \cite{Mey25}}
    
    \section*{Introduction}

    Let $G$ be a finite group and $k$ a field of prime characteristic $p$. The tensor-triangular geometry of permutation $kG$-modules - that is, modules obtained by $k$-linearizing $G$-sets - and their direct summands, the $p$-permutation modules, has garnered considerable attention \cite{BG25}. This interest stems in part from the fact that they recover the wild modular representation theory of $kG$ \cite{BG23a}, but their significance extends well beyond the domain of representations.
    These objects interact with many areas of equivariant mathematics through the so-called \emph{big derived category of permutation modules} $\catD(G;k)$, whose compact objects form the category $\catK(G;k) \coloneqq \catD(G;k)^c$. Indeed, this category admits several different incarnations, arising in the contexts of Artin motives, cohomological Mackey functors, and modules over equivariant spectra; see \cite{BG23b, F25} for further details. At the same time, $\catK(G;k)$ has a more down to earth description: it is simply the bounded homotopy category of $p$-permutation modules.
    
    Miller (the second-named author) deduced the Picard group of $\catD(G;k)$ \cite{Mil25b}, i.e., the group $\Pic(\catK(G;k))$ of (necessarily compact) invertible objects of $\catK(G;k)$ using purely representation-theoretic techniques. These objects are termed \emph{endotrivial complexes}, following the established nomenclature for the invertible objects in the stable module category of $kG$-modules, endotrivial modules. Endotrivial complexes play a significant role in the tt-geometry of $\catD(G;k)$: Balmer--Gallauer defined a \emph{`permutation twisted cohomology ring'} for elementary abelian $p$-groups using certain endotrivials to deduce the topology of the Balmer spectrum $\Spc(\catK(G;k))$ for any finite group. The second-named author extended this construction to all $p$-groups \cite{Mil25c}, generalizing numerous geometric results of Balmer--Gallauer about endotrivials and the `remixed' twisted cohomology ring. At present however, a complete description of the geometry of $\Spc(\catK(G;k))$ via the remixed twisted cohomology ring alone hinges on the conjectural Noetherianity of this ring. 
    
    Of course, permutation modules for $G$ can be defined over any arbitrary commutative base ring. G\'omez (the first-named author) initiated the study of `integral' permutation modules \cite{Gom25} by developing the theory of fiberwise stratification, then with Dubey, deduced the tt-geometry of $\catK(G;R)$ with $R$ a commutative Noetherian ring \cite{DG25}. These results continue a recent area of inquiry, the challenge of understanding classical tt-geometric situations over arbitrary Noetherian coefficient rings \cite{Lau23, BBIKP25}. 

    In this paper, we extend the classification of endotrivials and deduce the Picard group $\Pic(\catK(G;R))$ for any commutative Noetherian ring $R$. We say that a commutative ring $R$ is \emph{connected} if $\Spec(R)$ is connected, or equivalently, if $R$ is indecomposable.

    \begin{customthm}{A}[\Cref{thm:mainthm}, \Cref{prop:fingen}]
        Let $G$ be a finite group and $R$ a commutative connected Noetherian ring. Then \[\Pic(\catK(G;R)) = \Pic(R) \times \on{CF}_b^{gl}(G;R) \times \catM(G;R),\] where $\on{CF}_b^{gl}(G;R)$ is a tuple of Borel-Smith functions `glued' at the trivial subgroup (see \Cref{def:borel_smith}, \Cref{rmk:refined_h_mk_func}, \Cref{not:glued_borel_smith_all_grps} for definitions) and $\catM(G;R) \leq \Hom(G,R^\times)$ is a subgroup consisting of all group homomorphisms $\varphi\colon G\to R^\times$ for which $R_\varphi$ is a direct summand of a permutation module (see \Cref{thm:D} for the criterion). Moreover, $\Pic(\catK(G;R))$ is finitely generated if and only if $\Pic(R)$ is. 
    \end{customthm}

    It suffices to assume that $R$ is connected, since a decomposition $R = R_1 \times R_2$ affords a categorical decomposition $\catK(G;R) \simeq \catK(G;R_1) \times \catK(G;R_2)$ of tt-categories, hence a decomposition of Picard groups. This result in fact recovers the \emph{Picard spectrum} $\on{pic}(\catD(G;R))$, since the higher homotopy groups are fairly uninteresting. 
    
    When $R = \bbZ$, one has the following alternative characterization. Moreover, we can deduce that not every integral endotrivial comes from a \emph{homotopy representation}, i.e., an invertible object in the stable homotopy category $\on{SH}(G)$ of (necessarily finite) genuine $G$-spectra. 

    \begin{customthm}{B}[\Cref{cor:mainthmforZ}]
        Let $G$ be a finite group. Let $\on{CF}_b(G,\Phi)$ denote the group of Borel-Smith superclass functions valued on prime power subgroups of $G$. Then \[\Pic(\catK(G;\bbZ)) \cong \on{CF}_b(G,\Phi).\] In particular, $\Pic(\catK(G;\bbZ))$ is finitely generated. Moreover, if $G$ is non-nilpotent, the induced map \[\Pic(\on{SH}(G)^c) \to \Pic(\catK(G;\bbZ))\] need not be surjective, that is, not every integral endotrivial complex is induced from a homotopy representation. 
    \end{customthm}

    The analogous question of studying the Picard group over integral coefficients has been considered by Krause and Grodal, see e.g. \cite{Kra25} and \cite[Integral endotrivial modules]{Ober23}. 
    
    Unlike the case of coefficients over a field, arriving to our conclusion requires a significant amount of machinery and techniques arising from representation theory, homotopy theory, and both algebraic and tt-geometry. We now indulge the reader in the details. 
    
    \subsection*{Homotopical considerations} In Section \Cref{sec:infty_cat_stuff}, we show that $\catD(G;R)$, viewed as a symmetric monoidal stable $\infty$-category, is an $\infty$-categorical limit via group restriction. To our knowledge this has not yet been stated in the literature, but is possibly known to the experts. Perhaps miraculously, this limit descends to an isomorphism of Picard groups of homotopy 1-categories, thus allowing us to complete the classification in Section \Cref{sec:final}. 

    \begin{customthm}{C}[\Cref{prop-reduction-pgroups}, \Cref{coro-lim-picard}, \Cref{cor:limit_of_pic_groups}]
        Let $G$ be a finite group and let $R$ be a commutative ring. Let $\mathcal{O}_\Phi(G)\subseteq \mathcal{O}(G)$
        denote the full subcategory of the orbit category $\calO(G)$ consisting of those orbits whose isotropy groups have prime power order. We have an equivalence of symmetric monoidal stable $\infty$-categories
        \[
            \catD(G;R)\xrightarrow{\simeq}
            \lim_{G/H\in \mathcal{O}_{\Phi}(G)^\mathrm{op}} \catD(H;R),
        \]
        hence an equivalence of Picard spectra 
        \[
        \mathrm{pic}(\catD(G;R))
        \xrightarrow{\simeq}
        \lim_{G/H\in \mathcal{O}_{\Phi}(G)^\mathrm{op}}
        \mathrm{pic}(\catD(H;R)),
        \] which descends to an isomorphism of abelian groups 
        \[
            \mathrm{Pic}(\catK(G;R))
            \xrightarrow{\simeq}
            \lim_{G/H\in \mathcal{O}_{\Phi}(G)^\mathrm{op}}
            \mathrm{Pic}(\catK(H;R)).
        \]
    \end{customthm}

    To make any headway towards classifying endotrivials over arbitrary rings, we require examples. One source of endotrivials alluded to in \cite{Mil25b} and flushed out in detail in \cite{Mil25c} is \emph{representation spheres}, 1-point compactifications of real representations. In Section \Cref{sec:rep_spheres}, we continue this story over arbitrary coefficients. To each representation sphere one obtains an endotrivial chain complex of permutation modules via Bredon homology.

    \begin{customthm}{D}[\Cref{prop:hmkisdimfunc}, \Cref{prop:def_of_B}, \Cref{thm:picard_grp_p_grps}]
        Let $R$ be a commutative ring, let $G$ be a finite group, and let $S(V)$ be a representation sphere for $G$. Then the cellular chain complex \[\tilde{C}(S(V)^1; R) \in \catK(G;R)\] is endotrivial. This induces a group homomorphism \[\catB\colon \on{RO}(G) \to \Pic(\catK(G;R)),\] which is surjective when the following conditions hold: $G$ is a $p$-group, $\Pic(R)$ is trivial, and $\Spec(R/p)$ is connected.
    \end{customthm}

    We note that representation spheres provide only one model for this construction; there are also algebraic or `base-free' approaches. For example, as mentioned before, endotrivials can likewise be obtained from Kaledin's \emph{derived Mackey functors} \cite{K11} or from homotopy representations, invertible $G$-spectra. However even in these cases, a canonical source of invertible objects is representation spheres, as one has a homomorphism $\on{RO}(G) \to \Pic(\on{SH}(G)^c)$, which $\catB$ factors through. The question of classifying homotopy representations is classical \cite{tDP82, Bau89}, recently Krause linked them to integral endotrivial modules \cite{Kra25}. 
    
    For representation spheres, the associated chain complexes of free modules over the orbit category $\calO(G)$ have been studied in detail by Yal\c{c}in et al. \cite{HY14, Y17, GY21}. This is closely related to the extensively studied area of actions of groups on spheres; see \cite{HY26} for a recent survey.
    
    \subsection*{Extending modular fixed points and h-marks} One of the essential tools for studying permutation modules is Balmer--Gallauer's \emph{modular fixed points} (classically known as Brauer quotients \cite{Bro79}). For every $p$-subgroup $H$ of $G$, there is a tt-functor $\Psi^H\colon \catK(G;k) \to \catK(\Weyl{G}{H};k)$ sending $k(X)$ to $k(X^H)$, where $\Weyl{G}{H}\coloneqq N_G(H)/H$ denotes the \emph{Weyl subgroup} at $H$. Miller used these functors to define the \emph{h-marks} of an endotrivial complex $x$ over $k$: a Borel-Smith superclass function $h_x$ that completely determines the isomorphism class of an endotrivial over a $p$-group, and up to a degree 1 representation for an arbitrary finite group. This in turn defines a \emph{h-mark homomorphism} out of the Picard group with small kernel, and with image the group of Borel-Smith functions. 

    The situation over a general ring $R$ is more subtle. Every prime $\frakp \in \Spec(R)$ induces a base-change tt-functor $\catK(G;R) \to \catK(G;k(\frakp))$ along the fiber $R\to k(\frakp)$, and hence gives rise to a `fiberwise' h-mark function at each prime. As in the field case, an endotrivial whose fiberwise h-marks vanish identically must be isomorphic to an invertible $RG$-module concentrated in a single degree. This is \Cref{thm:ker_h_marks}. Proving this fact takes more work and requires a reduction to the case in which $R$ is reduced. 

    \begin{customthm}{E}[\Cref{thm:endotrivials_lift_from_reduction}]
        Let $R$ be a commutative Noetherian ring, let $\overline{R} = R/N$ where $N \subseteq \calN$ is contained in the nilradical of $R$, and let $G$ be a finite group. Then the base change $\overline{R} \otimes_R -\colon \catK(G;R) \to \catK(G;\overline{R})$ induces a group isomorphism \[\Pic(\catK(G;R)) \cong \Pic(\catK(G;\overline{R})).\] 
    \end{customthm}
    
    The upshot is that determining $\Pic(\catK(G;R))$ amounts to:

    \begin{enumerate}
        \item classifying the invertible $\natural$-permutation $RG$-modules;
        \item deducing whether any relationships exist between fiberwise h-marks.
    \end{enumerate}
    
    The former is easier to answer and is resolved in Section \Cref{sec:kernel_of_hmks} with a bit of representation theory. A new consideration that arises is that not every $R$-free rank 1 $RG$-module is $\natural$-permutation, to which we deduce a nice criterion.

    \begin{customthm}{F}[\Cref{thm:rk1ppermsforpgrp}, \Cref{cor:detect_rk1_pperms_for_all_grps}] \label{thm:D}
        Let $\varphi \in \Hom(G, R^\times)$. If $G$ is a $p$-group, the $R$-free $RG$-module $R_\varphi$ is a direct summand of a permutation module if and only if $R = I^\varphi_G + pR$, where $I^\varphi_G$ denotes the annihilator ideal of $\{1 - \varphi(g) \mid g \in G\}\subseteq R$. If $G$ is a finite group, $R_\varphi$ is a direct summand of a permutation module if and only if $\Res^G_S (R_\varphi)$ is for every Sylow subgroup $S$ of $G$.
    \end{customthm}

    To begin deducing the latter, we introduce in Section \Cref{sec:fixedpts_hmks} a form of modular fixed points over rings, $\Psi^H_p \colon \catK(G;R) \to \catK(\Weyl{G}{H}; R/p)$ for $H$ a $p$-subgroup of $G$, defined over any prime $p$ not invertible in $R$. The construction is analogous to Balmer--Gallauer's, and leverages the fact that many of the good properties of modular fixed points hold so long as the base ring has characteristic $p$. When $R/p$ has connected spectrum, this defines a `generalized' h-mark function, which we use to verify that if two primes $\frakp_1, \frakp_2 \in \Spec(R)$ live in the same connected component in the closed subspace $\Spec(R/p)\subseteq \Spec(R)$ (which need not be connected even if $R$ is), their associated fiberwise h-mark functions coincide. 

    \subsection*{Localization and gluing} In fact, each connected component of $R/p$, as $p$ ranges through the primes, essentially provides another `$p$-local degree of freedom' for $\Pic(\catK(G;R))$. To each connected component of $R/p$, as $p$ runs through the primes not invertible in $R$, pick a Borel-Smith function. Then as long as all the Borel-Smith functions share the same value at the trivial subgroup of $G$, one can define an endotrivial with these associated fiberwise h-marks. To show this, we relate the tt-geometry of $\catK(G;R)$ to the geometry of $\Spec(R)$, extending Thomason's classical results \cite{Tho97}. This relies on Balmer's comparison map \cite{Bal10}, 
    \[
    \comp\colon \Spc(\catK(G;R))\to \Spec(R)
    \]
    noting the endomorphism ring of the tensor unit of $\catK(G;R)$ is simply $R$.

    \begin{customthm}{G}[\Cref{prop:localization_identification}]\label{thm:F}
        Let $G$ be a finite group and $R$ a commutative Noetherian ring. Let $f \in R$ be non-nilpotent and consider the associated localization $ R_f$. Let $D(f)$ denote the principal open $\{\frakp \mid f\not\in\frakp\}\subseteq \Spec(R)$, and set $\open(f) = \comp\inv(D(f))$. Then we have an equivalence of tt-categories \[\catK(G;R_f) \cong \catK(G;R)(\open(f)).\]
    \end{customthm}

    From here, in Section \Cref{sec:gluing} we construct for $p$-groups an open cover of $\Spc(\catK(G;R))$ associated to certain well-chosen elements $f$ of $R$ such that each open contains at most one connected component of $R/p$. This cover allows us to employ gluing techniques introduced by Balmer--Favi \cite{BF07}; on each open, we choose an endotrivial, and so long as all the `global' homology coincides, the local endotrivials can be glued without issue.  

    \begin{customthm}{H}[\Cref{thm:gluing_etrivs_p_grps}]
        Let $G$ be a $p$-group and $R$ a commutative connected Noetherian ring. The image of the reduced fiberwise h-mark homomorphism, \[\Pic(\catK(G;R)) \to \prod_{\frakp \in [\pi_0(\Spec(R/p))]} \on{CF}_b(G,\ch(k(\frakp)))\] (see \Cref{rmk:refined_h_mk_func}) is precisely $\on{CF}_b^{gl}(G)$, the tuples of Borel-Smith functions which agree on the trivial subgroup of $G$.
    \end{customthm}
    
    \subsection*{tt-line bundles} 

    Miller deduced in \cite{Mil25c} that when $G$ is a $p$-group, every endotrivial $x \in \Pic(\catK(G;k))$ is a \emph{tt-line bundle}: there exists an open cover of $\Spc(\catK(G;k))$ under which every $x$ is locally isomorphic to a shift of the tensor unit. This result underlies how permutation twisted cohomology works (when Noetherian); locally, it identifies with the usual cohomology ring of a tt-category, and under this open cover, $\catK(G;k)$ becomes a Dirac scheme in the sense of \cite{HP23}. 

    We show in Section \Cref{sec:line_bundles} an analogous result here for $p$-groups using similar but improved techniques. For endotrivials $x\in \Pic(\catK(G;R))$ arising from representation spheres, if the sphere is \emph{oriented} (see \Cref{def:oriented}), one may construct a \emph{forerunner map}
    \[
    \iota_x^H\colon R \to x[-h_x(H)]
    \]
    for which $\Psi^H_p(\iota^H_x)$ is a quasi-isomorphism; see \Cref{cons:forerunners}. In fact, our construction answers a question posed by the second-named author, \cite[Remark 4.4(b)]{Mil25c}, on whether a `natural' topological construction exists.
    
    Inverting certain open sets of the spectrum associated to these forerunners verifies that such endotrivials are line bundles. Additionally, line bundles of $R$ are endotrivial as well, and an application of \Cref{thm:F} concludes they are line bundles in the tt-geometric sense. 

    \begin{customthm}{I}[\Cref{thm:linebundle}]
        Let $G$ be a $p$-group and $R$ a commutative Noetherian ring. Then every oriented endotrivial $x \in \Pic(\catK(G;R))$ is a line bundle. If $\Pic(R)$ is finitely generated, then moreover, there exists an open cover of $\Spc(\catK(G;R))$ that verifies every such endotrivial is a line bundle. 
    \end{customthm}

    This paves the way for constructing a twisted cohomology theory over arbitrary Noetherian bases, but we postpone that endeavor for another day. Finally, we list some open questions at the very end of the paper. 

    We hope this paper should be comprehendible by both a representation theorist less comfortable with tensor-triangular geometry and a tensor-triangular geometer less comfortable with representation theory. As such, we include a section recalling the basics of tensor-triangular geometry (Section \Cref{sec:ttgprelims}), a section recalling the basics of permutation modules both over fields and arbitrary rings (Section \Cref{sec-perm}), and a section recalling the basics of invertible objects and endotrivial complexes (Section \Cref{sec:etriv_review}). Additionally, we assume no familiarity with $\infty$-categories, and as such, quarantine their use to Section \Cref{sec:infty_cat_stuff}.

    \subsection*{Conventions} 
    Throughout this paper, $G$ denotes a finite group, occasionally assumed to be of prime-power order, and $R$ denotes a commutative ring, occasionally assumed to be Noetherian or connected. We also adopt the following conventions:
    \begin{enumerate}
        \item For an element $r \in R$, we write $R_r$ as shorthand for $R[1/r]$, and $R/r$ for the quotient ring $R/rR$.
        \item Given a homomorphism $f\colon R \to S$ of commutative rings, we write $f^\ast$ for the functor induced by base change in the corresponding context.
        \item  We often abbreviate `tensor-triangular' or `tensor-triangulated' by `tt' for short.  
    \end{enumerate}

    \subsubsection*{Acknowledgments} We thank Yorick Fuhrmann, Henning Krause, and Robert Boltje for their comments on previous iterations of this paper. We also thank Isaac Moselle for spotting a crucial error in an argument in Section 9 and providing an alternative argument. 
    This project got off the ground at the Mathematisches Forschungsinstitut Oberwolfach, and we thank the MFO for their hospitality. G\'omez was supported by the Deutsche Forschungsgemeinschaft (Project-ID 491392403 – TRR 358). Miller was partially supported by an AMS-Simons travel grant and a US Junior Oberwolfach Fellowship (NSF grant DMS-2230648). Finally, we thank Greyson Meyer for permitting us to share his poem.

    \section{Tensor-triangular geometric preliminaries} \label{sec:ttgprelims}

    To begin, we recall some basic tensor-triangular geometry. Primary references for these are \cite{Bal05, Bal10b, Ste18}.
   
    \begin{recollection}
        A \emph{tensor-triangulated category} is a symmetric monoidal triangulated category with an exact tensor product in each variable separately. A \emph{big tt-category} $\catT$ is a rigidly-compactly generated tt-category; that is, it is compactly generated as a triangulated category and its \emph{compact objects} coincide with its \emph{rigid} (i.e., \emph{dualizable}) objects. Note this is a ``unital algebraic stable homotopy category'' in the sense of \cite{HPS97}. We denote by $\catT^c$ the compact/dualizable part of $\catT$, which is necessarily essentially small. A \emph{tt-functor} is a triangulated functor between tt-categories which is strong monoidal. A tt-functor between big tt-categories is \emph{geometric} if it preserves coproducts. Note that in particular a geometric functor preserves compact objects.

        A full triangulated subcategory $\catI$ of a tt-category $\catK$ is \emph{thick} if it is closed under direct summands, a \emph{thick $\otimes$-ideal} if in addition it has ideal closure ($\catI \otimes\catK \subseteq \catI$), and \emph{prime} if furthermore, it is proper and satisfies the following primality conditions: $x \otimes y \in \catI$ implies $x \in \catI$ or $y \in \catI$. If $\catK$ is essentially small (e.g. $\catK = \catT^c$ for a big tt-category $\catT$), the \emph{Balmer spectrum} $\Spc(\catK)$ \cite{Bal05} is the set of prime thick $\otimes$-ideals, 
        \[\Spc(\catK) \coloneqq \{\catP \subseteq \catK \mid \catP \text{ is a prime thick $\otimes$-ideal}\}\] 
        endowed with the topology whose basis of \emph{closed} sets is given by supports of objects $x \in \catK$, 
        \[\supp(x) \coloneqq  \{\catP \in \Spc(\catK)\mid x \not\in \catP\}.\] 
        In fact, $\Spc(\catK)$ is a spectral space; that is, it is homeomorphic to the Zariski spectrum of a commutative ring. 
        The Balmer spectrum is contravariantly functorial: given a tt-functor $F\colon \catK\to\catL$, one obtains a continuous map on spectra 
        \[\Spc(F)\colon \Spc(\catL) \to \Spc(\catK)\]
        given by $\Spc(F)(\catP)\coloneqq F^{-1}(\catP)$.
    \end{recollection}
    
    \begin{recollection}
        We next recall the \emph{presheaf of triangulated categories} following \cite{BF07}. This is also due to Balmer (see \cite{Bal02}), and predates the construction of the spectrum \cite{Bal05}. With $\catK$ as before, let $U\subseteq \Spc(\catK)$ be a quasi-compact open subset and let $Z \coloneqq \Spc(\catK) \setminus U$. We define the thick subcategory 
        \[\catK_{Z}\coloneqq  \{x \in \catK \mid \supp(x) \subseteq Z\},\]
        i.e., the objects which should vanish on $U$. The category 
        \[\catK(U)\coloneqq (\catK/\catK_Z)^\natural\] 
        is the idempotent completion of the Verdier quotient $\catK/\catK_Z$. The composition $\catK \twoheadrightarrow\catK/\catK_Z \hookrightarrow \catK(U)$ yields a restriction functor $\catK\to\catK(U)$, and the induced map on Balmer spectra is nothing more than the open immersion $U \hookrightarrow \Spc(\catK)$. This is a genuine presheaf: the restrictions are compatible with intersections, see \cite{BF07} for details.
    \end{recollection}

    \begin{example}
        An important example due to Thomason \cite{Tho97} is as follows. For $\catK = \on{D}_{\text{perf}}(X)$ for $X$ a quasi-compact quasi-separated scheme, one has a homeomorphism $\Spc(\catK) \cong X$. Then for any quasi-compact open $U \subseteq X$, one has $\catK(U) \cong \on{D}_{\text{perf}}(U)$ via the restriction functor $\catK\to \catK(U)$ is given by the restriction functor to the open $U$. 
    \end{example}

    \begin{recollection}
        Finally, we recall Balmer's \emph{comparison map} \cite{Bal10} in the generality that we require. With $\catK$ as before, the `graded cohomology ring' 
        \[
        \End^\sbull_\catK(\bbone)\coloneqq \bigoplus_{n\in \bbZ}\End^n_\catK(\bbone) = \bigoplus_{n\in \bbZ}\Hom_\catK(\bbone, \bbone[n])
        \]
        of $\catK$ is the $\bbZ$-graded endomorphism ring of the monoidal unit $\bbone\in \catK$.
        Note that occasionally one only considers the $\bbN$-graded variant.  Then, we have a continuous map, the \emph{(graded) comparison map}:
        \begin{align*}
            \comp\colon \Spc(\catK) &\to \Spec^h(\End_\catK^\sbull(\bbone)) \\
            \catP &\mapsto \langle f \in \End_\catK^\sbull(\bbone) \mid \cone(f) \not\in \catP\rangle.
        \end{align*}
    \end{recollection}

    The comparison map is a homeomorphism in numerous classical examples.

    \begin{example}
        Let $G$ be a finite group (scheme) and $k$ a field of characteristic $p$. Set $\on{D}_b(kG) \coloneqq \on{D}_b(\mathsf{mod}(kG))$, then the comparison map induces a homeomorphism \[\Spc(\on{D}_b(kG)) \cong \Spec^h(\on{H}^\sbull(G;k))\] which, under the localization $\mathsf{stmod}(kG) \cong \on{D}_b(kG)/\on{D}_b(\mathsf{proj}(kG))$ (see \cite{Ric89}) reduces to a homeomorphism \[\Spc(\mathsf{stmod}(kG)) \cong \Proj(\on{H}^\sbull(G;k)).\] This result is classically due to Benson--Carlson--Rickard \cite{BCR97} for groups and Friedlander--Pevtsova \cite{FP07} for group schemes, however the presentation here is due to Balmer's framework \cite{Bal10}. 
    \end{example}

    \begin{example}
        For a commutative ring $R$, the comparison map induces a homeomorphism 
        \[\Spc(\on{D}_{\mathsf{perf}}(R)) \cong \Spec(R),\] 
        however, while one also has that $\Spc(\on{D}_{\mathsf{perf}}(X))$ is homeomorphic to the underlying topological space $X$ for any quasi-compact quasi-separated scheme \cite{Tho97}, the comparison map is not a homeomorphism in general; indeed, this already fails for the projective line. See \cite[Section 8]{Bal10} for discussion. 
    \end{example}

    \section{Permutation modules}\label{sec-perm}

    Throughout, $R$ is always assumed to be a commutative ring, which we may often assume is Noetherian. We next review tensor-triangular and representation-theoretic preliminaries about permutation modules, both over general rings and over fields of positive characteristic. 
    
    \begin{recollection}
        Let $G$ be a finite group. A \emph{permutation $RG$-module} is a $RG$-module admitting a $G$-stable $R$-basis. In fact, any permutation $RG$-module is isomorphic to a coproduct of \emph{transitive} permutation modules, i.e., those of the form $R(G/H)$, with $H$ a subgroup of $G$. Observe that $R(G/H) \cong \Ind^G_H (R)$.
        
        We refer to direct summands of permutation modules as \emph{$\natural$-permutation modules.} Any summand is necessarily still $R$-projective. When $R$ is either a field of characteristic $p$ or a complete discrete valuation ring with characteristic $p$ residue field, then being a $\natural$-permutation (or in this case, \emph{$p$-permutation}) module is equivalent to being permutation upon restriction to a Sylow $p$-subgroup of $G$. In those cases, if $G$ is a $p$-group, then the indecomposable ($p$-)permutation modules are precisely those of the form $R(G/H)$, and direct summands of permutation modules are again permutation. We refer the reader to \cite{Li18a, La23} for a representation-theoretic overview of $\natural$-permutation modules over fields and to \cite{BG23b, BG25} for the tensor-triangular-geometric perspective. 
        
        We denote the additive category of finitely generated permutation $RG$-modules by $\mathsf{perm}(G;R)$. Recall that the tensor product over the ground ring $R$ equipped with the diagonal $G$-action gives us a symmetric monoidal structure on the additive category of finitely generated modules, where the monoidal unit corresponds to the trivial module $R$. By the double coset formula, one  deduces that this monoidal structure descends to $\mathsf{perm}(G;R)$. Now, set 
        \[
        \catK(G;R) \coloneqq \on{K}_b(\mathsf{perm}(G;R)^\natural) = \on{K}_b(\mathsf{perm}(G;R))^\natural,
        \]
        where the superscript $\natural$ denotes idempotent completion. Note that $\mathsf{perm}(G;R)^\natural$ is the additive monoidal category of finitely generated $\natural$-permutation $RG$-modules, and each module is an \emph{$R$-lattice}, i.e., finite projective as $R$-module. We denote the shift functor by $(-)[1]$ as opposed to $\Sigma$.
        Then $\catK(G;R)$ is an essentially small rigid tensor-triangulated category with the monoidal structure induced by the one on $\mathsf{perm}(G;R)$. In particular, the internal dual functor is given by the internal hom functor, $x^\ast  = \Hom_R(x, R[0])$. 

        The \emph{(big) derived category of permutation modules} $\catD(G;R)$, is defined as the smallest \textit{localizing} subcategory of the homotopy category  $\on{K}(\mathsf{Mod}(RG))$ containing  $\catK(G;R)$; that is, the smallest triangulated subcategory which closed under coproducts. In fact, $\catD(G;R)$ is a big tt-category whose compact part corresponds to $\catK(G;R)$; see \cite[Section 3]{BG23b} for details. For this paper, $\catD(G;R)$ will not play as much of a role until the end. 
    \end{recollection} 

    \begin{recollection}
        Given a group homomorphism $\alpha\colon G' \to G$, restriction along $\alpha$ induces a geometric functor $\Res_\alpha \colon \catD(G;R) \to \catD(G';R)$ (in particular, it preserves compacts). In particular, if $\alpha$ is injective, we set $\Res^G_{G'} \coloneqq \Res_\alpha$ and if $\alpha$ is surjective, we set $\Inf^G_{G'} \coloneqq \Res_\alpha$, and call the resulting operation \emph{inflation}. 
        Finally, given any subgroup $H \leq G$, (co)induction induces a triangulated (but non-monoidal) functor $\Ind^G_H \colon \catD(H;R) \to \catD(G;R)$. One has $\Ind^G_H R(H/K) \cong R(G/K)$ for any subgroups $K \leq H \leq G$. The usual (co)induction/restriction adjunction statements hold. 
    \end{recollection}

    \begin{remark}
        Over fields, every indecomposable $p$-permutation module is a direct summand of a transitive permutation module, due to the Krull-Schmidt theorem holding. However, over arbitrary rings, this no longer need be the case for $\natural$-permutation modules. For instance, the \emph{Swan modules} are projective $\bbZ G$-modules which may arise as summands of free $\bbZ G$-modules of rank higher than 1, and which can become free after taking successive direct sums. Recently, Hofmann--Nicholson \cite{HN25} deduced examples of non-free, stably free Swan modules, answering an old question of Dyer. We refer the reader to the article for more details. 
    \end{remark}

    \begin{recollection}\label{rec:extension-scalars-perm}
        Let $f\colon R\to S$ be a ring map. It is straightforward to verify that extension of scalars along $f$ preserves permutation modules and hence it induces a  geometric functor 
        \[
        f^\ast\coloneqq S\otimes_R-\colon \catD(G;R)\to \catD(G;S).
        \]
        As any geometric functor, one obtains a triple of adjunctions $(f^\ast,f_\ast,f^!)$, where $f_\ast$ is induced by restriction of scalars. Since $f^\ast$ preserves compacts, we obtain a tt-functor 
        \[
         f^\ast\colon \catK(G;R)\to \catK(G;S).
        \]
    \end{recollection}

    \begin{proposition}
        Let $R$ be a commutative ring. Assume that  $R = R_1 \times R_2$ for nonzero commutative rings $R_1, R_2$. Let $\on{pr}_i\colon R\to R_i$ denote the projection map, for $i=1,2$. Then base change along $\mathrm{pr}_i$ induces an equivalence
        \[
       \mathrm{pr}_1^\ast\times \mathrm{pr}_2^\ast\colon \catK(G;R) \xrightarrow[]{\simeq} \catK(G;R_1) \times \catK(G;R_2)
        \]
        as tensor-triangulated categories. Here the tensor-triangulated structure on the right hand side is given pointwise.
    \end{proposition}

    \begin{proof}
        Since  $\mathrm{pr}_i^\ast$ are tt-functors, it is enough to verify that they induce an equivalence of categories. In fact, it enough to show that those functors induce an equivalence 
        \[
        \mathrm{pr}_1^\ast\times \mathrm{pr}_2^\ast\colon\mathsf{perm}(G;R)\xrightarrow[]{\simeq} \mathsf{perm}(G;R_1)\times \mathsf{perm}(G;R_2).
        \]
        We claim that the functor is essentially surjective. Let $(x_1,x_2)\in \mathsf{perm}(G;R_1)\times \mathsf{perm}(G;R_2)$. Since there is a ring map $\iota_i\colon R_i\to R$ such that $\mathrm{pr}_i\iota_\mathrm{i}=\mathrm{id}$, we deduce that $ \iota_1^\ast x_1\oplus \iota_2^\ast x_2\in \mathsf{perm}(G;R)$ satisfies that $(\mathrm{pr}_1^\ast\times \mathrm{pr}_i^\ast)(\iota_1^\ast x_1\oplus \iota_2^\ast x_2)\simeq (x_1,x_2)$. 

        Now, we claim that $\mathrm{pr}_1^\ast\times \mathrm{pr}_2^\ast$ is fully faithful. This is a consequence of the following observation. Note that for any subgroup $H$ of $G$, the permutation module $R(G/H)$ satisfies that $(\mathrm{pr}_1^\ast\times \mathrm{pr}_i^\ast)(R(G/H))\simeq (R_1(G/H), R_2(G/H))$. Indeed, note that  the induction functor is compatible with extension of scalars. From here it is straightforward to deduce the result. 
    \end{proof}

    \begin{remark}
        In view of the previous proposition, we may assume throughout that $R$ is indecomposable, or equivalently, $\Spec(R)$ is connected. We say \emph{$R$ is connected} for shorthand (as opposed to $R$ being connected as a topological ring).

        Note that if $R$ is connected, $\ch(R)$ is either 0 or $p^r$ for some prime $p$. Indeed, suppose otherwise, say $\ch(R) = n$ for $n$ having at least 2 prime divisors. Then we have an injective, unital ring homomorphism $\bbZ/n \hookrightarrow R$. Since $\bbZ/n$ decomposes, its unit can be written as the sum of two orthogonal idempotents. Therefore the unit of $R$ also decomposes in this way, affording a decomposition of $R$. 

        If $R$ is connected with $\ch(R) = p^r$, then the surjection $R \to R/p$ induces the identity map on Zariski spectra. Indeed, every prime ideal of $R$ contains $p$. Let $\frakp \in \Spec(R)$, then we have $p \cdot p^{r-1} \in \frakp$, and an inductive argument shows that $p \in \frakp$. 
    \end{remark}

    \begin{definition}
        Given a prime $\frakp \in \Spec(R)$, let $k(\frakp)$ denote its \emph{residue field} (also called its \emph{fiber}), i.e., the field of fractions of $R/\frakp$. The homomorphism $\lambda_\frakp\colon R \to k(\frakp)$ induces a geometric base change functor \[\lambda_\frakp^\ast\colon \catD(G;R) \to \catD(G;k(\frakp)), \, x \mapsto k(\frakp) \otimes_R x.\]  In particular, it restricts to the compact part $\lambda_\frakp^\ast\colon \catK(G;R) \to \catK(G;k(\frakp))$. See \Cref{rec:extension-scalars-perm}.
    \end{definition}

    \begin{theorem}{\cite[Theorem 4.8]{Gom25}}
        Let $R$ be a commutative Noetherian ring. The family of functors \[\{\lambda_\frakp^\ast\colon \catD(G; R) \to \catD(G;k(\frakp)) \}_{\frakp \in \Spec(R)}\] is jointly conservative.
    \end{theorem}

    Therefore, we we have a description of the points of $\Spc(\catK(G;R))$ as follows; see \cite[Theorem 6.18]{Gom25} (cf \cite[Proposition 5.18]{DG25}).

    \begin{theorem}
        There is a set bijection induced by the base change along all residue fields of $R$ \[\bigsqcup_{\frakp \in \Spec(R)} \Spc(\catK(G;k(\frakp))) \to \Spc(\catK(G;R)).\] In particular, every prime of $\Spc(\catK(G;R))$ is of the form $(\lambda_\frakp^\ast)\inv(\frakq)$ for some $\frakp \in \Spec(R)$ and some $\frakq \in \Spc(\catK(G;k(\frakp)))$.
    \end{theorem}

    \begin{remark}
        Balmer's \emph{comparison map} induces a continuous map 
        \[\Spc(\catK(G;R)) \to \Spec^h(\End^\bullet_{\catK(G;R)}(\bbone)) \cong \Spec(R).\]
        Note that in this case $R=\End_{\catK(G;R)}(\bbone)=\End_{\catK(G;R)}^\bullet(\bbone)$. 
        Therefore, for each open $U \subseteq \Spec(R)$, we obtain a corresponding open \[\comp\inv(U) = \bigsqcup_{\frakp \in U}\Spc(\catK(G;k(\frakp))) \subseteq \Spc(\catK(G;R)).\] 
        Here we are identifying $\Spc(\catK(G;k(\frakp)))$ with its image under $\Spc(\lambda_\frakp^\ast)$.
    \end{remark}
    
    Given $f \in R$ (morally, $f$ is an element of $\End_{\catK(G;R)}(\bbone)$), we write $R_f$ to denote the ring localization $R[1/f]$. We next verify that ring localizations correspond to genuine localizations on the level of tt-categories. This identification will play a crucial role in the sequel. 
    
    \begin{theorem}\label{prop:localization_identification}
        Let $G$ be a finite group and $R$ a commutative Noetherian ring. Let $f \in R$ and consider the associated localization $ R_f$. Let $D(f)$ denote the principal open $\{\frakp \mid f\not\in\frakp\}\subseteq \Spec(R)$, and set $\open(f) = \comp\inv(D(f))$. Then we have an equivalence of tt-categories \[\catK(G;R_f) \cong \catK(G;R)(\open(f)).\]
    \end{theorem}
    \begin{proof}
        First, if $f$ is nilpotent then the claim is clear, so we may assume otherwise. Now, observe that $\catK(G;R_f) \cong \big(S\inv\catK(G;R)\big)^\natural$, where $S = \{1, f, f^2,\dots\}$ and $S\inv\catK(G;R)$ denotes the central localization by $S$; see \cite[Construction 3.5]{Bal10}. Indeed, by \cite[Proposition 3.7]{Bal10}, we have $S\inv\catK(G;R) \cong \catK(G;R)/\catI$, with $\catI$ the following thick $\otimes$-ideal, 
        \[\catI = \langle \cone(f^i)\mid i \in \bbN\rangle = \{x \in \catK(G;R) \mid \exists i \in \bbN \text{ such that } \id_x \otimes f^i = 0\}.\]  
        Now, consider the base change functor along the map $\lambda_f\colon R\to R_f$, and note that $\lambda_f^\ast(x)=0$ for any  $x\in \catI$. Then we obtain a commutative diagram 
        \[
        \begin{tikzcd}
            \catK(G;R) \ar[dr,"\lambda_f^\ast"] \ar[d] &  \\ 
            \big(S\inv\catK(G;R)\big)^\natural \ar[r,"F"] & \catK(G;R_f).
        \end{tikzcd}
        \]
        Let $R(G/H)$ and $R(G/K)$ be two permutation modules.  We have 
        \begin{align*}
            S\inv\Hom^\bullet_{RG}(R(G/H),R(G/K))&\cong R_f\otimes_R\Hom^\bullet_{RG}(R(G/H),R(G/K))  \\
            &\cong \Hom^\bullet_{R_fG}(R_f \otimes_R R(G/H),R_f \otimes_R R(G/K)), 
        \end{align*}
        where the left hand side is precisely the homomorphisms in $ \big(S\inv\catK(G;R)\big)^\natural$. By a d\'evissage argument we obtain that $F$ is fully faithful. On the other hand, by the commutativity of the previous triangle, we have that the essential image of $F$ contains all permutation modules of the form $R_f(G/H)$, from which we deduce that $F$ is essentially surjective since the latter are a set of thick generators of $\catK(G;R_f)$ which  gives us the claimed equivalence.

        On the other hand, we claim that $\catI$ is equivalently the collection of elements $x \in \catK(G;R)$ such that 
        \[\supp(x) \subseteq \bigsqcup_{f\in\frakp \in \Spec(R)} \Spc(\catK(G;k(\frakp))) = \open(f)^c.\] 
        First, note that the above containment holds for $x$ if and only if for all $\frakp \in \Spec(R)$ with $f\not\in \frakp$, one has $k(\frakp) \otimes_R x = 0$ in $\catK(G;k(\frakp))$. Indeed, this follows by some tt-yoga using the surjectivity of the map 
        \[
        \bigsqcup_{\frakp \in \Spec(R)} \Spc(\catK(G;k(\frakp))) \to \Spc(\catK(G;R)).
        \]
        and the naturality of Balmer's comparison map.

        Now, we claim that the above condition is equivalent to multiplication by $f$ on $x$ being nilpotent. Indeed, it is straightforward that if multiplication by $f$ on $x$ is nilpotent, then for every $\frakp$ with $f\not\in \frakp$, $k(\frakp) \otimes_R x = 0$. Conversely, if multiplication by $f$ on $x$ is non-nilpotent, then there exists a prime $\frakp$ with $f\not\in\frakp$. Then multiplication by $f$ on $k(\frakp) \otimes_R x$ is a nontrivial morphism, thus $k(\frakp) \otimes_R x \neq 0$. We conclude that \[\catI = \left\{ x \in \catK(G;R) \mid \supp(x) \subseteq \open(f)^c \right\} \eqqcolon \catK_{\open(f)^c}.\] Putting everything together, we have the following equivalences: 
        \begin{align*}
            \catK(G;R)(\open(f)) &\coloneqq \big(\catK(G;R)/\catK_{\open(f)^c}\big)^\natural\\
            &= \big(\catK(G;R)/\catI\big)^\natural\\
            &\cong \big(S\inv\catK(G;R)\big)^\natural\\
            &\cong \catK(G;R_f),
        \end{align*}
        and we are done. 
    \end{proof}

    We prove one final lemma regarding base change to reduced rings for permutation modules, which will be useful in the sequel. 
    
    \begin{lemma}\label{lem:reduction_surjective_on_modules}
        Let $R$ be a commutative Noetherian ring, let $N \subseteq \calN$ be an ideal of $R$ contained in the nilradical $\calN$ of $R$, and let $\overline{R} = R/N$. Let $G$ be a finite group. The base-change functor 
        \[
        F\colon \overline{R} \otimes_R -\colon \mathsf{perm}(G;R)^\natural \to \mathsf{perm}(G;\overline{R})^\natural
        \]
        is full and essentially surjective.
    \end{lemma}
    \begin{proof}
        The reduction $F$ sends a permutation module $RX$ to $\overline{R}X$, hence $F$ is essentially surjective on permutation modules, and it is easy to see that $F$ is surjective on morphisms. To show $F$ is essentially surjective on direct summands, it suffices to show $F$ satisfies idempotent lifting. Since $R$ is Noetherian, $N$ is itself a nilpotent ideal. We have an isomorphism \[\End_{\mathsf{perm}(G;\overline{R})^\natural}(\overline{R}X) \cong \End_{\mathsf{perm}(G;R)^\natural}(RX)/\big(N\cdot \End_{\mathsf{perm}(G;R)^\natural}(RX)\big).\] It is well-known (see e.g., \cite{Khu21}) that idempotents lift modulo nilpotent ideals, thus $F$ satisfies idempotent lifting, as desired.
    \end{proof}

    \begin{remark}
        However, the base change $\overline{R} \otimes_R -\colon \catK(G;R) \to \catK(G;\overline{R})$ is not always essentially surjective. For instance, let $G = C_2$, let $R = \bbZ/4\bbZ$, in which case $\overline{R} = \bbZ/2\bbZ$. We have an exact 3-term chain complex of permutation $\overline{R}C_2$-modules \[\overline{R} \to \overline{R}C_2 \to \overline{R},\] and one can verify that such a complex cannot be lifted to a complex of $\natural$-permutation $RC_2$-modules, using the fact (see \Cref{thm:rk1ppermsforpgrp}) that the sign representation $R_{\text{sgn}}$ is not $\natural$-permutation. 
    \end{remark}

    \subsection{Permutation modules over fields} The tt-geometry of $p$-permutation modules for fields of positive characteristic $p$ was developed in detail by Balmer--Gallauer \cite{BG22,BG23a,BG25, BG25b}. We recall a few necessary facts. For this subsection, fix $G$ a finite group and $k$ a field of prime characteristic $p$ dividing the order of $G$. We set $\catK(G) \coloneqq \catK(G;k)$ and $\catD(G) \coloneqq \catD(G;k)$. We set $\catK_{ac}(G)$ to be the prime thick $\otimes$-ideal of $\catK(G)$ consisting of all acyclic complexes, and set $\on{D}_b(kG) \coloneqq \on{D}_b(\mathsf{mod}(kG))$. The story begins with the following remarkable observation: $\on{D}_b(kG)$ is a quotient of $\catK(G)$. 
    
    \begin{theorem}{\cite[Theorem 5.13]{BG23a}}
        Every finitely generated $kG$-module admits a finite $p$-permutation resolution. Moreover, the canonical functor \[\bar\Upsilon\colon\catK(G)/\catK_{ac}(G) \to \on{D}_b(kG)\] is an equivalence. In particular, the left-hand side is already idempotent complete. 
    \end{theorem}

    \begin{remark}
        The localization from above is the compact part of a finite localization  $\Upsilon\colon \catD(G) \twoheadrightarrow \on{KInj}(kG)$, see \cite[Remark 4.21]{BG22}. 
    \end{remark}
    
    We recall \emph{modular fixed points} as defined by Balmer--Gallauer. For finitely generated modules, these are known classically (in representation-theoretic contexts) as \emph{Brauer quotients}, and first due to Brou\'e \cite{Bro79, Bro85}. 

    \begin{proposition}{\cite[Proposition 2.7, Corollary 5.16]{BG25}}
        For every $p$-subgroup $H \leq G$, there exists a geometric functor \[\Psi^H\colon \catD(G) \to\catD(\Weyl{G}{H})\] such that $\Psi^H(kX) \cong k(X^H)$ for every $G$-set $X$. This functor preserves compacts and restricts to a tt-functor $\Psi^H\colon \catK(G)\to\catK(\Weyl{G}{H})$. Moreover, for any normal $p$-subgroup $N \trianglelefteq G$, $\id \cong \Psi^N \circ \Inf^{G}_{G/N}$. 
    \end{proposition}

    The collection of modular fixed points is \emph{conservative}, i.e., detects vanishing of objects. 
    
    \begin{theorem}{\cite[Theorem 2.11]{BG25}}\label{thm-conservativity-psiH}
        The family of functors \[\{\catD(G) \xrightarrow{\Psi^H} \catD(\Weyl{G}{H}) \xrightarrow{\Upsilon} \on{KInj}(k(\Weyl{G}{H}))\}_{H \in \on{Sub}_p(G)/G}\] indexed by conjugacy classes of $p$-subgroups $H\leq G$ is conservative. This restricts to a conservative family of functors \[\{\catK(G) \xrightarrow{\Psi^H} \catK(\Weyl{G}{H})\xrightarrow{\Upsilon} \on{D}_b(kG) \}_{H \in \on{Sub}_p(G)/G}.\]
    \end{theorem}
    
    \begin{notation}
        We write $\check{\Psi}^H \coloneqq \Upsilon \circ \Psi^H\colon \catK(G) \to \on{D}_b(k(\Weyl{G}{H}))$. 
    \end{notation}
    
    The collection of modular fixed points, somewhat remarkably, turn out to completely determine the Balmer spectrum $\Spc(\catK(G))$ as a set. Geometrically, $\catK(G)$ is determined by its $p$-local cohomological data. 
    
    \begin{theorem}{\cite[Theorem 2.10]{BG25}}
        The collection of functors $\check{\Psi}^H$ indexed over conjugacy classes of $p$-subgroups of $G$ induces a bijection of sets \[\Spc(\catK(G)) = \bigsqcup_{H \in \on{Sub}_p(G)/G} \Spc(\on{D}_b(k\Weyl{G}{H})) \cong \bigsqcup_{H \in \on{Sub}_p(G)/G} \Spec^h(\on{H}^\sbull(\Weyl{G}{H};k)).\]
    \end{theorem}

    There is much to say about the topology of $\Spc(\catK(G))$, which is deduced via \emph{permutation twisted cohomology} and an extended comparison map. This was first constructed by Balmer--Gallauer \cite[Part II]{BG25} for elementary abelian $p$-groups, and then `remixed' to $p$-groups by the 2nd author \cite{Mil25c}, although the question of deducing the topology of $\Spc(\catK(G))$ via remixed twisted cohomology for all $p$-groups is still conjectural. However these topics are rather technical and not required for this paper, so we will leave the intrigued reader to pursue these topics independently. 
    
    \section{Endotrivial complexes}\label{sec:etriv_review}

    Next, we cover some preliminaries about endotrivial complexes. 
    
    \begin{recollection}
        An \emph{invertible object} of a symmetric monoidal category $\catT$ is an object $x\in \catT$ for which there exists an object $y \in \catT$ such that $x \otimes y \cong \bbone$. In this case $y$ is the inverse $x\inv$ of $x$. 
        
        If $\catT$ is a big tt-category, then in fact any invertible object $x \in \catT$ is necessarily compact and satisfies $x\inv \cong x^\ast $; we leave the proof of this fact as a nice exercise. The \emph{Picard group} of $\catT$ (equivalently of $\catT^c$), denoted $\Pic(\catT^c)$, is the collection of isomorphism classes of invertible objects, and forms an abelian group under $\otimes$. 
    \end{recollection}

    \begin{definition}
         An \emph{endotrivial complex} is an invertible object of $\catK(G;R)$, that is, an object $x \in \catK(G;R)$ such that $x \otimes x^\ast  \simeq R[0] = \bbone$ as chain complexes of $RG$-modules. 
    \end{definition}

    \begin{lemma}\label{lem:commute_with_homology}
        Let $x \in \catK(G;R)$ be $R$-split; that is, $\Res^G_1(x)$ is isomorphic to its homology. Then for all $i \in \bbZ$ and any ring homomorphisms $f\colon  R\to S$, \[S \otimes_R \on{H}_i(x) \cong \on{H}_i(S \otimes_R x).\]
    \end{lemma}
    \begin{proof}
        Indeed, if one restricts to the subcategory of $\catK(G;R)$ consisting of complexes which are $R$-split, $f^\ast = S \otimes_R -$ is an exact functor, as it preserves $\ker(\Res^G_1)$. The result follows. 
    \end{proof}

    \begin{remark}\label{rem:prohomology-Rsplit}
     Note that a complex $x\in \catK(G;R)$ is $R$-split if and only if its homology is $R$-projective. Indeed, this follows since $x$ is a perfect $R$-complex. 
    \end{remark}

    \begin{lemma}\label{lem:endotrival-Rsplit}
        Let $x\in \Pic(\catK(G;R))$. Then $x$ is $R$-split. 
    \end{lemma}

    \begin{proof}
        In view of \Cref{rem:prohomology-Rsplit}, it is enough to verify that $x$ has $R$-projective homology. Now, consider the tt-functor $\Res^G_1\colon\catK(G;R)\to \catK(1;R)\cong \on{D}_\mathsf{perf}(R)$. Then $\Res^G_1(x)$ is a $\otimes$-invertible perfect complex over $R$. It follows by \cite[Lemma 3.3]{Fau03} that $\oplus_i\on{H}_i(\Res^G_1(x))$ is a tensor invertible $R$-module, in particular, $\on{H}_i(\Res^G_1(x))$ is $R$-projective for all $i$, as we wanted. 
    \end{proof}

    Recall that we say $R$ is \emph{connected} if $\Spec(R)$ is connected, or equivalently, if $R$ is indecomposable. The next proposition demonstrates why this property is relevant. 

    \begin{proposition}\label{prop:etrivhomologyinonedegree}
        Assume $R$ is connected and suppose $x \in \Pic(\catK(G;R))$. Then $x$ has nonzero homology in exactly one degree, which we denote by $b(x)$. Moreover, for any $\frakp \in \Spec(R)$, the equality $b(k(\frakp)\otimes_R x)=b(x)$ holds.
    \end{proposition}
    
    \begin{proof}
        Recall that $\oplus_i\on{H}_i(\Res^G_1(x))$ is a tensor invertible $R$-module by \cite[Lemma 3.3]{Fau03}. 
        We claim $\on{H}_i(\Res^G_1(x))$ must be nontrivial on exactly one degree. Indeed, note that for any $\frakp\in \Spec(R)$, we have 
        \[ R_\frakp \otimes_R \left(\bigoplus_i\on{H}_i(\Res^G_1(x))\right)\cong \bigoplus_i \on{H}_i\left(R_\frakp \otimes_R \Res^G_1(x)  \right) \cong R_\frakp[b]\] 
        for some unique $b \in \bbZ$, where the last equivalence holds by \cite[Proposition 3.2]{Fau03}. Since $R$ is assumed to be connected, it follows that $\on{H}_i(\Res^G_1(x))=0$ for all $i\not=b$. Thus $x$ must have homology in exactly degree $b(x)\coloneqq b$. Moreover, such homology must be a tensor invertible $R$-module. 
    
        The second claim is a consequence of \Cref{rem:prohomology-Rsplit} and \Cref{lem:endotrival-Rsplit} which allow us to conclude that 
        \[
            k(\frakp) \otimes_R \on{H}_i(x) \simeq \on{H}_i(k(\frakp) \otimes_R x )
        \]
        for all $i$. 
    \end{proof}

    \begin{remark}
        However, unlike in the case of $\catK(G;k)$ with $k$ a field, the aforementioned nonzero homology need not be isomorphic as $R$-module to the free $R$-module $R$ without additional assumptions. Indeed, any element $M\in \Pic(R) \coloneqq \Pic(\mathsf{mod}(R))$ produces an endotrivial $M[i] \in \catK(G;R)$, since all invertible $R$-modules are projective. However if $R$ satisfies that all projective modules are free (e.g. if $R$ is local or a principal ideal domain), then $\Pic(R) = \{R\}$, so in this case, the nonzero homology is indeed isomorphic to the $R$-free module $R$. 
    \end{remark}
 
    \begin{remark}
        If $R$ is non-connected, \Cref{prop:etrivhomologyinonedegree} does not hold, due to the decomposition of tt-categories $\catK(G;R_1\times R_2) \simeq \catK(G;R_1) \times \catK(G;R_2)$. In this case, the function 
        \[
        b\colon \Spec(R)\to \mathbb Z, \, \frakp\mapsto b( k(\frakp) \otimes_R x)
        \]
        is locally constant. 
    \end{remark}

     \begin{corollary}\label{coro-homology-tensorprod}
         Assume $R$ is connected. Let $x$ and $y$ be two endotrivial complexes of $RG$-modules. Then  $b(x\otimes y)=b(x)+b(y)$. Moreover, it holds that 
         \[
         \on{H}_{b(x)}(x)\otimes \on{H}_{b(y)}( y)\simeq \on{H}_{b(x)+b(y)}(x\otimes y).
         \]
     \end{corollary} 

     \begin{proof}
       Note that it is enough to verify the claim after forgetting the $G$-action.  By \Cref{lem:endotrival-Rsplit}, we know that any endotrivial complex has  $R$-projective homology.  In particular, $\Res^G_1(x)$ and $\Res^G_1(y)$ satisfy the conditions in  \cite[Theorem 3.6.3]{We94} for the K\"unneth formula. That is, 
         \[
         \on{H}_\ast(\Res^G_1(x) \otimes_R\Res^G_1(y))\cong \on{H}_\ast(\Res^G_1(x))\otimes_R \on{H}_\ast(\Res^G_1 (y)) 
         \]
         from which we deduce the result. 
     \end{proof}

    \begin{definition}
    \label{def:homologyextraction}
        Assume $R$ is connected. 
        By \Cref{coro-homology-tensorprod}, we obtain that 
        extracting homology therefore induces a group homomorphism 
        \[\catH\colon \Pic(\catK(G;R)) \to \Pic(\mathsf{mod}(RG)), \, x\mapsto \on{H}_{b(x)}(x).\]
    \end{definition}

    \begin{remark}
        We note that $\catH$ need not be surjective in general; part of the reason for this is that not every $RG$-module which is $R$-free of rank 1 is $\natural$-permutation; see \Cref{prop:directsummandlemma} and \Cref{thm:rk1ppermsforpgrp}.     
    \end{remark}

    Finally, we prove a lifting theorem for ring reduction akin to Rickard's lifting of splendid Rickard equivalences \cite[Theorem 5.2]{Ric96}. In fact, the proof follows via the same techniques as \cite[Lemma 5.1]{Ric96}. Again, this lifting result will play a role in the sequel. 

    \begin{theorem}\label{thm:endotrivials_lift_from_reduction}
        Let $R$ be a commutative Noetherian ring, let $\overline{R} = R/N$ where $N \subseteq \calN$ is contained in the nilradical of $R$, and let $G$ be a finite group. Then the base change $\overline{R} \otimes_R -\colon \catK(G;R) \to \catK(G;\overline{R})$ induces a group isomorphism \[\Pic(\catK(G;R)) \cong \Pic(\catK(G;\overline{R})).\] 
    \end{theorem}
    \begin{proof}
        First, we claim that any endotrivial $x \in \catK(G;R)$ satisfies that the complex $\Hom_{RG}(x,x)$ has homology concentrated in degree zero. Indeed, by the tensor-hom adjunction, we have \[\Hom_{RG}(x,x) \cong \Hom_{RG}(R[0], x^\ast  \otimes x) \cong (x^\ast  \otimes x)^G \cong R[0]^G \cong R[0].\] Next, because $R$ is Noetherian, there exists a finite filtration of 
        \[
        \calN = \calN_n \supset \calN_{n-1} \supset \cdots \supset \calN_0 = \{0\}
        \]
        such that $\calN_{i+1}/\calN_i$ is generated by one element. By taking successive lifts, it suffices to assume $\calN$ is principally generated. Finally, we have that every $\natural$-permutation $\overline{R}G$-module lifts to a $\natural$-permutation $RG$-module by \Cref{lem:reduction_surjective_on_modules}. The result now follows identically to the proof of \cite[Lemma 5.1]{Ric96}, with $\calN$ replacing $\calM$, and where we rely on the fact that $\calN^r = 0$ for some finite $r$, rather than completeness of the ring $R$ in \cite[Lemma 5.1]{Ric96}. 
    \end{proof}
    
    \subsection{Endotrivial complexes over fields}
    
    We recall the classification of endotrivial complexes over fields of prime characteristic $p$. The classification was completed by the second-named author in a series of papers \cite{Mil24a,Mil25a,Mil25b} using representation-theoretic techniques, however Bachmann, in his dissertation \cite{Bac16}, deduced some of the same results years prior in the context of Artin motives. We assume $k$ is a field of prime characteristic $p$ and $G$ is a finite group with $p$ dividing $|G|$. Set $\catK(G) \coloneqq \catK(G;k)$.

    \begin{example}
        We list some examples of endotrivial complexes. 
        \begin{enumerate}
            \item Given any $k$-dimension 1 $kG$-module $M$, the chain complex $M[i]$ is endotrivial. Note that $M$ is a $p$-permutation module, since $\Res^G_S (M)$ has $k$-dimension 1 for a Sylow $p$-subgroup $S$, but the only $k$-dimension $kS$-module is $k$ itself.
            \item If $p = 2$ and $G = C_2$, then any truncated periodic resolution or injective hull of the trivial module $k$ is endotrivial. For instance, the length 2 complex $kC_2 \to k$ is endotrivial. If $p > 2$ and $G = C_p$, then a truncated periodic resolution or injective hull is endotrivial if and only if the (bounded) complex has odd length. These complexes were extended to arbitrary coefficients in \cite{DG25}. 
            \item Let $G = D_{2^n}$ and $p = 2$. Then there exists a length 3 endotrivial $kG$-complex \[kG \to k(G/H_1) \oplus k(G/H_2) \to k,\] where $H_1, H_2$ are nonconjugate, noncentral subgroups of order 2, and the differentials are induced by augmentation homomorphisms between each indecomposable summands in each degree. In fact, this complex arises from the reflection representation of $D_{2^n}$ as a Coexter group; we explain how one obtains endotrivials from representations in the sequel. 
        \end{enumerate}
    \end{example}

    \begin{proposition}\label{prop-conservative-detect}
        Let $\catK$ be a rigid tt-category and $\{f_i\colon\catK \to \catL_i\}_{i \in I}$ be a jointly conservative family of functors. Then $x \in \catK$ is invertible if and only if $f_i(x) \in \catL_i$ is invertible for all $i \in I$.
    \end{proposition}
    \begin{proof}
        In any rigid monoidal category, $x$ is invertible if and only if the evaluation morphism $x^\ast  \otimes x \to \bbone$ is an isomorphism. The result follows since jointly conservative functors reflect isomorphisms. 
    \end{proof}

    \begin{proposition}{\cite[Theorem 3.4]{Mil24a}}
    \label{prop-detection-endo-fields}
        Let $x \in\catK(G)$. Then $x$ is endotrivial if and only if for all $p$-subgroups $H \leq G$, $\Psi^H(x)$ has nonzero homology in exactly one degree, with that homology having $k$-dimension 1. 
    \end{proposition}
    
     \Cref{prop-conservative-detect} provides a considerably shorter version than the original proof, which relies predominantly on properties of $p$-permutation modules:
     
    \begin{proof}
        The family of functors $\{\check\Psi^H\}_{H \in \on{Sub}_p(G)/G}$ is jointly conservative by \Cref{thm-conservativity-psiH}, and the invertible objects in $\on{D}_b(k(\Weyl{G}{H}))$ are up to isomorphism given by the $k$-dimension 1 $k(\Weyl{G}{H})$-modules concentrated in a single degree. 
    \end{proof}

    \begin{definition}\label{def-hmark-field}
        Let $x \in \Pic(\catK(G))$. Set $h_x(H)\coloneqq b(\Psi^H(x))$, that is,  the unique integer for which $\on{H}_{h_x(H)}(\Psi^H(x))\neq 0$. Then $h_x\colon \on{Sub}_p(G) \to \bbZ$ is a well-defined superclass function valued on $p$-subgroups, that is, it is constant on conjugacy classes. 
        
        We denote the additive group of superclass functions of $G$ valued on $p$-subgroups by $\on{CF}(G,p)$. More generally, given a \emph{family} $\catF \subseteq \on{Sub}(G)$, i.e., a collection of subgroups closed under conjugacy, we denote the additive group of superclass functions of $G$ valued on subgroups in $\catF$ by $\on{CF}(G,\catF)$. If $\catF = \on{Sub}(G)$, we simply write $\on{CF}(G)$. 
    \end{definition}

    \begin{proposition}{\cite[Theorem 3.6]{Mil24a}}\label{prop:h_marks_for_fields}
        The assignment 
        \[
        h\colon \Pic(\catK(G)) \to \on{CF}(G,p), \, x \mapsto h_x
        \]
         is a well-defined group homomorphism, with kernel given by $k$-dimension 1 $kG$-modules concentrated in degree 0. In particular, $\ker(h) \cong \Hom(G,k^\times)$. Moreover, $\ker(h)$ is the torsion subgroup of $\Pic(\catK(G))$.  
    \end{proposition}

    We call the map $h$ the \emph{h-mark homomorphism}. To deduce the isomorphism class of $\Pic(\catK(G))$, which one could consider a classification of endotrivials, it suffices to deduce the image of $h$. To discuss this, we recall the notion of $p$-Borel-Smith functions. Note this differs slightly from previously used terminology, in which we refer to such functions simply as `Borel-Smith functions,' but we wish to emphasize the prime $p$ in this definition.

    \begin{definition}\label{def:borel_smith}
        Fix a prime $p$. A $p$-\textit{Borel-Smith function} is a superclass function $f \in \on{CF}(G,p)$ satisfying the following three conditions, which we call the Borel-Smith conditions.
        \begin{enumerate}
            \item If $p$ is odd, then for any subquotient $T/S$ of $G$ of order $p$, $f(T) \equiv f(S) \mod 2$.
            \item If $p = 2$, then for any sequence of subgroups $H \trianglelefteq K \trianglelefteq L \leq N_G(H)$, with $[K:H] = 2$, $f(K) \equiv f(H) \mod 2$ if $L/K$ is cyclic of order 4 and $f(K) \equiv f(H) \mod 4$ if $L/K$ is quaternion of order 8.
            \item For any elementary abelian subquotient $T/S$ of $G$ of rank 2, the equality \[f(S) - f(T) = \sum_{S < X < T} \big(f(X) - f(T)\big)\] holds.
        \end{enumerate}
        The collection of $p$-Borel-Smith functions forms an additive subgroup $\on{CF}_b(G,p)$ of $\on{CF}(G,p)$. 
        
        More generally, let $\catF \subseteq \on{Sub}(G)$ be a family containing all prime power subgroups of $G$. We say a superclass function $\on{CF}(G, \catF)$ is a \emph{Borel-Smith} function if it is a $p$-Borel-Smith function when restricted to $p$-subgroups for all primes $p$ dividing $|G|$. If $\catF = \on{Sub}(G)$, then we simply write $\on{CF}_b(G)$. 
    \end{definition}

    \begin{remark}
         This definition differs slightly from the classical reference \cite{tD87}, in which one considers subquotients over \emph{all subgroups} rather than those of prime power order. However, for studying endotrivials, the definition we list above we expect is the correct one to consider. 
    \end{remark}
    
    \begin{theorem}{\cite[Theorem 4.6, Corollary 6.4]{Mil25b}}\label{thm:endotriv_classification_over_field}
        We have a complete characterization of endotrivials for $\catK(G)$.
        \begin{enumerate}
            \item If $G$ is a finite $p$-group, then the $h$-mark homomorphism induces an isomorphism \[h\colon \Pic(\catK(G)) \cong \on{CF}_b(G).\]
            \item If $G$ is a finite group, we have a split exact sequence \[0 \to \Hom(G,k^\times) \to \Pic(\catK(G)) \xrightarrow{h} \on{CF}_b(G,p) \to 0.\] 
        \end{enumerate}
    \end{theorem}

    \begin{remark}
        The question of explicitly constructing endotrivials remains nebulous; we can completely classify the endotrivials by their local homological properties, but if one were to ask a passerby on the street to draw an endotrivial given its h-marks, they would likely struggle! One source of endotrivials is representation spheres (i.e., real representations), and we expand on this in \Cref{sec:rep_spheres}. But at present, given an explicit representation sphere, writing the associated endotrivial down is difficult outside of special cases, such as the reflection representation of a finite Weyl group (the second-named author thanks Christopher Hone for this observation). 
    \end{remark}

    \section{Generalized modular fixed points}\label{sec:fixedpts_hmks}

    We introduce a generalized form of modular fixed points over arbitrary rings. Though these modular fixed points are not as well-behaved over arbitrary rings, they still allow us to make local-to-global observations for permutation modules. Note that for this section we need not assume Noetherianity of $R$.

    Let $H\leq G$ be a subgroup.  We write \[\catF_H \coloneqq \{K \leq G \mid (G/K)^H = \varnothing\} = \{K \leq G \mid H \not\leq_G K\}.\] Then $\catF_H$ is a \emph{family} of subgroups of $G$, i.e. a set of subgroups of $G$ which is closed under conjugation and subgroups. 
    
    \begin{proposition}
        Let $R$ be a commutative ring of characteristic $p$ and $N \trianglelefteq G$ be a normal $p$-subgroup. Set 
        \[\mathsf{proj}(\mathscr{F}_N;R) \coloneqq \mathsf{add}^\natural\{R(G/K)\mid K \in \catF_N\},\] 
        the thick additive closure of $\{R(G/K)\mid K \in \catF_N\}$ in $\mathsf{perm}(G;R)^\natural$. Then $\mathsf{proj}(\mathscr{F}_N;R)$ is a $\otimes$-ideal of $\mathsf{perm}(G;R)^\natural$, and the composite \[\mathsf{perm}(G/N; R)^\natural \xrightarrow{\Inf^G_{G/N}} \mathsf{perm}(G;R)^\natural \twoheadrightarrow \frac{\mathsf{perm}(G; R)^\natural}{\mathsf{proj}(\mathscr{F}_N;R)}\] is an equivalence of tensor categories. 
    \end{proposition}
    \begin{proof}
        The proof follows effectively the same as the proof of \cite[Proposition 5.4]{BG25}, noting that \cite[Lemma 5.3]{BG25} holds over any ring $R$ of characteristic $p$, as $\Hom_{\mathsf{perm}(G;R)^\natural}(R(G/H), R(G/K))$ has a $R$-basis indexed by conjugacy classes of the double coset $H \backslash G/K$. 
    \end{proof}

    \begin{definition}
        Let $R$ be a ring of prime characteristic $p$. Then for any normal $p$-subgroup $N \trianglelefteq G$, one has a tensor functor on the the categories of $\natural$-permutation modules.
        \[\Psi^N\colon \mathsf{perm}(G;R) \twoheadrightarrow \frac{\mathsf{perm}(G; R)^\natural}{\mathsf{proj}(\catF_N;R)} \cong \mathsf{perm}(G/N; R)^\natural.\] Extending this construction degreewise induces tt-functor on homotopy categories \[\Psi^N\colon \catK(G;R) \to \catK(G/N;R).\] If $H \leq G$ is an arbitrary $p$-subgroup, we define the \emph{modular $H$-fixed points functor} by the composite \[\Psi^H\colon \catK(G;R) \xrightarrow{\Res^G_{N_G(H)}} \catK(N_G(H);R) \xrightarrow{\Psi^H} \catK(\Weyl{G}{H};R).\]
    \end{definition}

    If $R$ does not have characteristic $p$, we cannot define a (reasonable) modular fixed-points functor on the nose. However, if a prime $p$ generates a proper ideal of $R$, we can partially recover such a functor. Let $\lambda_p^\ast\colon \catK(G;R) \to \catK(G;R/p)$ denote the functor induced via extension of scalars $R/p \otimes_R -$. 

    \begin{definition}
        Let $\Phi(G;R)$ denote the set of primes $p$ for which $p \mid |G|$ and $pR \neq R$, i.e., $p \not\in R^\times$. If $\on{char}(R) = 0$, we also include $0 \in \Phi(G;R)$. By a 0-subgroup of $G$, we mean the trivial subgroup $1$ of $G$. 
        
        We call the elements of $\Phi(G;R)$ the \emph{singular primes} of $R$ with respect to $G$. For example, if $R$ is a field of characteristic $p$ and $p \mid |G|$, then $\Phi(G;R) = \{p\}$. 
    \end{definition}
    
    \begin{definition}
        Let $p \in \Phi(G;R)$. Given a $p$-subgroup $H \leq G$, we define the \emph{generalized $H$-modular fixed points functor} by \[\Psi^H_p \coloneqq \Psi^H\circ\lambda_p^\ast \colon \catK(G;R) \to \catK(\Weyl{G}{H};R/p).\] 
    \end{definition}

    \begin{remark}
        Note that if $p = 0$, then $\Psi^1_0$ is simply the identity functor.
    \end{remark}

    One may verify that modular fixed points behave as they should on $G$-sets. The proof follows identically to that of \cite[Proposition 5.12]{BG25}. 
    
    \begin{proposition}[c.f. {\cite[Proposition 5.12]{BG25}}]
    \label{prop-gfixedpoints-Gsets}
        Let $p \in \Phi(G;R)$ and let $H \leq G$ be a $p$-subgroup. The following diagram commutes:
        \begin{figure}[H]
            \centering
            \begin{tikzcd}
                G\on{-}\mathsf{set} \ar[r, "R(-)"] \ar[d,"(-)^H"']& \mathsf{perm}(G;R) \ar[r, hookrightarrow] \ar[d, "\Psi^H_p"] & \catK(G;R) \ar[d, "\Psi^H_p"]\\
                (\Weyl{G}{H})\on{-}\mathsf{set} \ar[r, "R/p(-)"]& \mathsf{perm}(\Weyl{G}{H};R/p)\ar[r, hookrightarrow] & \catK(\Weyl{G}{H};R/p).
            \end{tikzcd}
        \end{figure}
    \end{proposition}

    \begin{remark}
        If $p$ is a prime such that $pR \neq R$, then reduction modulo $p$ induces a closed immersion $\Spec(R/p) \hookrightarrow \Spec(R)$, and we regard $\frakp \in \Spec(R/p)$ via its image. In other words, we identify $\Spec(R/p)$ with the prime ideals of $R$ containing the ideal $pR$. We will identify $\Spec(R/p) \subseteq \Spec(R)$ as a subspace, and equivalently regard points of $\Spec(R/p)$ as points of $\Spec(R)$ containing $pR$.
    \end{remark}

    \begin{proposition}\label{prop:modularfixedptscommutes}
        Let $R, S$ be commutative rings, $p \in \Phi(G;R) \cap \Phi(G;S)$ be prime, let $H \leq G$ be a $p$-subgroup, and let $f\colon R\to S$ be a ring homomorphism. Let $f_p \colon: R/p \to S/p$ be the induced ring homomorphism. Then for all $p$-subgroups $H\leq G$, the following diagram commutes. 
        \begin{figure}[H]
            \centering
            \begin{tikzcd}
                \catK(G;R) \ar[r, "\Psi^H_p"] \ar[d, "f^\ast "]& \catK(\Weyl{G}{H}; R/p) \ar[d, "f_p^\ast "] \\
                \catK(G; S) \ar[r, "\Psi^H_p"] & \catK(\Weyl{G}{H}; S/p)
            \end{tikzcd}
        \end{figure}
    \end{proposition}
    \begin{proof}
     Let $N$ denote $N_G(H)$.   The above diagram expands out as follows.
        \begin{figure}[H]
            \centering
            \begin{tikzcd}
                \catK(G;R) \ar[d, "f^\ast "] \ar[r, "\lambda_p^\ast "] & \catK(G;R/p) \ar[d, "f_p^\ast "] \ar[r,"\Res^G_{N}"]& \catK(N;R/p) \ar[d, "f_p^\ast "] \ar[r, "\Psi^H"] & \catK(\Weyl{G}{H}; R/p) \ar[d, "f_p^\ast "]\\
                \catK(G;S) \ar[r, "\lambda_p^\ast "] & \catK(G;S/p) \ar[r,"\Res^G_{N}"] & \catK(N;S/p)\ar[r, "\Psi^H"] & \catK(\Weyl{G}{H}; S/p)
            \end{tikzcd}
        \end{figure}
        Commutativity of the first square follows since $f$ and $f_p$ commute with modding out $p$. Commutativity of the second square is clear. Commutativity of the third square essentially follows from construction, noting that the base change $f_p^\ast $ sends $\mathsf{proj}(\mathscr{F}_H;R/p)$ to $\mathsf{proj}(\mathscr{F}_H;S/p)$, and that these constructions are well-defined since $p \in \Phi(G;R) \cap \Phi(G;S)$.
    \end{proof}

    In particular, \Cref{prop:modularfixedptscommutes} holds for the base change $R \to k(\frakp)$ for any prime $\frakp \in \Spec(R)$.
    
    \begin{definition}
        Fix a finite group $G$ and a commutative ring $R$. Given a prime $p$, we say $R$ is \emph{$p$-connected} if $\Spec(R/p)$ is connected. We say $R$ is \emph{strongly connected} (with respect to $G$) if $R$ is connected and for all $p \in \Phi(G;R)$, $\Spec(R/p)$ is connected. For instance, $\mathbb Z$ is strongly connected with respect to any group. 
    \end{definition}

    \begin{remark}
        If $R$ is connected and has positive characteristic, then $R$ is strongly connected. Indeed, this follows since all primes already contain $pR$. However, connectedness need not imply strong connectedness in general. For instance, if $G$ is a $2$-group, then $R = \bbZ[x]/(x^2 - x + 2)$ is connected, as $R$ is an integral domain. However, $R/2 \cong \bbF_2[x]/(x^2 - x) \cong \bbF_2 \times \bbF_2$. Similarly, $R$ can be $p$-connected but not connected.
    \end{remark}

    We establish a local-to-global detection theorem for Noetherian rings.

    \begin{theorem}\label{thm:dir_sum_of_homologies}
        Suppose $R$ is a commutative Noetherian ring, and let $x \in \catK(G;R)$. Then $x$ is an endotrivial complex if and only if for all $p \in \Phi(G;R)$ and $H \leq G$ a $p$-subgroup, the direct sum of the homology of $\Psi^H_p(x)$ belongs to $\Pic(R/p)$. 
    \end{theorem}

    \begin{proof}
    Assume that $x$ is endotrivial. Since $\Psi^H_p$ is a tt-functor, we must get that $\Psi^H_p(x)$ is an endotrivial complex in $\catK(\Weyl{G}{H};R/p)$, and hence $\mathrm{Res}^{\Weyl{G}{H}}_1(\Psi^H_p(x))$ is an invertible perfect $R/p$-complex. The conclusion follows then by \cite[Lemma 3.3]{Fau03}. 

    For the converse, fix $p \in \Phi(G;R)$ and $H \leq G$ a $p$-subgroup, and assume that $\on{H}_\ast(\Psi^H_p(x))=\bigoplus_i \on{H}_i(\Psi^H_p(x))$  belongs to $\Pic(R/p)$.  Now, let $\frakp\in \Spec(R/p)$. We obtain a commutative diagram 
    \[
    \begin{tikzcd}
        \catK(\Weyl{G}{H};R/p) \ar[r,"\lambda_\frakp^\ast"] \ar[d,"\mathrm{Res}_1"] & \catK(\Weyl{G}{H};k(\frakp)) \ar[d,"\mathrm{Res}_1"] \\ 
       \on{D}_\mathsf{perf}(R/p) \ar[r,"\lambda_\frakp^\ast"] & \on{D}_\mathsf{perf}(k(\frakp)).
    \end{tikzcd}
    \]
    Here we are identifying $ \catK(1;S)$ with $\on{D}_\mathsf{perf}(S)$. Now, our assumption on $\Psi^H_p(x)$ gives us that $\mathrm{Res}_1(\Psi^H_p(x))$ lies in $\mathrm{Pic}(\on{D}_\mathsf{perf}(R/p))$. Using again that $\lambda_\frakp^\ast$ is monoidal, we obtain that $\lambda_\frakp^\ast(\mathrm{Res}_1(\Psi^H_p(x)))$ must lie in $\mathrm{Pic}(\on{D}_\mathsf{perf}(k(\frakp)))\simeq \mathbb Z$. By the commutativity of the above square, we deduce that $\lambda_\frakp^\ast\Psi^H_p(x))\simeq \Psi^H\lambda_\frakp^\ast(x)$ must have homology in exactly one degree, and such cohomology must be 1-dimensional. Note that the previous equivalence corresponds to the commutativity of the square in \Cref{prop:modularfixedptscommutes}. 
    
    Keep $p$ and $\frakp$ fixed. We conclude that for any $H\leq G$ a $p$-subgroup,  $\Psi^H(\lambda_\frakp^\ast(x))$ has nonzero homology in exactly one degree, with that homology having $k$-dimension 1.  We now invoke \Cref{prop-detection-endo-fields} to deduce that $\lambda_\frakp^\ast(x)$ must be an endotrivial complex.
    
    Finally, letting $\frakp$ vary, by an application of \cite[Theorem 4.8]{Gom25} combined with \Cref{prop-conservative-detect}, we deduce that $x$ must be endotrivial as we wanted. 
    \end{proof}

    As an immediate consequence we obtain the following result. 
    
    \begin{corollary}\label{thm:localglobaletrivdetection}
        Suppose $R$ is Noetherian, and let $x \in \catK(G;R)$. If for all $p \in \Phi(G;R)$ and $H \leq G$ a $p$-subgroup, $\Psi^H_p(x)$ has nonzero homology in exactly one degree, with that homology belonging to $\Pic(R/p)$, then $x$ is an endotrivial complex. If $R$ is strongly connected, the converse holds. 
    \end{corollary}
    \begin{proof}
        The first claim is a particular case of the previous theorem, hence it remains to show the converse. But this follows since $R$ is strongly connected. Indeed, for $p \in \Phi(G;R)$ we have that $R/p$ is connected, and hence the conclusion follows by \Cref{prop:etrivhomologyinonedegree} applied to $\Psi^H_p(x)$.
    \end{proof}

    \begin{remark}
         As a result of \Cref{thm:dir_sum_of_homologies}, we can refine the homomorphism $\catH$ for non-connected rings as follows: for $x \in \Pic(\catK(G;R))$, set $\catH(x) \coloneqq \bigoplus_{i \in \bbZ} \on{H}_i(x)$. It is routine to verify that $\catH$ is a group homomorphism by repeated use of \Cref{prop:etrivhomologyinonedegree} and the K\"unneth spectral sequence, since $x$ decomposes in accordance with $R$ into a sum of complexes with homology in one degree. Additionally, when $R$ is connected, this coincides with the previous definition of $\catH$. 

        The upshot is that $\catH(\Psi^H_p(x))$ is always a well-defined invertible $R_p(\Weyl{G}{H})$--module (that is, an element of $\Pic(\mathsf{mod}(R_p(\Weyl{G}{H})))$), even if $R/p$ is not connected. 
    \end{remark}
    
    \subsection{Generalized and fiberwise h-marks}
    
    We introduce two types of h-marks for endotrivials akin to those of \cite{Mil24a}, one coming from the fibers $k(\frakp)$ for $\frakp\in \Spec(R)$, and one which we define over a prime $p$ when our base ring is $p$-connected using generalized modular fixed points. Both of these notions will be useful for classification. Recall the h-mark homomorphism for fields, $h\colon \Pic(\catK(G;k))\to \mathrm{CF}_b(G,\mathrm{char}(k))$ from \Cref{def-hmark-field}. 

    \begin{definition}
        Let $x \in \Pic(\catK(G;R))$. For every $\frakp \in \Spec(R)$, $\lambda_\frakp^\ast(x)$ is endotrivial, and if $\ch(k(\frakp)) =p > 0$, there is a corresponding `local' h-mark function $h_{\lambda_\frakp^\ast (x)}\colon \on{Sub}_{p}(G) \to \bbZ$. Otherwise if $\ch(k(\frakp)) = 0$, then the only homological data available for $\lambda_\frakp^\ast(x)$ is the location of its homology; we consider this the characteristic 0 h-mark. Putting everything together, for every $\frakp\in\Spec(R)$ and $\ch(k(\frakp))$-subgroup $H \leq G$, the \emph{fiberwise h-mark of $x$ at $\frakp$ and $H$} is $h_{\lambda_\frakp^\ast(x)}(H)$.

        In fact, we obtain a homomorphism \[h_{\frakp}\colon \Pic(\catK(G;R)) \to \on{CF}_b\big(G,\ch(k(\frakp))\big),\, x \mapsto h_{\lambda_\frakp^\ast (x)},\] where the image necessarily is a Borel-Smith function by \Cref{thm:endotriv_classification_over_field}. This is the \emph{fiberwise h-mark homomorphism at $\frakp$}. Note if $\ch(k(\frakp)) = 0$ then $\on{CF}_b(G,0) \cong \bbZ$. 
        
        Moreover, the global data induces a group homomorphism 
        \begin{align*}
            h\colon \Pic(\catK(G;R)) &\to \prod_{\frakp \in \Spec(R)} \on{CF}_b(G, \ch(k(\frakp)))\\
            x &\mapsto (h_{\lambda_\frakp^\ast  (x)})_{\frakp \in \Spec(R)}
        \end{align*}
        which we refer to as the \emph{fiberwise h-mark homomorphism}. 
    \end{definition}

    This fiberwise h-mark homomorphism is rather coarse but exists regardless of connectivity conditions on $R$. Our first goal is to show that there is redundant information present, and that we can reduce down the indexing set. 
    
    On the other end of the spectrum, we introduce a finer notion of h-marks, but only for sufficiently well-behaved rings. We first record the following observation. 

    \begin{proposition}\label{prop:connected_generic_locus}
        Let $G$ be a commutative Noetherian ring, $G$ a finite group of order $n$, and $x$ an endotrivial $RG$--complex. Then for any two points $\frakp$ and $\frakq$ in the same connected component of $\Spec(R_n)$, the fiberwise h-marks $h_{\frakp}$ and $h_{\frakq}$ agree.  
    \end{proposition}
    \begin{proof}
        First of all, note that any residue field of a point in $\Spec(R_n)$ must have characteristic coprime to $n$. Then the corresponding fiberwise h-marks at primes lying in $\Spec(R_n)$ take values only in the trivial group. The result now follows by noticing that $R_n\otimes x$ must satisfy that $\bigoplus_i\on{H}_i(R_n\otimes x)$ is a tensor invertible $R_n$--module. Since base change to any residue field of a prime lying in $\Spec(R_n)$ must factor through $R[1/p]$, one deduces that the degree of $\bigoplus_i \on{H}_i(R_n\otimes x)$ which survives after tensoring with the residue field $k(\frakp)$  only depends on the connected component of $\Spec(R_n)$ containing $\frakp$. 
    \end{proof}

    \begin{definition}
       Let $p$ be a singular prime of $G$ with respect to $R$. Assume $R$ is $p$-connected and let $x \in \Pic(\catK(G;R))$ be endotrivial. \Cref{prop:etrivhomologyinonedegree} asserts $R/p \otimes_R x$ has nonzero homology in one degree. For any
        $p$-subgroup $H \leq G$, $\Psi^H_p(x) \in \Pic(\catK(\Weyl{G}{H}; R/p))$ also has nonzero homology in one degree. Let $h'_x(H)$ denote the homological degree of the nonzero homology of $\Psi^H_p(x)$. Then the function $h'_x\colon \on{Sub}_p(G) \to \bbZ$ is a \emph{superclass function}, i.e., constant on conjugacy classes. We call $h'_x$ the \emph{generalized h-marks} of $x$ at $p$. Moreover, we have a group homomorphism, the \emph{generalized h-mark homomorphism} at $p$, \[h' \colon \Pic(\catK(G;R)) \to \on{CF}_b(G,p),\;  x\mapsto h'_x.\] 
        If $R$ is strongly connected, then more generally, let $\on{Sub}_\Phi(G;R)$ denote the collection of $p$-subgroups of $G$, where $p$-runs through the primes in $\Phi(G;R)$. Then $h'_x$ may be regarded as a function $\on{Sub}_\Phi(G) \to \bbZ$. We call the function $h'_x$ the \emph{generalized h-marks} of $x$. Let $\on{CF}^\Phi(G;R)$ denote the set of superclass functions of $G$ valued on subgroups in $\on{Sub}_\Phi(G;R)$. 
    \end{definition}

    We next show that generalized h-marks can be locally computed at any compatible prime; therefore, for $p$-connected rings, the two notions of h-marks coincide at $p$-local subgroups of $G$. 
    For non-$p$-connected rings, we need not have generalized h-marks at $p$ since homology over $R/p$ need not be contained in one degree. However, when such a situation arises, we can still use global h-marks to link fiberwise h-marks. 

    \begin{proposition}\label{prop:hmarkslocallydetermined}
        Let $p \in \Phi(G;R)$ and assume $R$ is $p$-connected. Let $x \in \Pic(\catK(G;R))$, and suppose $\frakp \in \Spec(R)$ satisfies $\on{char}(k(\frakp)) = p$. Then \[h_{\lambda_\frakp^\ast(x)} \equiv h'_x|_{\on{Sub}_p(G)}.\] In particular, $h'_x|_{\on{Sub}_p(G)}$ satisfies the Borel-Smith conditions.
    \end{proposition}
    \begin{proof}
        This is an immediate consequence of \Cref{prop:modularfixedptscommutes}, which gives that modular fixed points and generalized modular fixed points are compatible, and \Cref{prop:etrivhomologyinonedegree}, which gives that the fiberwise homology agrees with the `global' homology of $R/p$. The last statement follows from \Cref{thm:endotriv_classification_over_field}.
    \end{proof}

    \begin{corollary}\label{cor:connected_component_hmks_coincide}
        Let $R$ be a connected Noetherian ring and let $p \in \Phi(G;R)$. Suppose $\frakp_1,\frakp_2$ live in the same connected component of $\Spec(R/p)$. Then $h_{\frakp_1}\equiv h_{\frakp_2}$.
    \end{corollary}
    \begin{proof}
        We have a decomposition $R/p = R_1 \times\cdots \times R_n$ into indecomposable summands and a corresponding decomposition $\Spec(R/p) = \bigsqcup_{i \in I}\Spec(R_i)$ into connected components over some finite indexing set $I$. Suppose $\frakp_1,\frakp_2 \in \Spec(R_i)$ for a fixed $i \in I$. Given any endotrivial $x \in \Pic(\catK(G;R))$, we have an endotrivial $R_i \otimes_R x \in \catK(G;R_i)$ given by base change along the homomorphism $R \to R/p \to R_i$. Note that  base change to $k(\frakp_1)$ and to $ k(\frakp_2)$ factors through the base change to $R_i$, since $\frakp_1,\frakp_2 \in \Spec(R_i) \subseteq\Spec(R)$, therefore the fiberwise h-marks of $x \in \Pic(\catK(G;R))$ at $\frakp_1$, respectively at $\frakp_2$, coincide with those of $R_i \otimes_R x \in \Pic(\catK(G;R_i))$. Since $R_i$ is $p$-connected, \Cref{prop:hmarkslocallydetermined} asserts the fiberwise h-marks coincide with the global h-marks of $R_i \otimes_R x$; in particular, they must coincide. 
    \end{proof}

    \begin{definition}\label{rmk:refined_h_mk_func}
        Let $G$ be a finite group of order $n$ and $R$ be a connected, Noetherian ring.  As a consequence of \Cref{cor:connected_component_hmks_coincide} and \Cref{prop:connected_generic_locus}, the fiberwise h-mark homomorphism is determined locally on each connected component of $\Spec(R/p)$, for all $p \in \Phi(G;R)$. We deduce that the fiberwise h-mark homomorphism takes values in the following group. First, let us set some notation. Let $\Phi(G)$ denote the set of all primes dividing $n$. If $\Phi(G;R)$ contains at least one prime, we define 
        \[
            \pi_0(\Phi(G;R))\coloneqq \pi_0\left(\left( \bigsqcup_{p\in \Phi(G)} \Spec(R/p)\right) \sqcup \Spec(R_n) \right),
        \]
        and
        \[
            \tilde \pi_0(\Phi(G;R))\coloneqq \pi_0\left( \bigsqcup_{p\in \Phi(G;R)} \Spec(R/p)\right).
        \]
        Abusing notation, we write $ \frakp \in \tilde \pi_0(\Phi(G;R))$ to denote a choice of representative for a connected component. 
        
        On the other hand, if there is no such prime, we instead choose any $\frakp\in \Spec(R)$ and set $\tilde \pi_0(\Phi(G;R))\coloneqq \{\frakp\}$. Note that if $G$ is a $p$-group, then $\tilde \pi_0(\Phi(G;R))$ is simply the connected components of $\Spec(R/p)$ or a singleton containing an arbitrary prime of $\frakp$.
        Then, the fiberwise h-mark factors through 
        \begin{align*}
            h\colon \Pic(\catK(G;R)) &\to  \prod_{\frakp\in\tilde \pi_0(\Phi(G;R))}\on{CF}_b(G,\ch(k(\frakp)))\\
            x &\mapsto (h_{k(\frakp) \otimes_R x})_{\frakp \in \tilde \pi_0(\Phi(G;R))}
        \end{align*}
        We regard this homomorphism the \emph{(reduced) fiberwise h-mark homomorphism}, and unless we specify otherwise, let $h$ denote the reduced version. 
        Note the cardinality of $\pi_0(\Phi(G;R))$ is always finite! For instance, if $R=\mathbb{Z}$, then $\pi_0(\Phi(G;R))$ has as many elements as $n$ has prime divisors plus the number of connected components of $\Spec(R_n)$, and  $\tilde \pi_0(\Phi(G;R))$ has an element for each prime dividing $n$.  
    \end{definition}

    \begin{remark}
        One may also define a `generalized h-mark homomorphism' from $\Pic(\catK(G;R))$ for $p$-connected or strongly connected rings. These will not play as much of a role from here, and as we have seen, can already be recovered from the fiberwise h-mark homomorphisms.
    \end{remark}

    \section{Representation spheres and the orbit category}\label{sec:rep_spheres}

    Shifting gears, we show that, like in the case of endotrivials over a field, or the case of genuine $G$-spectra (see e.g. \cite[Section 4]{FLM01}), representation spheres induce endotrivials. We follow the exposition of \cite{HY14, Y17}.

    \begin{recollection}
        Let $G$ be a finite group. A \emph{geometric homotopy representation} is a finite $G$-CW-complex $X$ for which each fixed point set $X^H$ is homotopy equivalent to a sphere, for any subgroup $H\leq G$. A \emph{representation sphere}, a one-point compactification of a real (orthogonal) representation of $G$, can be given a finite $G$-CW-complex structure, and as such, is an example of a geometric homotopy representation. 

        Given a geometric homotopy representation $X$, we define its \emph{dimension function} $\dim_X\colon  \on{Sub}(G) \to \bbN$ by $\dim_X(H) = \dim(X^H)$. Similarly, given a (generalized) character $\rho$, we define the dimension function in the analogous way. The dimension function can be computed directly via \[\dim_\rho(H) = \frac{1}{|H|}\sum_{g \in H} \rho(g),\] where $\rho$ denotes the real-valued character associated to $X$.
    \end{recollection}

    \begin{recollection}\label{def:orbit_cat}
        Let $G$ be a finite group. The \emph{orbit category} $\calO(G)$ is defined as the full subcategory of $G\on{-}\mathsf{set}$ whose objects are the finite $G$-sets $G/H$ with $H \leq G$. In particular, morphisms are $G$-set homomorphisms. If $\catF$ is a subset of the subgroups of $G$ closed under $G$-conjugation, we define $\calO_\catF(G)\subseteq \calO(G)$ to be the full subcategory of $\calO(G)$ with objects the $G$-sets with isotropy subgroups in $\catF$, i.e., $G/H$ with $H \in \catF$. 

        We can consider the module category over $\calO(G)$ in the following sense. For any commutative ring, a (right) $R\calO(G)$-module $M(-)$ is a contravariant functor $\calO(G)  \to \mathsf{mod}(R)$. We denote the $R$-module $M(G/K)$ by $M(K)$ for shorthand, and write $M(f)\colon M(L) \to M(K)$ for a $G$-map $G/K \xrightarrow[]{f} G/L$. Note that the $R$-module $M(K)$ is in fact a (right) $R(\Weyl{G}{K})$-module via the action of $\on{Out}(G/K) = \Weyl{G}{K}$. In particular, $M(1)$ is a (right) $RG$-module. All constructions may be translated into statements about left modules as well, since $G \cong G\op$, and we do so without further mention. 

        The category of $R\calO(G)$-modules is abelian, so usual homological algebraic concepts apply. Note kernels and cokernels are computed pointwise, that is, a sequence \[0 \to A(-) \to B(-) \to C(-) \to 0\] is exact if and only if for all $H \in \on{Sub}(G)$, \[0\to A(H) \to B(H) \to C(H) \to 0\] is an exact sequence of $R$-modules. 

        For a fixed subgroup $H\leq G$, we write $R(G/H)^?$ to denote the $R\calO(G)$-module given by the contravariant  functor  $K\mapsto R((G/H)^K)$ which maps an equivariant map $G/K\xrightarrow[]{g}G/L$ to the corresponding $R$-linear extension of multiplying by $g$, $R((G/H)^K)\to R((G/H)^L)$.

        One has by Yoneda that for every subgroup $H\leq G$, the $R\calO(G)$-module $R(G/H)^?$ is projective. A $R\calO(G)$-module is \emph{free} if it is isomorphic to a direct sum of $R\calO(G)$-modules of this form, and in fact, a $R\calO(G)$-module is projective if and only if it is a direct summand of a free module. We write $\mathsf{proj}(R\calO(G))$ to denote the full subcategory of $\mathsf{mod}(R\calO(G))$ on projective $R\calO(G)$-modules. 
    \end{recollection}

    \begin{construction}
        Given a $G$-CW-complex $X$, there is an associated chain complex of $R\calO(G)$-modules \[\tilde{C}(X^?;R): \cdots \to R(X_n^?) \xrightarrow{\partial_n} R(X_{n-1}^?) \xrightarrow{\partial_{n-1}} \cdots \xrightarrow{\partial_1} R(X_0^?) \to \underline{R} \xrightarrow{\varepsilon} 0,\] where $X_i$ denotes the set of (oriented) $i$-dimensional cells in $X$, $R(X_i^?)$ denotes the associated free $R\calO(G)$-module given by $R(X_i^?)(H)\coloneqq R(X_i^H)$, $\underline{R}$ denotes the trivial representation satisfying $\underline{R}(H) = R$ for all $H \leq G$, and $\varepsilon$ denotes the augmentation homomorphism. 
    \end{construction}

     \begin{remark}
         Observe that evaluation at $H \leq G$ induces a functor 
        \[
            -(H)\colon\on{K}_b(\mathsf{proj}(R\calO(G))) \to \catK(\Weyl{G}{H};R).
        \]
     \end{remark}
    
     The next proposition, whose proof is a routine exercise using \Cref{prop-gfixedpoints-Gsets}, demonstrates that evaluation at subgroups identifies with modular fixed points. 

    \begin{proposition}\label{prop:evandmodfixedptscommute}
        Let $G$ be a finite group, $p \in \Phi(G;R)$, and $H \leq G$ a $p$-subgroup. Then the following diagram commutes.
        \begin{figure}[H]
            \centering
            \begin{tikzcd}
                \on{K}_b(\mathsf{proj}(R\calO(G))) \ar[r, "-(H)"] \ar[d, "-(1)"'] & \catK(\Weyl{G}{H};R) \ar[d, "\lambda_p^\ast "] \\
                \catK(G;R) \ar[r, "\Psi_p^H"]& \catK(\Weyl{G}{H};R/p)
            \end{tikzcd}
        \end{figure}
    \end{proposition}

    \begin{theorem}\label{prop:hmkisdimfunc}
        Let $R$ be a commutative ring, let $G$ be a finite group, and let $S(V)$ be a representation sphere for $G$. Then \[x \coloneqq \tilde{C}(S(V)^1; R) \in \catK(G;R)\] is endotrivial. 
        Moreover, the generalized h-mark function (if it exists) of $x$ satisfies \[h'_x(H) = \dim_V(H)\] for all $p$-subgroups $H \leq G$, with $p \in \Phi(G;R)$. 
    \end{theorem}
    \begin{proof}
        First note that it suffices to assume $R = \bbZ$, as $\tilde{C}(-; R) \cong R \otimes_\bbZ\tilde{C}(-; \bbZ)$ and extending scalars preserves endotriviality. Since $S(V)^H$ is again a $\Weyl{G}{H}$-representation sphere for every subgroup $H \leq G$, $x_H \coloneqq \tilde{C}(S(V)^H;R)$ has nonzero homology in degree $i = \dim(S(V)^H)$ and is exact in all other degrees. Moreover, $\on{H}_i(x_H)$ is free as $R$-module of rank 1. By \Cref{rem:prohomology-Rsplit}, we obtain that $x_H$ is $R$-split, and from \Cref{lem:commute_with_homology} we obtain 
        \[
        R/p \otimes_R \on{H}_\ast(x_H)\simeq \on{H}_\ast(R/p\otimes_R x_H).
        \]       
        Now, \Cref{prop:evandmodfixedptscommute} implies that if $H$ is a $p$-subgroup of $G$ for $p \in \Phi(G;R)$, then $R/p \otimes_R x_H \cong \Psi^H_p(x)$. Therefore \Cref{thm:localglobaletrivdetection} implies that $x$ is endotrivial, as desired.  For the final statement, we have that $x_H$ has nonzero homology only in the degree corresponding to the dimension of $S(V)^H$, which is precisely $\dim_V(H)$. The result now follows again from the isomorphism $R/p \otimes_R x_H \cong \Psi^H_p(x)$ for $H $ and $ p$ as specified. 
    \end{proof}

    \begin{proposition}\label{prop:def_of_B}
        Let $R$ be a commutative ring, and let $G$ be a finite group. The assignment $V \mapsto \tilde{C}(S(V)^1; R)$ induces a group homomorphism \[\catB\colon\on{RO}(G) \to \Pic(\catK(G;R)),\] where $\on{RO}(G)$ denotes the real representation ring of $G$.
    \end{proposition}
    \begin{proof}
        If $V = [V_1] - [V_2] \in \on{RO}(G)$ for two real representations $V_1, V_2$, we set \[\catB(V) = \tilde{C}(S(V_1)^1;R) \otimes_R \tilde{C}(S(V_2)^1;R)^\ast .\] It suffices to check $\tilde{C}(S(V\oplus W)^1;R) \cong \tilde{C}(S(V)^1;R) \otimes \tilde{C}(S(W)^1;R)$ in $\catK(G;R)$ for any two real representations $V,W$ of $G$. We have an isomorphism $S(V \oplus W) \cong S(V)\wedge S(W)$, see e.g. \cite[A.2.5]{HHR16}, and there is an standard isomorphism $\tilde{C}((S(V)\wedge S(W))^?;R)\cong \tilde{C}(S(V)^?;R) \otimes \tilde{C}(S(W)^?;R)$, see e.g. \cite[Lemma 8.2]{K11}.
    \end{proof}

    \begin{remark}    
        The homomorphism $\catB$ is generally non-injective. For instance, if $k$ is a field of prime characteristic $p$ and $G$ is a $p$-group, then if $V$ and $W$ are Adams-conjugate, $\catB(V) = \catB(W)$, see \cite[Proposition 5.9, Page 213]{tD87}. This coincides with the stable homotopy-theoretic situation for the analogous map $\on{RO}(G) \to \Pic(\on{SH}(G))$, see e.g. \cite{Kra25}.
    \end{remark}

    \begin{remark}
        Note that in \cite{tDP82, HY14, Y17}, there is a (harmless) change in convention, as dimension functions are defined with a shift, and the authors take joins of spheres rather than smash products. This is due to the fact that the authors consider \emph{unit spheres} rather than \emph{representation spheres} of real representations; a representation sphere of $V$ is a unit sphere of $\bbR \oplus V$. See e.g. \cite[Page 2]{MP04}. In this case, if $S'(V)$ denotes the unit sphere of $V$, then $S'(V\oplus W) \cong S'(V) \ast S'(W)$. 
    \end{remark}

    \subsection{The h-mark homomorphism for $p$-connected rings}
    
    For $p$-connected rings, characterizing the image of the fiberwise and generalized h-mark homomorphism for $p$-groups is now easy: it does not differ from the field case. 
    
    \begin{theorem}\label{thm:p_grp_h_mark_surj}
        Let $p$ be prime, let $R$ be a $p$-connected commutative Noetherian ring satisfying $pR \neq R$, and let $G$ be a $p$-group. Then the composition \[\on{RO}(G) \xrightarrow{\catB} \Pic(\catK(G;R)) \xrightarrow{h'} \on{CF}_b(G)\] is surjective, where $h'$ denotes the generalized h-mark homomorphism at $p$. In particular, $h'$ is surjective. 
    \end{theorem}
    \begin{proof}
        \cite[Theorem 5.4, Page 211]{tD87}, asserts that the dimension function homomorphism \[\dim\colon \on{RO}(G) \to \on{CF}_b(G)\] is surjective for nilpotent groups. This homomorphism factors through $\Pic(\catK(G;R))$ by \Cref{prop:hmkisdimfunc}, and the result follows.
    \end{proof}

    \begin{remark}
        The above homomorphism is also surjective if $R$ is Noetherian and satisfies $pR = R$, for trivial reasons. 

        If $R$ is not necessarily $p$-connected, then one may formulate an analogous surjectivity statement for the fiberwise h-marks at every prime $\frakp \in \Spec(R)$. We leave the details to the reader. The remaining, difficult, question for $G$ a $p$-group is to determine relationships between h-marks on different connected components of $R/p$. 
    \end{remark}

    \section{Gluing endotrivials}\label{sec:gluing}

    Even if $R$ is connected of characteristic 0, $R/p$ need not be, so even though $\catK(G;R)$ is indecomposable as a tt-category, $\Psi^H_p$ may pass to a decomposable tt-category. In this case, endotrivials over $R/p$ need not have homology concentrated in one degree $p$-locally. We consider this case when $G$ is a $p$-group and show that indeed, more endotrivials can arise, but we can describe them by gluing from each connected component of $\Spec(R/p)$.  

    Recall for each open $U \subseteq \Spec(R)$, we have a corresponding open $\comp\inv(U) \subseteq \Spc( \catK(G;R))$. If $U = D(f)$ for some element $f \in R$ (morally considered as an endomorphism of $\bbone \in \catK(G;R)$), we denote the corresponding open by $\open(f) \subseteq \Spc(\catK(G;R))$, which as a set is the subset \[\open(f) = \bigsqcup_{f\not\in \frakp \in \Spec(R)}\Spc(\catK(G;k(\frakp))) \subseteq \Spc(\catK(G;R)).\]

    \begin{lemma}\label{lem:intersections_in_generic_locus}
        Let $G$ be a $p$-group and let $R$ be a connected Noetherian ring. The open $D(p) \subseteq \Spec(R)$ consists of all primes $\frakp$ for which $R/\frakp$ has characteristic not $p$. Consequently, if $R_p=R[1/p] \neq 0$, then all $\natural$-permutation $R_pG$-modules are projective, the open $\open(p) \subseteq \Spc(\catK(G;R))$ is in bijection with $D(p)$, and \[\catK(G;R)(\open(p)) \cong \catK(G;R_p) \cong \on{D}_\mathsf{perf}(R_pG).\]Moreover, if $f \in R$ satisfies $D(f) \subseteq D(p)$, then \[\catK(G;R)(\open(f))\cong \on{D}_\mathsf{perf}(R_fG).\]
    \end{lemma}
    \begin{proof}
        We may assume $R_p \neq 0$, as the statement is trivial otherwise. The first statement is routine. For the bijection of sets, $\open(p)$ is the disjoint union of subsets $\Spc(\catK(G;k(\frakp))) \subseteq \Spc(\catK(G;R))$ over all $\frakp$ not containing $p$. In these cases, $k(\frakp)G$ is semisimple, so $\Spc(\catK(G;k(\frakp)))$ is a point, corresponding to the point $\frakp \in \Spec(R)$, and the bijection follows. Now, we have by a Maschke argument that every submodule of the free module $R_pG$ is a direct summand. Since $R_p(G/H)$ injects into $R_pG$ via the coaugmentation homomorphism, the map \[R_p(G/H) \to R_pG, \, gH \mapsto \sum_{g' \in [G/H]} gg',\] every $\natural$-permutation module is projective. Consequently,  
        we obtain the equivalence $\mathrm{perm}(R_pG)\cong\mathrm{proj}(R_pG)$, and hence $\catK(G;R_p) \cong \on{D}_{\mathsf{perf}}(R_pG)$. The equivalence $\catK(G;R)(\open(p)) \cong \catK(G;R_p) $ follows from \Cref{prop:localization_identification}. The last statement is an easy consequence, since if $D(f) \subseteq D(p)$, then $p$ is again invertible in $R_f G$. 
    \end{proof}

    We call the open $\open(p)$ the \emph{generic locus} of $\Spc(\catK(G;R))$ and similarly for $D(p) \subseteq \Spec(R)$. Now, we construct an open cover of the spectrum that stratifies each connected component of $R/p$, as well as the generic locus. 

    \begin{construction}\label{cons:open_cover_of_spc}
        Fix $G$ a $p$-group, and let $R$ be a commutative Noetherian ring with $\ch(R) = 0$. We assume for non-triviality purposes that $pR \neq R$ and that $R/p$ is non-connected. Consider the unit $1 +pR \in R/p$. Connected components of $\Spec(R/p)$ are in bijection with the components of an orthogonal decomposition of $1 +pR \in $. Let $a_1 + pR, \cdots, a_n + pR$ be the primitive orthogonal idempotents, that is, $a_1 + \cdots + a_n + pR = 1 + pR$ and $a_ia_j \in pR$ for all $i,j$. Note this list is necessarily finite since $R/p$ is Noetherian. We have that $D(a_i + pR) \subseteq \Spec(R/p)$ is a maximally connected open, i.e., a connected component. In particular, $\{D(a_i + pR)\}_{i=1}^n$ is an open cover of $\Spec(R/p)$ and for all $i\neq j$, $D(a_i + pR) \cap D(a_j + pR) = \varnothing$.
        
        Let $a_1, \dots, a_n$ be lifts of each idempotent to $R$. Set $a_0 = p$. Then we claim the following:
        \begin{enumerate} 
           \item each $\open(a_i)$ is quasi-compact open.
            \item $\{\open(a_i)\}_{i=0}^n $ is a cover of $\Spc(\catK(G;R))$.
            \item For all $i \neq j$, $\open(a_i) \cap \open(a_j) \subseteq \open(p)$. 
        \end{enumerate}
        Indeed, (a) follows since $\Spc(\catK(G;R))$ is Noetherian; see \cite[Section 6]{Gom25}. For (b) and (c),   it suffices to check this in $\Spec(R)$, and both of these facts follow easily from the fact that $\Spec(R/p) \subseteq \Spec(R)$ is the closed subset consisting of primes with residue field having characteristic $p$. In particular, $\Spec(R) \setminus \Spec(R/p)$ is the generic locus. 

        Note that it need not be the case that the opens $\open(a_i)$ (or even $D(a_i)$) are connected, however this is fine. The key point we require is that $D(a_i) \cap \Spec(R/p)$ is connected for each $i \in \{1, \dots, n\}$. 
        
    \end{construction}

    We recall the following gluing result of Balmer--Favi. 

    \begin{definition}
        Let $\catK$ be an essentially small tt-category $\Spc(\catK) = U_1 \cup \cdots\cup U_n$ be a cover by quasi-compact opens. Given objects $a_i \in \catK(U_i)$ and isomorphisms $\sigma_{i,j}$ in $\catK(U_i \cup U_j)$ such that $\sigma_{ki} = \sigma_{kj}\sigma_{ji}$ in $\catK(U_i\cap U_j \cap U_k)$ for $1\leq i,j,k\leq n$, a \emph{gluing} of the objects $a_i$ along isomorphisms $\sigma_{ij}$ is an object $a \in \catK$ and isomorphism $f_i\colon a \to a_i$ in $\catK(U_i)$ such that $\sigma_{ji} f_i = f_j$ in $\catK(U_i \cup U_j)$
    \end{definition}

    \begin{theorem}{\cite[Theorem 5.13]{BF07}}
        Let $\catK$ be an essentially small tt-category and let $\Spc(\catK) = U_1 \cup \cdots \cup U_n$ be a cover by quasi-compact open subsets for $n \geq 2$. Consider objects $a_i \in \catK(U_i)$ and isomorphisms $\sigma_{ji}\colon a_i \cong a_j$ in $\catK(U_i \cap U_j)$ satisfying the cocycle condition $\sigma_{kj}\sigma_{ji} = \sigma_{ki}$ in $\catK(U_i \cap U_j \cap U_k)$ for $1 \leq i, j, k\leq n$. Assume moreover the following \emph{connectivity condition}: for any $i = 2, \dots, n$ and any quasi-compact open $V \subseteq U_i$, we suppose that \[\Hom_{\catK(V)}(\Sigma a_i, a_i) = 0.\] Note it suffices to check at unions of intersections of $U_1, \dots, U_n$. Then there exists a gluing, unique up to unique isomorphism.
    \end{theorem}

    \begin{remark}\label{rmk:local_endotriv_observations}
        Before we prove our gluing theorem, we make a few preliminary observations. First, by \Cref{prop:localization_identification}, in our open cover $\{\open(a_i)\}_{i=0}^n$, we have identifications $\catK(G;R)(\open(a_i)) \cong \catK(G;R_{a_i})$ and the localization 
        \[
        \catK(G;R)\to \catK(G;R)(\open(a_i)) \cong \catK(G;R_{a_i})
        \]
        identifies with extension of scalars $R_{a_i} \otimes_R -$. Therefore, given an object $x \in \catK(G;R)$, $x$ is endotrivial if and only if $x \in \catK(G;R)(\open(a_i))$ is endotrivial for every $i \in \{0,\dots, n\}$; indeed, this follows since endotriviality is detected by the collection of base change functors $\{\lambda_\frakp^\ast\}_{\frakp \in \Spec(R)}$, and each base change to $k(\frakp)$ factors through either base change to $R_{a_i}$ for some $a_i$ or base change to $R_p$. 

        Moreover, observe that if $\frakp \in \Spec(R/p)$ lives in $D(a_i)\cap \Spec(R/p) \subseteq \Spec(R)$, then the base-change functor $k(\frakp) \otimes_R - \colon \catK(G;R) \to \catK(G;k(\frakp))$ factors as follows: \[\catK(G;R) \to \catK(G;R/p) \to \catK(G;(R/p)_{a_i}) \to \catK(G;k(\frakp)).\] 
        Since $(R/p)_{a_i}$ has connected spectrum and has characteristic $p$, any endotrivial $x \in \catK(G;(R/p)_{a_i})$ has well-defined global h-marks at $p$, which then coincide with the usual fiberwise h-marks of $\lambda_\frakp^\ast(x)$ for any $\frakp \in D(a_i)$. Therefore, on each connected component on $R/p$, fiberwise h-marks coincide, and all that remains is for us to check whether there exist relations between non-connected components. 
    \end{remark}

    \begin{theorem}\label{thm:gluing_etrivs_p_grps}
        Let $G$ be a $p$-group and $R$ a commutative connected Noetherian ring. Recall the reduced fiberwise h-mark homomorphism from \Cref{rmk:refined_h_mk_func}, \[\Pic(\catK(G;R)) \to \prod_{\frakp \in \tilde \pi_0(\Phi(G;R))} \on{CF}_b(G,\ch(k(\frakp))).\] The image is precisely the tuples of Borel-Smith functions which agree on the trivial subgroup of $G$. Note $\tilde \pi_0(\Phi(G;R))$ is indexed by the connected components of $\Spec(R/p)$. 
    \end{theorem}
    \begin{proof}
        Note that if $\ch(R) = p^r$ for some $r$, then $R/p$ is connected if and only if $R$ is, in which case \Cref{thm:p_grp_h_mark_surj} proves the result. Similarly if $pR = R$, then the result is trivial, so assume neither case holds, in which case $\ch(R) = 0$. 
    
        First, since $R$ is connected, \Cref{prop:etrivhomologyinonedegree} holds and directly implies that the tuples of functions must agree at the trivial subgroup. It only remains to show that this is the only condition necessary. It suffices to show that there do not exist any relationships between the fiberwise h-marks corresponding to primes on different connected components of $\Spec(R/p)$ outside of those at 1. To do this, we construct endotrivials in $\catK(G;R)$ whose h-marks are constant away from one connected component. 
        
      Using the notation of \Cref{cons:open_cover_of_spc}, let $f \in \on{CF}_b(G)$ be a Borel--Smith function, and fix $l\in\{1,\dots,n\}$. We claim that there exists an endotrivial object $x \in \Pic(\catK(G;R))$ satisfying the following properties:
        \begin{enumerate}
            \item $h_{\lambda_\frakp^\ast (x)} = f$, for all $\frakp$ in the connected component $D(a_l+pR) \subseteq \Spec(R)$;
            \item $h_{\lambda_\frakq^\ast (x)} \equiv f(1)$, for all $\frakq\not\in D(a_l+pR)$.
        \end{enumerate}
         We construct this endotrivial by gluing on each component $\open(a_i)$ of $\Spc(\catK(G;R))$.   By \Cref{prop:localization_identification}, we have $\catK(G;R)(\open(a_i)) \cong \catK(G;R_{a_i})$. Moreover, $D(a_i) \cap D(a_j) \subseteq D(p) \subseteq\Spec(R)$ as noted in \Cref{cons:open_cover_of_spc}, so by \Cref{lem:intersections_in_generic_locus}, \[\catK(G;R)(\open(a_i) \cap \open(a_j)) \cong \catK(G;R_{a_ia_j}) \cong \on{D}_\mathsf{perf}(R_{a_ia_j}G).\] 
        Now, for $i \neq l$, we specify the local endotrivial $x_i \coloneqq R[f(1)] \in \catK(G;R)(\open(a_i))$. On the other hand since $\catK(G;R)(\open(a_l)) \cong \catK(G;R_{a_l})$ and $R_{a_l}$ is $p$-connected, \Cref{thm:p_grp_h_mark_surj} asserts that there exists an endotrivial $x_l \in \catK(G;R_{a_l})$ with global h-mark function $f$, and we may normalize $x_l$ to satisfy $\catH(x_l) = R$. 

        It remains to check the cocycle and connectivity conditions of \cite[Theorem 5.13]{BF07}, however these are easy since any intersection not involving $l$ has cocycles given by the identity. Similarly, $\Hom_{\catK(G;R)(V)}(\Sigma x_i, x_i) = 0$ for any union of intersections of opens of $\open(a_i)$s, again since the unions of intersections of $D(a_i)$s always remain in the generic locus of $\Spec(R)$, so \Cref{lem:intersections_in_generic_locus} applies. It follows that $\bbone \in \catK(G;R)(V)$ has no nontrivial extensions, so connectivity applies. Thus we conclude there exists an $x \in \catK(G;R)$ which descends to each $x_i$ in each open $\open(a_i)$, and by our observations in \Cref{rmk:local_endotriv_observations}, $x$ is endotrivial with fiberwise h-marks corresponding to $f$ on the component $D(a_l)$ and equivalently $f(1)$ otherwise. The result follows. 
    \end{proof}

    \begin{notation}\label{not:glued_borel_smith}
        For a finite $p$-group $G$ we let  $\on{CF}^{gl}_b(G;R)$ denote the subgroup of $\prod_{\tilde \pi_0(\Phi(G;R))} \on{CF}_b(G)$ consisting of tuples of Borel-Smith functions agreeing on the trivial subgroup of $G$. We refer to these as `glued' tuples. Then \Cref{thm:gluing_etrivs_p_grps} states that the reduced fiberwise h-mark homomorphism can be refined to a surjective homomorphism 
        \[
        \Pic(\catK(G;R)) \twoheadrightarrow \on{CF}^{gl}_b(G;R).
        \]
        This surjection splits since $\on{CF}^{gl}_b(G;R)$ is a free abelian group, and the proof of \Cref{thm:gluing_etrivs_p_grps} essentially constructs an explicit splitting. Therefore, to deduce $\Pic(\catK(G;R))$ for $G$ a $p$-group, it only remains to deduce the kernel of the (refined) fiberwise h-mark homomorphism. 
    \end{notation}

    \begin{remark}
        Another consequence of \Cref{thm:gluing_etrivs_p_grps} is that as soon as $R/p$ is disconnected, $\catB\colon \on{RO}(G) \to \Pic(\catK(G;R))$ is non-surjective for $G$ a $p$-group, contrasting the case of $R$ a field of prime characteristic $p$. 
    \end{remark}
    
    \section{The kernel of the fiberwise h-mark homomorphism}\label{sec:kernel_of_hmks}

    In this section, we deduce the kernel of the fiberwise h-mark homomorphism. Recall that in the case of fields of prime characteristic, $\ker(h)\cong \Hom(G,k^\times)$. In particular, if an endotrivial's h-marks are entirely 0, then the endotrivial is nothing more than an invertible $kG$-module in degree 0. We witness the same phenomenon here.

    \begin{theorem}\label{thm:ker_h_marks}
        Let $G$ be a finite group, $R$ a connected commutative Noetherian ring, and $x\in \Pic(\catK(G;R))$ endotrivial. If $h_\frakp \equiv 0$ for all $\frakp \in \Spec(R)$, then $x \cong M[0]$ for some invertible $\natural$-permutation $RG$-module $M \in \Pic(\mathsf{perm}(G;R)^\natural)$. In other words, $\ker(h) \cong \Pic(\mathsf{perm}(G;R)^\natural)$, where $h$ denotes the reduced fiberwise h-mark homomorphism.
    \end{theorem}
    \begin{proof}
        It suffices to assume $\catH(x) = R$. Indeed, since $\catH(x) \in \Pic(\mathsf{mod}(RG))$, we may replace $x$ with $x' = x \otimes (\catH(x)^\ast )[0]$ which satisfies $\catH(x') = R$, as $\calH$ is a well-defined group homomorphism by \Cref{def:homologyextraction}. Moreover, it suffices to assume that $R$ is reduced by \Cref{thm:endotrivials_lift_from_reduction} and \Cref{lem:reduction_surjective_on_modules}, which together assert that endotrivial $\natural$-permutation modules lift uniquely. 

        First, for every $\frakp \in \Spec(R)$, the following diagram commutes. 

        \begin{figure}[H]
            \centering
            \begin{tikzcd}
                \catK(G;R) \ar[r, "\lambda_\frakp^\ast "] \ar[d, "\Res^G_1"']& \catK(G; k(\frakp)) \ar[d, "\Res^G_1"] \\
                \catK(1;R) \ar[r, "\lambda_\frakp^\ast "] & \catK(1; k(\frakp))
            \end{tikzcd}
        \end{figure}

        Note $\catK(1;R) = \on{K}_b(\mathsf{proj}(R))$ and $\catK(1;k(\frakp)) = \on{K}_b(\mathsf{vec}(k(\frakp)))$ Let $x \in \Pic(\catK(G))$ satisfy $h_x \equiv 0$ and $\catH(x) = R$. Then $\lambda_\frakp^\ast(x) \cong k[0]$ in $\catK(G;k(\frakp))$ by \Cref{prop:h_marks_for_fields}, and a routine homological algebra argument verifies $\Res^G_1 (x) \cong R[0]$ in $\catK(1;R)$. Therefore, there exists a free $R$-module summand $M \subseteq x_0$ satisfying $M \subseteq \ker(d_0)$ and $M \cap \im(d_1) = 0$. Therefore, $\lambda_\frakp^\ast(M)$ identifies with $k[0]$ in $\catK(G;k(\frakp))$. 
        
        We claim that $M$ is in fact a $RG$-submodule of $x_0$, in other words, $G \cdot M \subseteq M$. Let $x_0 = M \oplus N$ as $R$-modules, then $N$ is $R$-projective. Now, this follows because we assume $R$ is reduced: in this case, for every nonzero $n \in N$, there exists some $\frakp \in \Spec(R)$ for which $n$ has nonzero image $k(\frakp) \otimes_R N$. Therefore, $p_N(g \cdot m) \neq n$, where $p_N$ denotes the projection $M \oplus N \to N$, since otherwise, the $g$-action would be detected under $\lambda_\frakp^\ast $. W conclude $G \cdot M \subseteq M$.

        Now, since $M$ is a $RG$-submodule of $x_0$, we can construct an injective chain complex homomorphism $f\colon R[0] \to x$ with image in degree 0 exactly $M$. We have 
        \[
        \lambda_\frakp^\ast (f)\colon k(\frakp)[0] \to x\simeq k(\frakp)[0]
        \]
        is an isomorphism, for all $\frakp \in \Spec(R)$, so conservativity implies $x \cong R[0]$, as desired. 
    \end{proof}

    \begin{observation}\label{obs:ker_h_mark_decomp}
        We obtain a split injective group homomorphism \[\Pic(\mathsf{perm}(G;R)^\natural) \to \catK(G;R), M \mapsto M[0]\] with retraction given by $\catH$. For $p$-groups over commutative connected Noetherian rings, this affords a short exact sequence
        \[ 0 \to \Pic(\mathsf{perm}(G;R)^\natural) \to \Pic(\catK(G;R)) \to \on{CF}_b^{gl}(G)\to 0.\] 
        All that remains is to deduce $\Pic(\mathsf{perm}(G;R)^\natural)$, and this can be reduced to considering free $R$-rank 1 $\natural$-permutation modules as follows. Recall $\Pic(R)$ denotes the group of line bundles over $R$, i.e., invertible $R$-modules. If $R$ is a commutative Noetherian ring, then we moreover have an split injective group homomorphism \[\Pic(R) \hookrightarrow  \Pic(\mathsf{perm}(G;R)^\natural)\] with retract given by restriction $\Res^G_1$. This affords a decomposition 
        \[
        \Pic(\mathsf{perm}(G;R)^\natural) = \Pic(R) \times \catM(G;R)
        \]
        where $\catM(G;R)$ denotes the subgroup of $ \Hom(G,R^\times)$ consisting of homomorphisms $\varphi$ for which the free $R$-rank 1 $RG$-module $R_\varphi$ is $\natural$-permutation. All that remains is to determine $\catM(G;R)$, or in other words, determine a criteria for when $R_\varphi$ is $\natural$-permutation.
    \end{observation}

    \subsection{Characterizing free $R$-rank 1 $\natural$-permutation modules}

    Throughout, we only assume $R$ is a commutative ring. 

    \begin{notation}
        For $\varphi \in \Hom(G, R^\times)$ and $H \leq G$, let $I^\varphi_H\subseteq R$ denote the ideal 
        \[I^\varphi_H \coloneqq  \{r \in R \mid r = \varphi(h)r \text { for all } h \in H\}.\] 
        In other words, $I^\varphi_H$ is the annihilator ideal of the collection of elements \[\{1 - \varphi(h)\}_{h \in H}.\] 
    \end{notation}

    \begin{remark}
        By the induction/restriction adjunction, any $RG$-module homomorphism $f\colon R_\varphi \to R(G/H)$ is induced from the assignment \[1 \mapsto a\sum_{g\in [G/H]} \varphi(g\inv)gH,\] for some $a \in R$, however this assignment may not always be well-defined. In fact, it is an easy verification that $f$ is well-defined if and only if $a \in I^\varphi_H$. Similarly, any $RG$-module homomorphism $g\colon R(G/H) \to R_\varphi$ is induced from the assignment \[gH \mapsto \varphi(g)b\] for some $b \in R$, and $g$ is well-defined if and only if $b \in I^\varphi_H$. The composition of these morphisms is $g\circ f = [G:H]\cdot ab\cdot  \id_{R_\varphi}.$
    \end{remark}
    
    \begin{proposition}\label{prop:directsummandlemma}
        Let $\varphi \in \Hom(G,R^\times)$. The $RG$-module $R_\varphi$ is a $\natural$-permutation module if and only if \[R = \sum_{H\leq G}[G:H](I^\varphi_H)^2.\]
    \end{proposition}
    \begin{proof}
        If $R_\varphi$ is a $\natural$-permutation module, then there exists a permutation module $M = \bigoplus_{i=1}^n R(G/H_i)$ admitting $R_\varphi$ as a direct summand. In particular, there exist maps $f_i\colon R_\varphi \to R(G/H_i)$ and $g_i\colon R(G/H_i) \to R_\varphi$ with $g_i\circ f_i = [G:H_i]\cdot a_ib_i\cdot \id$ such that $\sum_{i=1}^n g_i\circ f_i = \id$. Therefore, $1 = \sum_{i=1}^n [G:H_i]a_ib_i$, and since each $a_i, b_i \in I^\varphi_{H_i}$, it follows that $1 \in \sum_{H \leq G}[G:H](I^\varphi_H)^2$, so $R = \sum_{H\leq G}[G:H](I^\varphi_H)^2$.
        
        Conversely, if $R = \sum_{H\leq G}[G:H](I^\varphi_H)^2$, then $1 \in \sum_{H \leq G}[G:H](I^\varphi_H)^2$, so we have $1 = \sum_{i=1}^n [G:H_i]a_ib_i$. From this, we can construct associated maps
        \[
        f_i\colon R_\varphi \to R(G/H_i) \; \mbox{ and } \; g_i\colon R(G/H_i) \to R_\varphi 
        \]
        satisfying $g_i\circ f_i = [G:H]a_ib_i \cdot \id$. It follows that $\sum_{i=1}^n g_i\circ f_i = \id$, hence $R_\varphi$ is a direct summand of the permutation module $\bigoplus_{i=1}^n R(G/H_i)$, as desired.
    \end{proof}

    \begin{corollary}\label{cor:rk1_for_int_domain}
        Let $R$ be an integral domain. Then given $\varphi \in \Hom(G,R^\times)$, the following are equivalent:
        \begin{enumerate}
            \item $R_\varphi$ is $\natural$-permutation;
            \item $N \coloneqq \ker(\varphi)$ satisfies $[G:N] \in R^\times$;
            \item For all prime $p \in \Phi(G;R)$, and $p$-subgroups $H \leq G$, $\Res^G_H(\varphi)$ is trivial. 
        \end{enumerate}
    \end{corollary}
    \begin{proof}
        We first show (a) if and only if (b). By \Cref{prop:directsummandlemma},  we have that $R_\varphi$ is $\natural$-permutation if and only if $R = \sum_{H\leq G}[G:H](I^\varphi_H)^2$. However, if $H \not\leq N$, then $I^\varphi_H = \{0\}$ since $R$ is an integral domain. On the other hand,  if $H \leq N$, then $I^\varphi_H = R$ trivially. Therefore, $R_\varphi$ is $\natural$-permutation if and only if \[R = \sum_{H \leq N}[G:H]R^2 = [G:N]R^2,\] and the equality holds if and only if $[G:N] \in R^\times$.   

        (b) if and only if (c) is quick. We have $[G:N] \in R^\times$ if and only if $N$ contains all $p$-subgroups of $G$ for every prime $p \in \Phi(G;R)$ if and only if $\Res^G_H(\varphi)$ is trivial for all $p$-subgroups $H \leq G$ with $p\in \Phi(G;R)$. 
    \end{proof}
    
    \begin{example}
        If $|G|$ is invertible in $R$, then all $RG$-modules of rank 1 are $\natural$-permutation modules, since $I^\varphi_1 = R$. Of course, in this case all $R$-projective $RG$-modules are projective, hence $\natural$-permutation, so the point is moot. 
    \end{example}

   \begin{example}\label{ex-R-retract-sylows}
        If $\varphi = 1$, then $I^\varphi_H = R$ for all $H \leq G$. Then $\gcd([G:S] \mid S \in \Syl(G)) = 1$, hence \[R = \sum_{S \in \Syl(G)} [G:S] R^2,\] and we recover the fact that $R$ is a direct summand of \[\bigoplus_{S\in \Syl(G)/G} R(G/S).\] This argument is essentially the proof of \cite[Lemma 3.2]{Car00}.
   \end{example}
        
    For $p$-groups, we can reduce the criteria much further.
    
    \begin{theorem}\label{thm:rk1ppermsforpgrp}
        Let $G$ be a $p$-group and $\varphi \in \Hom(G, R^\times)$. The $R$-free $RG$-module $R_\varphi$ is a $\natural$-permutation module if and only if $R = I^\varphi_G + pR$.
    \end{theorem}
    \begin{proof}
        The forward direction is implied by the previous lemma, as we have an inclusion \[\sum_{H\leq G}[G:H](I^\varphi_H)^2 \subseteq I^\varphi_G + pR.\] Now suppose $R = I^\varphi_G + pR$, then we can write $1 = a + pb$ for $a \in I^\varphi_G$ and $b \in R$. Suppose $|G| = p^n$; if $n = 0$ there is nothing to show so we may assume $n \geq 1$. We have \[1 = (a + pb)^{2n} = \sum_{i=0}^{2n} {2n \choose i} a^i (pb)^{2n-i}.\] 
        The sum of terms indexed from $i = 0$ to $i = n$ is an element of $[G:1]R $, since the minimum power of $p$ attached to each term is $p^n$. Additionally, the sum of terms indexed from $i = n+1$ to $i = 2n$ lives in $(I^\varphi_G)^2$ since the minimum power of $a$ attached to each term is $n \geq 1$. Thus, \[1 = (a + pb)^{2n} \in [G:1]R + (I^\varphi_G)^2 = [G:1](I^\varphi_1)^2 + [G:G](I^\varphi_G)^2,\] so \Cref{prop:directsummandlemma} implies $R_\varphi$ is a $\natural$-permutation module, as desired.
    \end{proof}
    
    \begin{example}\label{ex:p_grp_rk_1_pperm_exs}
        Assume $G$ is a $p$-group. If $R$ has characteristic $p$, then the only $R$-rank 1 $RG$-module which is $\natural$-permutation is the trivial representation $R$. 
        
        If $R$ is an integral domain, then more generally, the only $\natural$-permutation $RG$-module of $R$-rank 1 is $R$ if and only if $pR \neq R$. Indeed, here $I^\varphi_G = 0$. 
    \end{example}

    \begin{remark}
        The above fact need not hold for non-integral domains however. For instance, let $R = \bbZ/6\bbZ$, $G = C_2$, and let $\varphi \in \Hom(C_2, R^\times)$ be the unique nontrivial group homomorphism. Then $I^\varphi_G = 3\bbZ/6\bbZ$. We have $\bbZ/6\bbZ = 3\bbZ/6\bbZ + 2\bbZ/6\bbZ$, so the sign representation $R_\varphi$ is $\natural$-permutation, even though 2 is not invertible in $R$.
    \end{remark}

    We conclude with a classification for endotrivials for $p$-groups.  

    \begin{theorem}\label{thm:picard_grp_p_grps}
        Let $G$ be a $p$-group and $R$ a commutative, connected, Noetherian ring.
        \begin{enumerate}
            \item If $pR \neq R$, then \[\Pic(\catK(G;R))) = \on{CF}^{gl}_b(G) \times \Pic(R) \times \catM(G;R),\] where $\catM(G;R)$ denotes the group of $\varphi\in \Hom(G, R^\times)$ for which \[R = I^\varphi_G + pR.\]
            \item If $pR = R$, then \[\Pic(\catK(G;R)) = \bbZ \times \Pic(R) \times \Hom(G,R^\times).\]
        \end{enumerate}
    \end{theorem}
    \begin{proof}
        We have a split exact sequence \[0 \to \ker(h) \to \Pic(\catK(G;R)) \to \on{CF}_b^{gl}(G;R) \to 0\] by \Cref{thm:gluing_etrivs_p_grps}, and \Cref{thm:ker_h_marks}, \Cref{obs:ker_h_mark_decomp}, and \Cref{thm:rk1ppermsforpgrp} deduce that $\ker(h) = \Pic(R) \times \catM(G;R)$.
    \end{proof}

    As an example, when $R = \bbZ$, we obtain the following.

    \begin{corollary}\label{cor:pgrp_integer_classification}
        Let $G$ be a $p$-group. Then $\Pic(\catK(G;\bbZ)) \cong \on{CF}_b(G)$, and the base change functor $\lambda_p^*\colon \catK(G;\bbZ) \to \catK(G;\bbF_p)$ induces an isomorphism \[\Pic(\catK(G;\bbZ)) \cong \Pic(\catK(G;\bbF_p)).\] In particular, $\catB\colon \on{RO}(G) \to \Pic(\catK(G;\bbZ))$ is surjective. 
    \end{corollary}

    \section{Endotrivial line bundles}\label{sec:line_bundles}

    Let $G$ be a $p$-group and $R$ a commutative Noetherian ring. All invertible $R$-modules are line bundles; that is, they are all locally trivial. Similarly, in \cite{Mil25c}, the second-named author showed that when $R$ is a field $k$ of characteristic $p$, there exists an open cover of $\Spc(\catK(G;k))$ under which every endotrivial is a line bundle. We show here that under certain conditions on our ring $R$, the analogous result holds: one can construct an open cover of $\Spc(\catK(G;R))$ under which invertible $R$-modules and most (but not necessarily all) endotrivials arising from representation spheres are line bundles. 

    \begin{observation}
        Let $G$ be a finite group, $R$ a commutative, connected, Noetherian ring, $V$ a real $n$-dimensional orthogonal representation of $G$, and $x \in \Pic(\catK(G;R))$ the associated endotrivial complex (see \Cref{prop:hmkisdimfunc}). Then $\on{H}_n(x)$ is free of $R$-rank 1, but need not be the trivial $RG$-module. Indeed, any $g \in G$ induces a continuous invertible map on $S(V)$, hence the degree of the map induced by $g$ must be $\pm 1$. The degree, with scalars in $R$, can be read off from the induced map on $R \otimes_\bbZ \on{H}_n(S(V)) \cong \on{H}_n(x)$, hence $g \cdot r = \det(\rho_V(g)) \in \on{H}_n(x)$, where $\rho_V \colon G \to \Aut(V)$ is a matrix representation of $V$.

        In particular, if $G$ has odd order or $R$ has characteristic 2, any element in the image of $\on{\calB}\colon \on{RO}(G) \to \Pic(\catK(G;R))$ has \textit{$G$-trivial homology}, i.e., $\catH(x)$ has trivial $G$-action. On the other hand, this need not happen outside of these cases. Indeed, the 2-term endotrivial $x = \bbZ C_2 \to \bbZ$ associated to the unique nontrivial irrep of $C_2$ satisfies $\catH(x) \cong \bbZ_{sgn}$. It is easy to see from the above discussion however that any (virtual) real representation of the form $V = W\oplus W$ for another (virtual) real representation $W$ has $G$-trivial homology. 
        
        Additionally, it need not be the case that if $\calH(x)$ is $G$-trivial, then $\calH(\Psi^H(x))$ will also be $\Weyl{G}{H}$-trivial for a $p$-subgroup $H$ of $G$ with $p \in \Phi(G;R)$ (if it exists). Indeed, let $G = C_2\times C_2 \coloneqq  \langle \sigma, \tau\rangle$, set $H_1 \coloneqq  \langle\sigma\rangle$ and $H_2 \coloneqq  \langle \tau\rangle$. Let $x\coloneqq  \bbZ(G/H_1) \to \bbZ$, let $x'$ be the sign representation for the subgroup $H_1$, and let $y \coloneqq  x \otimes (x'[0])$. Then $y$ is endotrivial and satisfies $\catH(y) \cong \bbZ$, but $\catH(\Psi^{H_2}(y)) \cong \Psi^{H_2}(x')$, which is non-$\Weyl{G}{H_2}$-trivial. 
    \end{observation}
    
    \begin{definition}\label{def:oriented}
        Let $V$ be a real representation of $G$. We call $V$ \emph{oriented} over $R$ if $\tilde{C}(S(V)^H; R)$ has nonzero homology isomorphic to $R$ for all $H \leq G$ (following \cite{HY14}). More generally, if $V \in \on{RO}(G)$ is a virtual real representation with $V = [V_1] - [V_2]$, we say $V$ is oriented if $\tilde{C}(S(V_1)^H; R) \otimes_R \tilde{C}(S(V_2)^H; R)^\ast $ has nonzero homology isomorphic to $R$ for all $H \leq G$. Note that this definition depends on $R$: if $\ch(R) = 2$, then all real representations are automatically oriented. 
    
        Let $x \in \Pic(\catK(G;R))$ be endotrivial (e.g., $x = \catB(V)$ for $V \in \on{RO}(G)$). We say $x$ is \emph{oriented} if $\catH(\Psi^H_p(x))$ has trivial $\Weyl{G}{H}$-action for all $p \in \Phi(G;R)$ and $p$-subgroups $H \leq G$. If a (virtual) real representation $V$ is oriented, then $\catB(V)$ is oriented. 

    \end{definition}

    \begin{example}
        Any $M \in \Pic(R)$ defines an oriented endotrivial $M[0]$. Given any real representation $V$ of $G$, the endotrivial associated to $W = V\oplus V$ is oriented. Moreover, if $G$ is a $p$-group and $R$ is a principal ideal domain of characteristic $p$, then every endotrivial for $G$ over $R$ is oriented. Indeed, this follows from the classification \Cref{thm:picard_grp_p_grps}, as here, every endotrivial is induced from a virtual representation sphere (and if $p=2$, $1 = -1$).
    \end{example}   

    There are a few instances where we can deduce if a representation is oriented, but in general, this appears to be a nontrivial question. 

    \begin{proposition}
        Suppose $V$ is a real representation. Consider the complexification $\bbC \otimes_R V$: if $\bbC \otimes_R V$ is complex or quaternionic type, then $V$ is oriented. 
    \end{proposition}
    \begin{proof}
        If $\bbC \otimes_R V$ is complex or quaternionic, then $\bbC \otimes_R V = W \oplus \overline{W}$, where $\overline{W}$ indicates complex conjugation (if the type is quaternionic, then $\overline{W} = W$). In this case, we have for any $g \in G$ that \[\det(\rho_V(g)) = \det(\rho_W(g))\det(\rho_{\overline{W}}(g)) = \det(\rho_W(g))\overline{\det(\rho_W(g))}.\] If $W$ is complex, then \[\det(\rho_W(g))\overline{\rho_W(g)} = |\mathrm{det}(\rho_W(g))|^2 = 1,\] and if $W$ is quaternionic, then since $W = \overline{W}$, \[\det(\rho_W(g))\overline{\det(\rho_W(g))} = \det(\rho_W(g))^2 = 1,\] as desired. Repeating the argument for the fixed-point subspaces $V^H$ completes the result. 
    \end{proof}

    On the other hand, if $\bbC \otimes_R V$ has real type, we are unaware of any criterion deducing when $V$ is orientable.
        
    \subsection{Forerunners}

    We present a topological approach to constructing forerunner homomorphisms in this section, which works beyond the setting of fields of prime characteristic. This answers the question posed by the second-named author in \cite{Mil25c}, of a `topological' construction of the forerunner homomorphisms.

    \begin{construction}[Forerunner homomorphisms]\label{cons:forerunners}
        Let $V$ be a real representation for $G$ orientable over $R$, and let $S(V)$ be its associated representation sphere, endowed with the structure of a finite $G$-CW-complex, and let $H \leq G$ be a subgroup. Then as a CW-complex, we have an inclusion inducing a homotopy equivalence $f\colon S^n \to S(V^H)$ with $n$ the dimension of $V^H$, and $S^n$ the $n$-sphere, viewed as a CW-complex with a $0$-cell and a $n$-cell. Then the homomorphism $f\colon S^n \to S(V^H)$ sends the $n$-cell of $S^n$ to a (finite) union of $n$-cells of $S(V^H)$, $\{e_1, \dots, e_l\}$. In fact, the image must contain all $n$-cells of $S(V^H)$, which form a $\Weyl{G}{H}$-set. 

        Now, the $R$-span of the element $e_1 + \cdots + e_l \in \tilde{C}_n(S(V^H)^1; R)$ forms a $R(\Weyl{G}{H})$-submodule isomorphic to $R$, since we assumed $V$ is orientable over $R$. Because $f\colon S^n \to S(V^H)$ is a homotopy equivalence, the submodule belongs to $\ker(d_n^{H,1})$ and in fact generates $\on{H}_n(\tilde{C}(S(V^H)^1; R)) \cong R$. We denote the differentials of $\tilde{C}_n(S(V^?); R)$ by $d_i^?$; note there is an identification $\tilde{C}_n(S(V^H)^1; R) = \tilde{C}_n(S(V^H); R)$, and $d^{H,1}_i=d^H_i$.  
        
        We first note an easy case: if the $R(\Weyl{G}{H})$-submodule \[T \coloneqq  \langle e_1 + \cdots + e_l\rangle \subseteq \ker(d_n^H)\subseteq \tilde{C}_n(S(V^H)^1; R)\] is in fact a $G$-stable submodule (i.e., $\catS \coloneqq  \{e_1,\dots, e_l\}$ is a $G$-set), then we define the \emph{forerunner} of $x \coloneqq  \tilde{C}_n(S(V)^1;R) $ at $H$ to be the homomorphism $\iota^H_x \colon R \to x[-n]$ given by the inclusion in homological degree $n$ \[R \to T \subseteq \tilde{C}_n(S(V^H)^1;R) = \tilde{C}_n(S(V)^H;R) \hookrightarrow \tilde{C}_n(S(V)^1;R) = x.\] The set $\catS$ will be $G$-stable e.g. when $H$ is normal in $G$ or $S(V^H) \simeq S^0$, but not in general. 

        For the more general case in which $\catS = \{e_1,\dots, e_l\}\subseteq S(V)^H$ is not closed under $G$-action, let $\catS'$ denote the $G$-orbit of $\catS$. We claim that all elements in $\catS'\setminus\catS$ are necessarily non-$H$-fixed. Indeed, suppose for contradiction there exists some $e \in \catS'\setminus \catS$ that is $H$-fixed, and let $g \in G$ satisfy ${}^ge_i = e$ for some $i \in I$. Now, note that ${}^g(e_1 + \cdots + e_l)$ is $H$-fixed and satisfies ${}^g(e_1 + \cdots + e_l)\in \ker(d_n^H)$. Moreover, it is distinct from $e_1 + \cdots + e_l$ since ${}^ge_i=e \not\in \catS$. This contradicts the fact that $e_1 + \cdots + e_l$ generates $\ker(d_n^H)$, hence no element of $\catS'\setminus \catS$ is $H$-fixed. 
        
        A similar argument shows the following: if $g \in G$ satisfies ${}^g\catS\neq \catS$, then ${}^g\catS \cap \catS = \varnothing.$ In particular, \[\catS' = \bigsqcup_{g \in [G/H']} {}^g\catS,\] where $H'$ denotes the largest subgroup of $G$ under which $\catS$ is closed under $H'$-action. Indeed, such a largest group exists, since if $\catS$ is closed under the action of two subgroups $H_1, H_2$ of $G$, then $\catS$ is closed under $H_1H_2$-action as well. In particular, ${}^g\catS$ is precisely the set of $n$-cells for $S(V)^{{}^gH}$.
        
        Since we assume $V$ is oriented over $R$, ${}^g(e_1 + \cdots + e_n)$ is a trivial $R({}^g(\Weyl{G}{H}))$-submodule of $\ker(d_n^H)\subseteq \ker(d_n^1)$. Therefore, \[T' \coloneqq  \left\langle \sum_{g \in [G/H']} {}^g(e_1 + \cdots + e_l) \right\rangle=\left\langle\sum_{e \in \catS'} e\right\rangle \subseteq \ker(d_n^1)\] is isomorphic to $R$ as $RG$-module. We define the forerunner $\iota^H_x\colon R \to x[-n] \coloneqq \tilde{C}_n(S(V)^1;R)$ via the inclusion in homological degree $n$ \[R \to T' \subseteq \tilde{C}_n(S(V^H)^1;R) = \tilde{C}_n(S(V)^H;R) \hookrightarrow \tilde{C}_n(S(V)^1;R) =: x.\]

        We remark that the forerunner $\iota^H_x$ can be expressed as a trace sum. With $T \coloneqq  \langle e_1 + \cdots + e_l\rangle$ as before, we have the homomorphism of $R(\Weyl{G}{H})$-modules \[\iota'\colon R \to T \hookrightarrow \tilde{C}_n(S(V)^H;R) \hookrightarrow \tilde{C}_n(S(V)^1;R).\] The fact that $\catS' = \bigsqcup_{g \in [G/H']}{}^g\catS$ implies that \[\iota^H_x = \tr^G_{H'} \iota' \coloneqq  \sum_{g \in [G/H']} {}^gf.\] 
    \end{construction} 

    \begin{remark}
        Note that forerunners are constructed given a cellularization of a representation sphere $S(V)$, not just an endotrivial. The forerunners can be thought of as `singling out' certain $n$-cells, and their construction relies on the fact that some \emph{effective} endotrivials, those whose h-mark functions are monotonically decreasing, arise up to homotopy from chain complexes of $R\calO(G)$-modules. However, they cannot be extended to homomorphisms of chain complexes of $R\calO(G)$-modules in general. Indeed, in degree $n$, one has a $RG$-module homomorphism of the form $R \to RX$ for some $G$-set $X$, and this in general will not satisfy the `stabilizers grow' property that homomorphisms of free $R\calO(G)$-modules satisfy. 

        If $G$ is a $p$-group and $R$ a field of characteristic $p$, it was deduced in \cite{Mil25c} that after removing contractible summands of an effective endotrivial $x$, the highest nonzero homological degree $x_n$ of $x$ consists of a single transitive permutation module. We do not know if an analogous topological statement holds for a $n$-dimensional representation sphere $S(V)$: does there always exist a cellular structure for $S(V)$ such that the $n$-cells form one transitive $G$-set? 
    \end{remark}
        
    We observe the following properties the forerunners satisfy; showing these takes considerably less work than in \cite{Mil25c}.

    \begin{proposition}\label{prop:forerunner_props}
        Let $V$ be a real representation of $G$ orientable over $R$. Set $x \coloneqq  \tilde{C}(S(V)^1; R)$ $p \in \Phi(G;R)$ a prime, and $H$ a $p$-subgroup of $G$. The following hold:

        \begin{enumerate}
            \item $\Psi^H_p(\iota^H_x)$ is a quasi-isomorphism;
            \item If $K =_G H$, then $\iota^H_x = \iota^K_x$. 
            \item If $K\neq_G H$ and all minimal subgroups $L$ with respect to the property that $K,H \leq_G L$ satisfy $h_x(H), h_x(K) > h_x(L)$, then $\Psi^K(\iota_x^H) = 0$ and $\Psi^H(\iota_x^K) = 0$. In particular, $\iota^K_x$ and $\iota^H_x$ differ. 
        \end{enumerate}
    \end{proposition}

    \begin{proof}
        We use the notation of \Cref{cons:forerunners}. For (a), recall that we have an identification $\Psi_p^H(\tilde{C}(S(V^1);R)) \cong \tilde{C}(S(V^H);R/p)$. Since all elements of $\catS'\setminus \catS$ are non-$G$-fixed, the image of $\iota^H_x$ identifies with the submodule $T = \langle e_1 + \cdots + e_l\rangle$, which generates $\ker(d_n^H)$. For (b), one has from the above construction that $\catS = \{e_1, \dots, e_l\}$ consists of the $n$-simplices whose sum generates $\ker(d_n^H)\subseteq \tilde{C}_n(S(V)^H;R)$, and if ${}^gH = K$, then ${}^g\catS$ consists of the $n$-simplices whose sum generates $\ker(d_n^K) \subseteq \tilde{C}_n(S(V)^H; R)$. Since the $G$-orbits of $\catS$ and ${}^g\catS$ agree, it follows that the maps $\iota^H_x$ and $\iota^K_x$ coincide. 

        For (c), we let $\catS_K, \catS_{K}'$ and $\catS_H,\catS_{H}'$ denote the sets $\catS, \catS'$ in \Cref{cons:forerunners} for $K$ and $H$ respectively. The elements of $\catS_K$ all have stabilizer containing $K$, and the elements of $\catS_H$ all have stabilizer containing $H$. Similarly, the elements of $\catS_K'$ all have stabilizer containing $K$ up to conjugacy, and same for $\catS_H'$. The only way for $\Psi^H(\iota^K_x)$ to be nonzero is if an element of $\catS_K'$ has stabilizer also containing $H$ up to conjugacy; otherwise, the image of $\iota_x^K$ is killed by $\Psi^H$. However, $h_x(L) < h_x(K)$ equivalently states $\dim_\bbR V^K > \dim_\bbR V^L$; in particular, any cell with stabilizer $L$ can be at most a $h_x(L)$-cell. Note that we use the fact that applying 1-point compactification commutes with taking fixed points. Since the elements of $\catS_K'$ are all $h_x(K)$-cells, none of them can be stabilized by $L$. Since this holds for any subgroup containing both $K, H$ up to conjugacy, we conclude none of the elements in $\catS_K'$ are $H$-fixed, and thus $\Psi^H(\iota^K_x) = 0$. Swapping the roles of $H,K$ yields the same for $\Psi^K(\iota^H_x)$. 
    \end{proof}

    The forerunners allow us to deduce that oriented endotrivials are `tt-line bundles'.

    \begin{definition}
        Let $\catK$ be a tt-category and $x\in \catK$. We say $x$ is a \emph{line bundle} if there exists an open cover $\{U_i\}_{i=1}^n$ of $\Spc(\catK)$ such that $x \cong \bbone[j_i] \in \catK(U_i)$ for each $i \in \{1, \dots, n\}$. Say that the open cover $\{U_i\}_{i=1}^n$ \emph{verifies} that $x$ is a line bundle. 
    \end{definition}
    
    \begin{lemma}\label{lem:single_open}
        Let $\catK$ be a tt-category and suppose $x_1, \dots, x_n \in\catK$ are line bundles. Then $x_1^{\otimes a_1} \otimes \cdots \otimes x_n^{\otimes a_n}$ is a line bundle for all $a_1, \dots, a_n \in \bbN$, and there is a single open cover verifying each possible combination is a line bundle. 
    \end{lemma}
    \begin{proof}
        Suppose $x_i$ has open cover $\{U^i_j\}_{j=1}^{m_i}$ which verifies $x_i$ is a line bundle. Then taking combinations of intersections, that is, all opens of the form $U^1_{j_1} \cap U^2_{j_2} \cap \cdots \cap U^n_{j_n}$ for $j_i \in \{1,\dots, m_i\}$ gives an open cover of $\Spc(\catK)$, and it is a routine exercise to show that every $x_i$ is isomorphic to a shift of the tensor unit in each open. Then, it is easy to see that if two elements are line bundles under the same open, their tensor product is as well, whence the result. 
    \end{proof}

    \begin{lemma}\label{lem:remains_oriented}
        Let $x \in \Pic(\catK(G;R))$ be an oriented endotrivial, and let $f \in R$. Then $R_f \otimes_R x \in \Pic(\catK(G;R_f))$ is oriented as well. 
    \end{lemma}
    \begin{proof}
        We check that  $\catH(\Psi^H_p(R_f \otimes_R x))$ has trivial $G$-action for every $p \in \Phi(G;R_f)$ and $H$ a $p$-subgroup. Indeed,  by \Cref{prop:modularfixedptscommutes}, along with \Cref{lem:commute_with_homology}, we obtain \[\catH( \Psi^H_p({R_f} \otimes_{R} x)) \cong \catH(R_f/p \otimes_{R/p} \Psi^H_p(x))\cong R_f/p\otimes_{R_p}\catH(\Psi^H_p(x)),\] 
        and since the $G$-action on $\catH(\Psi^H_p(x))$ is trivial, it remains trivial after base change, as desired. 
    \end{proof}
    
    \begin{theorem}\label{thm:linebundle}
        Let $G$ be a $p$-group and $R$ a commutative Noetherian ring. Then every oriented endotrivial $x \in \Pic(\catK(G;R))$ is a line bundle. If $\Pic(R)$ is finitely generated, then moreover, there exists an open cover of $\Spc(\catK(G;R))$ that verifies every such endotrivial is a line bundle. 
    \end{theorem}
    \begin{proof}
        First, we claim that any endotrivial of the form $M[i] \in \Pic(\catK(G;R))$ with $M \in \Pic(R)$ and $i \in \bbZ$ is a line bundle. Indeed, for any line bundle $M \in \Pic(R)$, there exists an open cover $D(f_1),\dots, D(f_n)$ of $\Spec(R)$ such that $R_{f_i} \otimes_R M \cong R_{f_i}$ as $R_{f_i}$-modules. We consider the preimage $\open(f_i) \coloneqq  \on{comp}\inv(D(f_i))$; then \Cref{prop:localization_identification} gives us that $\catK(G;R)(\open(f_i)) \cong \catK(G;R_f)$, and the image of $M[i]$ in $\catK(G;R_f)$ is $R_f \otimes_R M[i] \cong R_f[i]$. Thus $M[i]$ is a line bundle via the open $\{\open(f_i)\}_{i=1}^n$. 
       
        Next, we consider an arbitrary oriented endotrivial $x \in \Pic(\catK(G;R))$ under the assumption that $R$ is $p$-connected. If $R$ is not connected, there is a connected component containing $R/p$, and in all other components $p$ is invertible. Let $R = R_1 \times \cdots \times R_n$, then we have a decomposition of Picard groups $\Pic(\catK(G;R)) \cong \Pic(\catK(G;R_1)) \times \cdots \times \Pic(\catK(G;R_n))$, however for the components where $p$ is invertible, all orientable endotrivials are already line bundles, so it suffices to consider the connected component containing $\Spec(R/p)$. Thus, we may assume $R$ is connected as well.
        
        Now, in this case by \Cref{thm:picard_grp_p_grps}, we have that $x = M[0] \otimes y$ where $\catH(x) = M \in \Pic(R)$ and $y$ is in the image of $\catB$; in this case $\catH(y) = R$. Let $V \in \on{RO}(G)$ satisfy $\catB(V) = y$ and be written as $V = [V_1] - [V_2]$ for real representations $V_1, V_2$; since $V$ is oriented over $R$, we can choose $V_1, V_2$ such that they are oriented over $R$ as well. We have $\catB(V_1)\otimes_R \catB(V_2)^\ast  = y$, so it suffices to show $\catB(V_1), \catB(V_2)$ are line bundles. Consequently, we have reduced to the case of $y = \catB(V)$ for a real representation $V$ oriented over $R$. 

        Since $V$ is oriented, the forerunner homomorphisms $\iota_{y}^H\colon R \to y[-h_y(H)]$ are well-defined. We define the following open sets: 
        \[U_{y}(H) \coloneqq \open(\iota_{y}^H) = \supp(\cone(\iota_{y}^H))^c.\] 
        In other words, $U_y(H)$ consists of the open locus of $\Spc(\catK(G;R))$ where $\iota_y^H$ is invertible. Consequently, in $\Spc(\catK(G;R))(U_y(H))$, $y \cong R[h_y(H)]$. Therefore, it suffices to show that $\{U_y(H)\}_{H \in \on{Sub}_p(G)/G}$ is an open cover of $\Spc(\catK(G;R))$. However, this follows from \Cref{prop:forerunner_props}; since $\Psi^H_p(\iota^H_y)$ is a quasi-isomorphism, it follows that the generic locus is contained in $U_y(H)$ (see \Cref{lem:intersections_in_generic_locus}). On the other hand, 
        the minimal prime $\catP(H, 0,\frakp) \in \Spc(\catK(G;R))$, with $\frakp \in \Spec(R)$ satisfying $\ch(k(\frakp)) = p$, satisfies $\catP(H, 0,\frakp) \in U_y(H)$. Indeed, since $\catP(H, 0,\frakp) = \ker(\check\Psi^H \circ \lambda_\frakp^\ast )$, $\cone(\iota^H_y)\in \catP(H, 0,\frakp)$, thus $\catP(H, 0,\frakp) \in U_y(H)$. Thus, the set $\{U_y(H)\}_{H \in \on{Sub}_p(G)/G}$ verifies $y$, hence $x$, is a line bundle. 

        Finally, in the case where $R$ is not necessarily $p$-connected, we use the open cover of $\Spc(\catK(G;R))$ constructed in \Cref{cons:open_cover_of_spc}, $\{\open(a_i)\}_{i=0}^n$. Given an oriented endotrivial $x \in \Pic(\catK(G;R))$, it suffices to construct an open cover in each open, i.e., an open cover verifying $x \in \catK(G;R)(\open(a_i))\cong \catK(G;R_{a_i})$ is a line bundle, where the equivalence holds by \Cref{prop:localization_identification}. Note $x$ remains oriented by \Cref{lem:remains_oriented}. We may construct an open cover in each open, since $R_{a_i}$ is either $p$-connected or satisfies that $p$ is invertible. Taking the collection of open covers on each open $\open(a_i)$ determines an open verifying that $x$ is a line bundle, and we conclude that if $R$ is a commutative Noetherian ring, every oriented endotrivial $x$ is a line bundle.

        For the final statement, by \Cref{lem:single_open}, it suffices to show that the subgroup of $\Pic(\catK(G;R))$ consisting of oriented endotrivials is finitely generated. This follows from the classification \Cref{thm:picard_grp_p_grps}: one may easily check that the subgroup of oriented endotrivials is isomorphic to $\Pic(R) \times H$ for some subgroup $H \leq \on{CF}_b^{gl}(G)$, which is finitely generated free. The result follows. 
    \end{proof}

    \begin{remark}
        These forerunners and the line bundle property are the first important steps for defining a `twisted cohomology' ring for $\catK(G;R)$ as in \cite[Part II]{BG25}, \cite[Section 8]{DG25}, and \cite{Mil25c}. The groundwork has been laid, but we postpone developing that theory for another day. 
    \end{remark}

    \section{A reduction to $p$-groups} \label{sec:infty_cat_stuff}
    
    Let $G$ be a finite group, and let $R$ be a commutative ring. We write $\Phi(G)$ for the set of primes dividing the order of $G$. In this section, we explain how the study of the Picard group of $\catK(G;R)$ can be reduced to the case of $p$-groups together with certain $q$-local information associated to $G$, for each $q \in \Phi(G)$. To do so, we make use of the language of $\infty$-categories, and we refer the reader to \cite{Ant16}, \cite{Lur17}, or \cite{Jas26} for the necessary background. 
    
    In particular, throughout this section we adopt the convention that a `tensor-triangulated category' means a symmetric monoidal stable $\infty$-category. The advantage of this framework is that it provides access to tools that are unavailable when working solely with triangulated categories, although the final results ultimately depend only on invariants of the underlying $1$-categories.
    
    We also note that some of the results presented here may be recovered from \cite{GP26}, or with some additional work from \cite{MNN17}; however, the arguments given here are more direct in the specific setting relevant to our purposes.
    
    \begin{definition}
        Let $G$ be a group and $R$ a commutative ring. We regard the homotopy category of chain complexes of $RG$-modules, equipped with the monoidal structure given by the tensor product over $R$, as an $\infty$-category via the dg-nerve construction. Through this construction, we view the big derived category of permutation modules $\catD(G;R)$ as an $\infty$-category. By abuse of notation, we continue to denote it by $\catD(G;R)$, since the intended meaning will always be clear from the context.
    \end{definition}
    
    \begin{remark}
        The category $\catD(G;R)$ is a rigidly compactly generated tt-category, and its homotopy category agrees with the big derived category of permutation modules introduced in \Cref{sec-perm}.
    \end{remark}

    \begin{recollection}\label{rec-res-ind-monadic}
        Let $H \leq G$ be a subgroup. Recall that there is an adjunction given by induction and restriction:
        \[
        \mathrm{Res}_H^G \colon \catD(G;R)\rightleftarrows \catD(H;R) \colon  \mathrm{Coind}_H^G\simeq \mathrm{Ind}_H^G.
        \]
        This adjunction satisfies the projection formula, the functor $\mathrm{Coind}_H^G$ preserves both limits and colimits, and it is conservative. Hence, by \cite[Proposition 5.29]{MNN17}, there is an equivalence
        \[
        \mathsf{Mod}_{\catD(G;R)}\!\bigl(\mathrm{Coind}_H^G\mathrm{Res}_H^G(\bbone)\bigr)\simeq \catD(H;R).
        \]
        Under this equivalence, extension of scalars corresponds to restriction, while restriction of scalars corresponds to induction. Moreover, the underlying object of the commutative algebra
        \[
        A_H\coloneqq \mathrm{Coind}_H^G\mathrm{Res}_H^G(\bbone)
        \]
        may be identified with $R(G/H)[0]$.
    \end{recollection}


    \begin{remark}
        Since the functor $\mathrm{Coind}_H^G$ preserves compact objects, and hence dualizable objects, as our categories are rigidly-compactly generated, we conclude that the commutative algebra $A_H$ is dualizable.
    \end{remark}
    
    \begin{remark}
        The previous observation may also be deduced from the results of \cite{Bal17}. Indeed, one can show that the restriction/induction adjunction
        \[
        \mathrm{Res}_H^G \dashv \mathrm{Coind}_H^G
        \]
        on the big derived categories of permutation modules is separable, and hence monadic. Moreover, the associated monad can be identified with tensoring by the separable algebra $R(G/H)[0]$.
    \end{remark}

    \begin{notation}
        Let $\on{Sub}_\Phi(G)$ denote the family of subgroups of $G$ of $p$-power order, where the prime $p$ ranges over all elements of $\Phi(G)$.
    \end{notation}

    \begin{proposition}\label{prop-Aphi-descent}
        Let $G$ be a finite group and let $R$ be a commutative ring. Then the commutative algebra
        \[
        A_\Phi \coloneqq \prod_{E\in \on{Sub}_\Phi(G)} A_E
        \]
        satisfies descent; that is, the smallest thick tensor ideal in $\catD(G;R)$ containing $A_\Phi$ is all of $\catD(G;R)$. Equivalently,
        \[
        \mathrm{thick}_\otimes(A_\Phi)=\catD(G;R).
        \]
    \end{proposition}

    \begin{proof}
        This is a direct consequence of \Cref{ex-R-retract-sylows}. Indeed, the set of Sylow subgroups $\mathrm{Syl}(G)$ is contained in $\on{Sub}_\Phi(G)$. Hence $R$ is a retract of
        \[
        \bigoplus_{S\in \mathrm{Syl}(G)} R(G/S),
        \]
        and therefore also a retract of
        \[
        \bigoplus_{H\in \on{Sub}_\Phi(G)} R(G/H).
        \]
        Equivalently, $R[0]$ belongs to the thick subcategory of $\catD(G;R)$ generated by $A_\Phi$. Since $R[0]$ is the tensor unit, the claim follows.
    \end{proof}

    \begin{construction}\label{cons-functor-Dperm}
        Let $G$ be a finite group, and let $\mathcal{O}(G)$ denote its orbit category (see \Cref{def:orbit_cat}). Then there is a functor
        \[
        \catD(-;R)\colon \mathcal{O}(G)\to \mathsf{CAlg}(\mathrm{Pr}^L_{\mathrm{st}}),
        \qquad
        G/H\mapsto \mathsf{Mod}_{\catD(G;R)}(A_H)\simeq \catD(H;R),
        \]
        where $\mathsf{CAlg}(\mathrm{Pr}^L_{\mathrm{st}})$ denotes the $\infty$-category of presentable symmetric monoidal stable $\infty$-categories whose tensor product preserves colimits separately in each variable, together with symmetric monoidal left adjoint functors. See \cite[Section 2]{Mat16} for further details.
    \end{construction}

     \begin{theorem}\label{prop-reduction-pgroups}
        Let $G$ be a finite group and let $R$ be a commutative ring. Let
        \[
            \mathcal{O}_\Phi(G)\subseteq \mathcal{O}(G)
        \]
        denote the full subcategory consisting of those orbits whose isotropy groups lie in $\on{Sub}_\Phi(G)$. Then evaluation at $G/G$ of the functor from \Cref{cons-functor-Dperm}, restricted to $\mathcal{O}_\Phi(G)$, induces an equivalence of symmetric monoidal categories
        \[
            \catD(G;R)\xrightarrow{\simeq}
            \lim_{G/H\in \mathcal{O}_{\Phi}(G)^\mathrm{op}} \catD(H;R).
        \]
    \end{theorem}

    Here, the arrows in the limit diagram are given by restriction. The main ingredient is the following result, due to Mathew, which applies to any presentable symmetric monoidal stable $\infty$-category whose tensor product preserves colimits separately in each variable. In particular, it applies to $\catD(G;R)$. 

    \begin{proposition}{\cite[Proposition 3.22]{Mat16}}\label{prop-descent-mathew}
        Let $\calC$ be an object of $\mathsf{CAlg}(\mathrm{Pr}^L_{\mathrm{st}})$, and let $A$ be a commutative algebra object in $\calC$ that admits descent. Then the adjunction given by tensoring with $A$ and the forgetful functor is comonadic. In particular, the natural functor to the totalization
        \[
        \calC \to \mathrm{Tot}\left(
        \mathsf{Mod}_\calC(A)
        \doublerightarrow{}{}
        \mathsf{Mod}_\calC(A\otimes A)
        \triplerightarrow{}{}
        \cdots
        \right)
        \]
        is an equivalence.
    \end{proposition}

    \begin{proof}[Proof of \Cref{prop-reduction-pgroups}]
        The argument is a standard cofinality argument. Applying \Cref{prop-descent-mathew} to the descendable algebra $A_\Phi$ of \Cref{prop-Aphi-descent}, we obtain an equivalence
        \begin{equation}\label{eq-tot}
            \catD(G;R)\xrightarrow{\simeq}
            \mathrm{Tot}\left(
            \mathsf{Mod}_{\catD(G;R)}(A_\Phi)
            \doublerightarrow{}{}
            \mathsf{Mod}_{\catD(G;R)}(A_\Phi\otimes A_\Phi)
            \triplerightarrow{}{}
            \cdots
            \right).
        \end{equation}
        Let $\mathsf{Set}_{G,\Phi}$ denote the category obtained from $\mathcal{O}_\Phi(G)$ by freely adjoining finite coproducts. Then the functor of \Cref{cons-functor-Dperm} extends uniquely to a finite-coproduct-preserving functor
        \[
        \tilde{\catD}(-;R)\colon \mathsf{Set}_{G,\Phi}\to \mathsf{CAlg}(\mathrm{Pr}^L_{\mathrm{st}}).
        \]
        By \cite[Lemma 6.30]{MNN17}, there is an equivalence of $\infty$-categories
        \[
        \lim_{\mathsf{Set}_{G,\Phi}^\mathrm{op}} \tilde{\catD}(-;R)
        \simeq
        \lim_{\mathcal{O}_{\Phi}(G)^\mathrm{op}} \catD(-;R).
        \]
        Now let
        \[
        X=\coprod_{H\in \on{Sub}_\Phi(G)} G/H \in \mathsf{Set}_{G,\Phi}.
        \]
        Every object $Y$ of $\mathsf{Set}_{G,\Phi}$ admits a morphism $Y\to X$. Therefore, \cite[Proposition 6.28]{MNN17} implies that the functor
        \[
        X^{\bullet+1}\colon \Delta^\mathrm{op}\to \mathsf{Set}_{G,\Phi}, \; [n]\mapsto X^n
        \]
        is cofinal, where $\Delta$ denotes the simplex category. Consequently, we obtain an equivalence
        \[
        \lim_{\mathsf{Set}_{G,\Phi}^\mathrm{op}} \tilde{\catD}(-;R)
        \simeq
        \lim_{\Delta} \tilde{\catD}(-;R)\circ X^{\bullet+1}.
        \]
        The right-hand side identifies with the totalization appearing in \Cref{eq-tot}, and the result follows.
    \end{proof}

    It is known that the Picard spectrum functor commutes with limits; see \cite[Proposition 2.2.3]{MS16}. We therefore obtain the following corollary.
    
    \begin{corollary}\label{coro-lim-picard}
        Let $\mathrm{pic}(\catD(G;R))$ denote the connective Picard spectrum of $\catD(G;R)$. Then there is an equivalence of spectra
        \[
        \mathrm{pic}(\catD(G;R))
        \xrightarrow{\simeq}
        \lim_{G/H\in \mathcal{O}_{\Phi}(G)^\mathrm{op}}
        \mathrm{pic}(\catD(H;R)).
        \]
    \end{corollary}


   Let us recall several general facts about the Picard spectrum. First, the homotopy groups of the Picard spectrum $\mathrm{pic}(\calC)$ of a symmetric monoidal stable $\infty$-category $\calC$ are given by
    \[
        \pi_n (\mathrm{pic}(\calC))=
        \left\{
        \begin{array}{ll}
        \mathrm{Pic}(\calC) & \text{if } n=0,\\[4pt]
        \pi_0 \bigl(\mathrm{Hom}_\calC(\bbone,\bbone)\bigr)^\times & \text{if } n=1,\\[4pt]
        \pi_{n-1}(\mathrm{Hom}_\calC(\bbone,\bbone)) & \text{if } n\geq 2,
        \end{array}
        \right.
    \]
    where $\mathrm{Pic}(\calC)$ denotes the Picard group of the homotopy category $\mathrm{ho}(\calC)$; see \cite[Section 2.2]{MS16}.
    
    On the other hand, suppose that a symmetric monoidal stable $\infty$-category $\calC$ is expressed as the limit of a diagram
    \[
        D\colon I^\mathrm{op}\to \mathrm{Cat}^\otimes_{\mathrm{st}}.
    \]
    Then there is an associated spectral sequence
    \begin{equation}\label{spectral sequence for Pic}
        E_2^{p,q} = \on{H}^p\big(I;\pi_q(\mathrm{pic}(D(-)))\big) \Rightarrow
        \pi_{q-p}\!\left(\varprojlim \mathrm{pic}(D)\right)
    \end{equation}
    with differentials of the form $d_r\colon E_r^{p,q}\to E_r^{p+r,q+r-1}$. In particular, this spectral sequence computes the Picard group of $\mathrm{ho}(\calC)$.
    
    We now specialize to the case of $\catD(G;R)$. See e.g. \cite[Section 2.5]{Gro23} for details on homotopy groups of small categories. For this next proposition, the authors thank Isaac Moselle for noticing a serious flaw with a previous version of the proof, and for providing a sketch of part of the following argument. 

    \begin{proposition}\label{prop:pi1_trivial}
        The orbit category $\calO_\Phi(G)$ is contractible. In particular, for any abelian group $A$
        \[
        \on{H}^i(\calO_\Phi(G);A)=0 
        \] 
        for any $i>0$. 
    \end{proposition}
    \begin{proof}
       First, we show that $\calO_\Phi(G)$ is simply connected.
        We have that $|\calO_\Phi(G))| \simeq |E\calO_\Phi|/G$ (\cite[Proposition 2.10]{Gro02}), where $E\calO_\Phi$ denotes the category of pointed $G$-sets with isotropy in the family $\Phi$, see also \cite[A.1]{Gro23}. Note $|E\calO_\Phi|$ is a classifying space for $\Phi$ in the sense of \cite[Section 7]{DL98} - in particular, $|E\calO_\Phi|^H$ is contractible if $H \in \Phi$ and $|E\calO_\Phi|^H = \varnothing$ otherwise. Since $1 \in \Phi$, in particular $|E\calO_\Phi|^1 = |E\calO_\Phi|$ is contractible, hence simply connected. 

        Therefore, we may apply the main result of \cite{Arm68}, which gives us an isomorphism  $\pi_1(|E\calO_\Phi|/G) \cong G/N$, where $N$ is the normal subgroup of $G$ generated by elements which have fixed points. Note that the conditions of the theorem are satisfied since $E\calO_\Phi$ has the structure of a finite $G$-CW-complex. Every element $g$ of prime power order in $G$ has a fixed point in $|E\calO_\Phi|$, as $E\calO_\Phi^{\langle g\rangle}$ is nonempty. Since every element of $G$ can be written (uniquely) as a product of commuting elements of prime power order, it follows that $N = G$. Thus $\pi_1(|E\calO_\Phi|/G)$ is trivial, therefore $\pi_1(\calO_\Phi(G))$ is as well.

        Now, we claim that $\on{H}^i(\calO_\Phi(G);\mathbb Z)=0$ for all $i>0$. Since $\calO_\Phi(G)$ is a finite category, it's enough to verify that $\on{H}^i(\calO_\Phi(G);\mathbb Z)\otimes \mathbb Z_{(p)}\cong \on{H}^i(\calO_\Phi(G);\mathbb Z_{(p)})=0$ for all primes $p$ dividing the order of $G$. Fix a prime $p$, and a $p$-Sylow subgroup $S$ of $G$.  Recall that the Bredon cohomology $\on{H}_G^\ast(E\calO_\Phi;\mathbb{Z}_{(p)})$ agrees with $\on{H}^\ast(\calO_\Phi(G);\mathbb Z_{(p)})$; see for instance \cite[Section 3.1]{MNN19}. In particular, since the $H$-fixed points of $E\calO_\Phi$ are contractible for any $H\leq S$, we obtain that $\on{H}_S^\ast(E\calO_\Phi;\mathbb{Z}_{(p)})$ is trivial in positive degrees. Now, we may use a transfer argument as in \cite[Section 6]{Dwy96} to deduce that $\on{H}_G^\ast(E\calO_\Phi;\mathbb{Z}_{(p)})$ is trivial in positive degrees. Indeed, this follows since  the latter embeds in the former (cf. \cite[Theorem 6.4]{Dwy96}). 

        Using that $|\calO_\Phi(G)|$ has the structure of a CW-complex, trivial integral cohomology together with simply connectedness implies that the space is contractible. 
    \end{proof}

    As a consequence, we complete the reduction: somewhat miraculously, the limit of Picard groupoids descends to the group level. 
    
    \begin{corollary}\label{cor:limit_of_pic_groups}
        The equivalence of \Cref{coro-lim-picard} descends to an isomorphism of abelian groups \[
            \mathrm{Pic}(\catD(G;R))
            \cong
            \lim_{G/H\in \mathcal{O}_{\Phi}(G)^\mathrm{op}}
            \mathrm{Pic}(\catD(H;R)).
        \]
    \end{corollary}

    \begin{proof}
        Applying \Cref{spectral sequence for Pic} to the equivalence of \Cref{coro-lim-picard}, we obtain a spectral sequence
        \[
            E_2^{p,q} = \on{H}^p(\mathcal{O}_{\Phi}(G);\pi_q(\mathcal{F}))\Rightarrow\pi_{q-p}\!\left(\varprojlim_{\mathcal{O}_{\Phi}(G)^\mathrm{op}}\mathrm{pic}(\catD(H;R))
            \right) \cong\pi_{q-p}\left( \mathrm{pic}(\catD(G;R))\right),
        \]
        where $\mathcal{F}$ denotes the presheaf
        \[
            \mathcal{O}_{\Phi}(G)^\mathrm{op}\to\mathsf{Sp},
            \; G/H\mapsto \mathrm{pic}(\catD(H;R)).
        \]
        Now observe that
        \[
            \pi_{n-1}\big(\mathrm{Hom}_{\catD(G;R)}(\bbone,\bbone)\big)=0
            \; \text{for all } n\geq 2,
        \]
        while
        \[
            \pi_0\bigl(\mathrm{Hom}_{\catD(G;R)}(\bbone,\bbone)\bigr)^\times\cong R^\times.
        \]
        Consequently, the spectral sequence is concentrated in rows $q=0,1$. Therefore, we obtain the following short exact sequence of abelian groups, \[
        0 \to \on{H}^1(\mathcal{O}_{\Phi}(G);R^\times)\to    \mathrm{Pic}(\catD(G;R))
        \to \ker(d^{0,0}_2) \to 0.
        \]
        Now, by \Cref{prop:pi1_trivial} we have $E^{p,1}_2=0$ for all $p>0$, hence both  
        $\on{H}^1(\mathcal{O}_{\Phi}(G);R^\times)$ and $d^{0,0}_r$
        are trivial. Thus we have an isomorphism 
        \[
            \mathrm{Pic}(\catD(G;R)) \cong \on{H}^0(\mathcal{O}_{\Phi}(G);\pi_0(\mathcal{F})),
        \]
        and the right-hand term identifies with  $ \lim_{G/H\in \mathcal{O}_{\Phi}(G)^\mathrm{op}} \mathrm{Pic}(\catD(H;R))$, as desired. 
    \end{proof}

    \section{The Picard group of $\catK(G;R)$ for a finite group}\label{sec:final}
    
    We spell out the consequences of \Cref{cor:limit_of_pic_groups}. Explicitly, to specify an endotrivial $x \in \catK(G;R)$, it suffices to specify a restriction-coherent family $(x_H)$ of endotrivials for each $\catK(H;R)$. 

    \begin{lemma}\label{lem:res_via_h_mks}
        Let $R$ be  commutative connected ring, and  $G$ be a $p$-group and $H \leq G$, and suppose $x \in \Pic(\catK(H;R))$ satisfies that $H$ acts trivially on $\catH(x)$. Then \[x \in \im\big(\Res^G_H\colon\Pic(\catK(G;R)) \to  \Pic(\catK(H;R))\big)\] if and only if for all $\frakp \in \Spec(R)$, $h_\frakp\in \on{CF}_b(H,\ch(k(\frakp)))$ satisfies \[h_{\frakp} \in \mathrm{im}\big(\Res^G_H\colon \on{CF}_b(G) \to \on{CF}_b(H)\big).\] 
    \end{lemma}
    \begin{proof}
        Straightforward.
    \end{proof}

    \begin{remark}\label{rmk:h_mks_vs_limit}
        Consider an element 
        \[x = (x_H)_{G/H \in \calO_\Phi(G)\op} \in
            \lim_{G/H\in \mathcal{O}_{\Phi}(G)^\mathrm{op}}
            \mathrm{Pic}(\catD(H;R)).\] 
        We may take fiberwise h-marks in each component. Fix $\frakp \in \Spec(R)$ and let $p = \ch(k(\frakp))$, then we obtain a homomorphism 
        \[h_{\frakp}\colon \lim_{G/H\in \mathcal{O}_{\Phi}(G)^\mathrm{op}}
            \mathrm{Pic}(\catK(H;R)) \to \lim_{G/H\in \mathcal{O}_{\Phi}(G)^\mathrm{op}}\on{CF}_b(H,p),\] and it is an easy verification that the group homomorphism induced by restriction of superclass functions,\[ \on{CF}_b(G,p) \xrightarrow{\Res^G_H}\lim_{G/H\in \mathcal{O}_{\Phi}(G)^\mathrm{op}}\on{CF}_b(H,p)\]
        is in fact an isomorphism. Explicitly, the isomorphism is induced by restriction.  Moreover, the following diagram commutes.
        \begin{figure}[H]
            \centering
            \begin{tikzcd}
                \Pic(\catK(G;R)) \ar[r] \ar[d, "h_{\frakp}"]& \lim_{G/H\in \mathcal{O}_{\Phi}(G)\op}\mathrm{Pic}(\catK(H;R))  \ar[d, "h_{\frakp}"] \\
                \on{CF}_b(G;p) \ar[r] & \lim_{G/H\in \mathcal{O}_{\Phi}(G)\op}\on{CF}_b(H,p)
            \end{tikzcd}
        \end{figure}
        
    \end{remark}

    \begin{theorem}\label{thm:img_of_hmk_all_grps}
        Let $G$ be a finite group of order $n$ and $R$ a commutative connected Noetherian ring. Recall the reduced h-mark homomorphism from \Cref{rmk:refined_h_mk_func},
        \[\Pic(\catK(G;R)) \to \prod_{\frakp \in \tilde \pi_0(\Phi(G;R))} \on{CF}_b(G, \ch(k(\frakp))).\] The image is precisely the tuples of Borel-Smith functions which agree on the trivial subgroup of $G$. 
    \end{theorem}
    \begin{proof}
        As in the proof of \Cref{thm:gluing_etrivs_p_grps}, given some it suffices to specify an endotrivial $x \in \Pic(\catK(G;R))$ for which, given a $\frakp \in \tilde \pi_0(\Phi(G;R))$ and $f \in \on{CF}_b(G, \ch(k(\frakp)))$, $h_{k(\frakq)\otimes_R x} = f$ if $\frakp\sim \frakq$ (i.e., $\frakp$ and $\frakq$ are representatives of the same component in $\tilde \pi_0(\Phi(G;R))$)  and $h_{k(\frakq) \otimes_R x}\equiv f(1)$ if not. If $\ch(k(\frakp)) = 0$ or $\ch(k(\frakp)) \not\in \Phi(G;R)$, one can simply take $x \coloneqq R[f(1)]$, so assume $\ch(k(\frakp)) \in \Phi(G;R)$. Set $p \coloneqq \ch(k(\frakp))$.
        
        Let $H$ be a Sylow $p$-subgroup of $G$, then $\tilde \pi_0(\Phi(H;R))$ consists of representatives of the connected components of $\Spec(R/p)$ and $\Spec(R_p)$. We have a commutative diagram induced from restriction.
        \begin{figure}[H]
            \centering
            \begin{tikzcd}
                \Pic(\catK(G;R)) \ar[r, "h"]\ar[d, "\Res^G_H"] & \prod_{\frakp \in \tilde \pi_0(\Phi(G;R))} \on{CF}_b(G;\on{char}(k(\frakp)) \ar[d, "\Res"]\\
                \Pic(\catK(H;R)) \ar[r, twoheadrightarrow, "h"] & \on{CF}_b^{gl}(H;R)
            \end{tikzcd}
        \end{figure}
        Recall that $\on{CF}_b^{gl}(H;R)$ is the subgroup of $\prod_{\frakp \in \tilde \pi_0(\Phi(H;R))} \on{CF}_b(H;p)$ consisting of tuples of Borel-Smith functions agreeing on the trivial subgroup. The bottom arrow is then surjective by \Cref{thm:p_grp_h_mark_surj}, therefore there exists an endotrivial $x_H \in \Pic(\catK(H;R))$ with $h_{k(\frakq) \otimes_R x_H} = \Res(f)$ if $\frakp \sim \frakq$ and $h_{k(\frakq) \otimes_R x_H} \equiv f(1)$ otherwise. Moreover we may choose $\catH(x_H) = R$. 
        
        Note $\Res(f)$ is $G$-stable. Therefore, by conjugation and restriction, given any $p$-subgroup $K \leq_p H$, we can define $x_K$ analogously via conjugation and restriction. This forms a coherent family of endotrivials $(x_H)$, as every endotrivial satisfies $\catH(x_K) = R$ and their h-marks are compatible via \Cref{lem:res_via_h_mks}. For any $q$-subgroup $K' \leq G$ with $q \neq p$ on the other hand, we set $x_{K'} = R[f(1)]$. Together, this defines an element of $\lim_{G/H\in \mathcal{O}_{\Phi}(G)^\mathrm{op}}\mathrm{Pic}(\catD(H;R))$, hence by \Cref{coro-lim-picard}, an endotrivial $x \in \Pic(\catK(G;R))$. 
    \end{proof}

    We already have by \Cref{obs:ker_h_mark_decomp} that $\ker(h) = \Pic(R)\times \catM(G;R)$, where $\catM(G;R)$ denotes the subgroup of $\Hom(G, R^\times)$ consisting of elements $\varphi$ for which $R_\varphi$ is $\natural$-permutation. \Cref{prop:directsummandlemma} gives a criterion for determining $\catM(G;R)$, which was simplified considerably for $p$-groups by \Cref{thm:rk1ppermsforpgrp}. We now can state a slightly easier criterion for checking whether $R_\varphi$ is $\natural$-permutation. This is reminiscent of the fact that for a finite group $G$ and a field $k$ of prime characteristic $p$, a $kG$-module is a direct summand of a permutation module if and only if it is permutation upon restriction to a Sylow $p$-subgroup.

    \begin{corollary}\label{cor:detect_rk1_pperms_for_all_grps}
        Let $\varphi \in \Hom(G,R^\times)$. Then $R_\varphi$ is $\natural$-permutation if and only if $\Res^G_H (R_\varphi)$ is $\natural$-permutation for all Sylow subgroups $H\leq G$.
    \end{corollary}
    \begin{proof}
        The forward direction follows since restriction preserves $\natural$-permutation, and the reverse direction follows from \Cref{cor:limit_of_pic_groups}, since the collection of $\Res^G_H (R_\varphi) \in \Pic(\mathsf{perm}(H;R))^\natural$ defines an element of $\lim_{G/H\in \mathcal{O}_{\Phi}(G)^\mathrm{op}}
            \mathrm{Pic}(\catD(H;R))$.
    \end{proof}

    \begin{notation}\label{not:glued_borel_smith_all_grps}
        We extend the notation of \Cref{not:glued_borel_smith} as follows. For a finite group $G$, we let $\on{CF}_b^{gl}$ denote the subgroup of $\prod_{\frakp\in\tilde \pi_0(\Phi(G;R))} \on{CF}_b(G;\ch(k(\frakp)))$ consisting of tuples of Borel-Smith functions agreeing on the trivial subgroup of $G$. Then \Cref{thm:img_of_hmk_all_grps} states that the reduced fiberwise h-mark homomorphism induces a surjective group homomorphism $\Pic(\catK(G;R)) \twoheadrightarrow \on{CF}_b^{gl}(G;R)$.
    \end{notation}

    We conclude with the complete description of $\Pic(\catK(G;R))$. Recall that  $\catM(G;R)$ denotes the subgroup of $ \Hom(G,R^\times)$ consisting of homomorphisms $\varphi$ for which the free $R$-rank 1 $RG$-module $R_\varphi$ is $\natural$-permutation. 

    \begin{theorem}\label{thm:mainthm}
        Let $G$ be a finite group and $R$ a commutative, connected, Noetherian ring. Then we have a split exact sequence \[0 \to \Pic(R) \times \catM(G;R) \to \Pic(\catK(G;R)) \xrightarrow{h} \on{CF}_b^{gl}(G;R) \to 0,\] with the inclusion having canonical section $\catH\colon \Pic(\catK(G;R)) \to \Pic(R) \times \catM(G;R)$. 
    \end{theorem}
    \begin{proof}
        \Cref{thm:img_of_hmk_all_grps} deduces $h$ is surjective (and is necessarily split since $\on{CF}_b^{gl}(G;R) $ is free abelian), and \Cref{thm:ker_h_marks} and \Cref{obs:ker_h_mark_decomp} deduce that $\ker(h) \cong \Pic(R) \times \catM(G;R)$. The map $\catH\colon\Pic(\catK(G;R)) \to \Pic(\mathsf{perm}(G;R)^\natural)$ is a section to the inclusion $\ker(h) \cong \Pic(\mathsf{perm}(G;R)^\natural) \to  \Pic(\catK(G;R))$. 
    \end{proof}

    \begin{remark}
        In fact, we have also deduced the Picard groupoid of $\catK(G;R)$, since $\pi_{n}(\mathrm{pic}(\catK(G;R)))$ is trivial for $n > 1$ and $\pi_1(\mathrm{pic}(\catK(G;R))) = R^\times$. 
    \end{remark}

    We remark on finite generation of $\Pic(\catK(G;R))$. It is clear that $\on{CF}_b^{gl}(G;R)$ is a finitely generated free abelian group, as it is a subgroup of a direct product of finitely many copies of $\on{CF}(G)$, which is itself free abelian. Therefore, finite generation of $\Pic(\catK(G;R))$ hinges on finite generation of $\Pic(R)$ and $\catM(G;R)$. However, even with $R$ Noetherian, $\Pic(R)$ need be finitely generated. For instance, if $R$ is an elliptic curve, $\Pic(R)$ is uncountable. However, it turns out that this is the only obstruction.

    \begin{theorem}\label{prop:fingen}
        We have that $\Pic(\catK(G;R))$ is a finitely generated abelian group if and only if $\Pic(R)$ is. In particular, if $R$ is a principal ideal domain, then $\Pic(\catK(G;R))$ is finitely generated.
    \end{theorem}
    \begin{proof}
        Given the above discussion, it suffices to show that $\catM(G;R)$ is finitely generated, and for this, it suffices to show that the torsion subgroup of $R^\times$ has finitely many elements of order $r$ for all positive integers $r$. Moreover, by \Cref{thm:endotrivials_lift_from_reduction}, it suffices to assume $R$ is reduced. 
        
        Since $R$ is reduced and Noetherian, the 0 ideal is the intersection of finitely many prime ideals $\frakp_1, \dots, \frakp_n$. Therefore, we have an injective ring homomorphism $R \to R/\frakp_1 \times \cdots R/\frakp_n$. Since $R/\frakp_i$ is an integral domain for all $i$, the polynomial $x^r - 1$ has at most $r$ solutions, therefore there are at most $r^n$ units of order $r$ in $R$. 
    \end{proof}

    \subsection{Integral endotrivial complexes} For $R = \bbZ$, we have a nice description of $\Pic(\catK(G;\bbZ))$. We moreover relate these endotrivials to \emph{homotopy representations}, i.e., invertible genuine $G$-spectra. These were first studied in detail by tom Dieck and Petrie \cite{tDP82}, see e.g. \cite{Kra25} for a modern introduction on homotopy representations. We omit details for the sake of brevity. 
    
    Recall that $\on{SH}(G)$ is equivalently the homotopy category of the stable $\infty$-category $\mathsf{Sp}^G$ of genuine $G$-spectra. We have a tt-functor $\on{SH}(G) \to \catD(G;R)$ which factors through $\catD(G;\bbZ)$. Therefore, we obtain a group homomorphism $\Pic(\on{SH}(G)^c) \to \Pic(\catK(G;\bbZ))$. Moreover, the map \[\catB\colon \on{RO}(G) \to \Pic(\catK(G; \bbZ))\] factors through $\Pic(\on{SH}(G)^c)$. In particular, if $G$ is a $p$-group,  \Cref{cor:pgrp_integer_classification} asserts that every integer endotrivial arises from a homotopy representation. Additionally, the dimension homomorphism \[\dim\colon \Pic(\on{SH}(G)^c) \to \on{CF}(G)\] restricted to the family of $p$-subgroups of $G$ factors through $\Pic(\catK(G;\bbZ))$, via the fiberwise h-mark homomorphism at the prime $\frakp = p\bbZ$. See \cite{F25} for details on these connections, in particular the correspondence between modular fixed points for permutation modules and geometric fixed points for spectra. 

    \begin{corollary}\label{cor:mainthmforZ}
        Let $G$ be a finite group. Let $\on{CF}_b(G,\Phi)$ denote the group of Borel-Smith superclass functions valued on prime power subgroups of $G$. Then \[\Pic(\catK(G;\bbZ)) \cong \on{CF}_b(G,\Phi).\] In particular, $\Pic(\catK(G;\bbZ))$ is finitely generated. Moreover, if $G$ is non-nilpotent, the induced map \[\Pic(\on{SH}(G)^c) \to \Pic(\catK(G;\bbZ))\] need not be surjective, that is, not every integral endotrivial complex is induced from a homotopy representation. 
    \end{corollary}
    \begin{proof}
        Since $p\bbZ$ is a point for each prime $p$, it easily follows that $\on{CF}_b^{gl}(G;\bbZ) \cong \on{CF}_b(G;\Phi)$. Because $\bbZ$ is a principal ideal domain, $\Pic(\bbZ)$ is trivial. Finally, $\catM(G;R)$ is trivial by \Cref{cor:rk1_for_int_domain}. Finite generation follows from \Cref{prop:fingen}.

        For the final statement, let $X$ be a homotopy representation, i.e., $X \in \Pic(\on{SH}(G)^c)$. Then \cite[Proposition 1.2]{Bau89} implies that $\dim(X)$ restricted to a family of $p$-subgroups satisfies the following \emph{Bauer condition}: for every triple $H \triangleleft K \triangleleft M \leq G$ of subgroups with $K/H \cong C_p$ and $H$ a $p$-subgroup, \[\dim(H) \equiv \dim(K) \mod q^{r-l},\] where $M/K \cong C_{q^r}$ with $q$ prime acts on $K/H$ with kernel of order $q^l$. On the other hand, given an endotrivial $x \in \catK(G;\bbZ)$, $h_{\lambda_{p\bbZ}^*(x)}$ need not satisfy this condition if $G$ is non-nilpotent (if $G$ is nilpotent, the condition is trivial). 

        For an explicit example, let $G$ be the Frobenius group $G = C_3 \rtimes C_7$. Then the superclass function $f\in \on{CF}_b(G; \Phi)$ defined by \[f(1) = 2,\quad f(C_3) = 2, \quad f(C_7) = 0\] can represent the fiberwise h-mark homomorphism of an endotrivial $x \in \Pic(\catK(G;\bbZ))$, but does not satisfy Bauer's condition, since $f(1) \not\equiv f(C_7) \mod 3$, so no homotopy representation for $G$ can induce $x$. 
    \end{proof}
    
    \begin{remark}
        Moreover, $\catB\colon \on{RO}(G) \to \Pic(\catK(G;R))$ is non-surjective for non-nilpotent-groups, as the dimension homomorphism $\dim\colon \on{RO}(G) \to \on{CF}_b(G)$ is in general non-surjective. Here, $\on{CF}_b(G)$ denotes the group of superclass functions on $G$ satisfying the Borel-Smith conditions at every subquotient, i.e., the definition provided in \cite{tD87}. 

        We note that a modified version of Bauer's condition was considered by \cite{GY21}; Gelvin--Yal\c{c}in questioned whether the condition precisely characterized the kernel of the so-called \emph{Bouc homomorphism}. This was confirmed by the second-named author \cite[Theorem 6.13]{Mil25c}.
        
        Where these `new' endotrivials come from, if not homotopy representations or representation spheres, and what they `look like' remains a mystery for now... 
    \end{remark}

    \subsection*{Open questions} We conclude by listing and motivating some pertinent questions.
    \begin{enumerate}
        \item As listed above, do the integral endotrivials which are not induced from homotopy representations have an origin, or are they purely integral permutation module phenomena? 
        \item Given an endotrivial, is there an algorithm to explicitly writing down the modules in each degree and the differentials? Can one deduce whether it can be built from permutation modules, as opposed to $\natural$-permutation modules?
        \item The previous question is relevant for the following question: given an endotrivial built from permutation modules with \emph{effective}, i.e., monotonically decreasing, fiberwise h-marks, can it be upgraded to the structure of a complex of free $R\calO(G)$-modules? Such complexes, \emph{algebraic homotopy representations} are studied in \cite{HY14}. If one can deduce that an endotrivial can be upgraded in this way, and the endotrivial is orientable, then one can construct its forerunners using \cite{Mil25c}.
        \item Let $k$ be a field of characteristic $p$. In \cite{Mil25b}, we deduced that $\Pic(\catK(-;k))$ has the structure of a \emph{rational $p$-biset functor} (in the sense of Bouc, see \cite{Bou10}), with the induction operation arising from the topological norm map on representation spheres (see e.g. \cite[Proposition A.59]{HHR16}). Though it is unlikely the Picard group can be realized as a biset functor in full generality, as deflation, which corresponds to modular fixed points, likely will not be realized for all normal subgroups, does a more general induction operation exist? That is, is $\Pic(\catK(-,R))$ a (global) Mackey functor?  
    \end{enumerate}

    \bibliography{bib}
    \bibliographystyle{alpha}
    
\end{document}